\pdfoutput=1
\documentclass[reqno]{amsart}
\usepackage[utf8]{inputenc}
\usepackage[T1]{fontenc}
\usepackage[margin={3cm,2cm}]{geometry}
\usepackage[foot]{amsaddr}
\numberwithin{equation}{section}
\usepackage{ntrees}

\usepackage{hyperref, cleveref}
\usepackage{bookmark}
\usepackage[ocgcolorlinks]{ocgx2} 
\usepackage{xcolor}
\usepackage{tikz-cd}
\usepackage{mathtools, lmodern, microtype, mathabx, stmaryrd, shuffle, amssymb}
\usepackage{mleftright, xspace}
\usepackage{fixcmex}
\usepackage{mathrsfs}
\usepackage[sort,compress]{cite}
\mleftright

\usepackage{todonotes}
\hypersetup{citecolor=blue,%
pdfinfo={Author={P. Friz, P. Hager, and N. Tapia}, Title={Unified signature cumulants and generalized Magnus expansions},%
Keywords={Signatures, L\'evy processes, moment-cumulant relations, characteristic functions, diamond product},%
Subject={MSC(2020) 60L10, 60L90, 60E10, 60G44, 60G48, 60G51, 60J76}}
}

\DeclareMathOperator{\ad}{ad}
\newcommand{\tg}{{\tau_\Gamma}}
\newcommand{\tltg}{\mathbf{1}_{\{t<\tau_\Gamma\}}}
\newcommand{\etfrak}{\mathfrak{y}}
\newcommand{\xlone}{\mathbf{1}_{|x|\le1}}
\newcommand{\xgone}{\mathbf{1}_{|x|>1}}
\def\bA{\mathbf{A}}
\def\bB{\mathbf{B}}
\def\bC{\mathbf{C}}
\def\bF{\mathbf{F}}
\def\bJ{\mathbf{J}}
\def\bL{\mathbf{L}}
\def\bM{\mathbf{M}}
\def\bN{\mathbf{N}}
\def\bS{\mathbf{S}}
\def\bU{\mathbf{U}}
\def\bV{\mathbf{V}}
\def\bW{\mathbf{W}}
\def\bX{\mathbf{X}}
\def\bY{\mathbf{Y}}
\def\bZ{\mathbf{Z}}
\def\bm{\mathbf{m}}
\def\ba{\mathbf{a}}
\def\bb{\mathbf{b}}

\def\bu{\mathbf{u}}
\def\bv{\mathbf{v}}
\def\Sesup{\mathscr{S}}
\def\indik{\mathbf 1}
\def\dparl{\mathopen{(\mkern-3mu(}}
\def\dparr{\mathclose{)\mkern-3mu)}}

\newtheorem{theorem}{Theorem}[section]
\newtheorem{proposition}[theorem]{Proposition}
\newtheorem{lemma}[theorem]{Lemma}
\newtheorem{corollary}[theorem]{Corollary}
\newtheorem{remark}[theorem]{Remark}
\theoremstyle{definition}
\newtheorem{definition}[theorem]{Definition}

\newtheorem{xmp}[theorem]{Example}
\newtheorem{claim}[theorem]{Claim}
\newenvironment{example}{\renewcommand{\qedsymbol}{\(\lozenge\)}\pushQED{\qed}\begin{xmp}}{\popQED\end{xmp}}

\newcommand{\R}{\mathbb{R}}
\newcommand{\Rd}{{\mathbb{R}^d}}
\newcommand{\F}{\mathcal{F}}
\newcommand{\PM}{\mathbb{P}}
\newcommand{\E}{\mathbb{E}}
\newcommand{\bx}{\mathbf{x}}
\newcommand{\by}{\mathbf{y}}
\newcommand{\N}{{\mathbb{N}}}
\newcommand{\None}{{\mathbb{N}_{\ge1}}}

\newcommand{\Lcal}{\mathcal{L}}

\newcommand{\Id}{\mathrm{Id}}
\newcommand{\I}{\mathbf{I}}
\newcommand{\W}{\mathcal{W}}

\newcommand{\mul}{\mathrm{m}}
\newcommand{\dd}{\mathrm{d}}
\newcommand{\TT}{\mathcal{T}}
\newcommand{\Lie}[2]{\left[ #1, #2 \right]}
\newcommand{\tf}{\TT_0}
\newcommand{\ideal}{\mathcal{I}}

\newcommand{\outerbracket}[2]{\left\llbracket #1, #2 \right\rrbracket}

\newcommand{\innerbracket}[2]{\left\langle #1, #2 \right\rangle}
\newcommand{\innerbracketsmall}[2]{\langle #1, #2 \rangle}

\newcommand{\QV}[1]{\left\langle #1 \right\rangle}
\newcommand{\GQV}[1]{\left[ #1 \right]}
\newcommand{\GQVsmall}[1]{[ #1]}
\newcommand{\QVsmall}[1]{\langle #1 \rangle}
\newcommand{\CV}[2]{\left\langle #1, #2 \right\rangle}
\newcommand{\CVsmall}[2]{\langle #1, #2 \rangle}

\DeclarePairedDelimiter{\abss}{\vert}{\vert}
\def\abs{\abss*}
\DeclarePairedDelimiter{\Abss}{\Vert}{\Vert}
\def\Abs{\Abss*}

\DeclarePairedDelimiter{\Aabss}{\vert\mkern-2.5mu\vert\mkern-2.5mu\vert}{\vert\mkern-2.5mu\vert\mkern-2.5mu\vert}
\def\Aabs{\Aabss*}
\newcommand{\wt}{\widetilde}
\newcommand{\ol}{\overline}

\newcommand{\kap}{\pmb{\kappa}}
\newcommand{\kapT}[1]{\pmb{\kappa}_{#1}(T)}

\newcommand{\tkap}{\tilde\kap}
\newcommand{\esig}{\pmb{\mu}}

\newcommand{\KK}{\mathbb{K}}

\newcommand{\fmu}{\pmb{\mu}}

\newcommand{\wpr}{{w^{\prime}}}

\newcommand{\Mag}{\mathrm{Mag}}
\newcommand{\Qua}{\mathrm{Qua}}
\newcommand{\Cov}{\mathrm{Cov}}
\newcommand{\Jmp}{\mathrm{Jmp}}
\newcommand{\Sig}{\mathrm{Sig}}
\newcommand{\T}{\mathcal{T}}
\newcommand{\Sy}{\mathcal{S}} 

\newcommand{\hatexp}{\widehat \exp}
\newcommand{\hatlog}{\widehat \log}

\newcommand{\expN}[1]{\exp_N({#1})}

\newcommand{\Se}{\mathscr{S}} 
\newcommand{\D}{\mathscr{D}} %
\newcommand{\Sec}{\Se^c} 
\newcommand{\Ma}{\mathscr{M}} 
\newcommand{\Mac}{\mathscr{M}^c} 
\def\cadlag{càdlàg\xspace}
\newcommand{\loc}{\mathrm{loc}}
\newcommand{\Fv}{\mathscr{V}} 
\newcommand{\var}{\mathrm{var}}
\newcommand{\HSe}{\mathscr{H}}
\newcommand{\HSehom}{\mathscr{H}}
\newcommand{\HSehomSym}{\widehat{\mathscr{H}}}
\newcommand{\homnqN}[1]{\Aabs{#1}_{\HSehom^{q,N}}}
\newcommand{\homn}[1]{\Aabs{#1}_{\HSehom^{1,N}}}
\newcommand{\homnqNs}[1]{\Aabss{#1}_{\HSehom^{q,N}}}
\newcommand{\homns}[1]{\Aabss{#1}_{\HSehom^{1,N}}}

\begin{document}
\title{Unified signature cumulants and generalized Magnus expansions}
\author[P.~Friz]{Peter~K.~Friz$^{\dagger,\ddagger}$}
\author[P.~Hager]{Paul~Hager$^\dagger$}
\author[N.~Tapia]{Nikolas Tapia$^{\dagger,\ddagger}$}
\email{\{friz,phager,tapia\}@math.tu-berlin.de}
\address{$^\dagger$Institut für Mathematik, TU Berlin, Str. des 17. Juni 136, 10586 Berlin, Germany.}
\address{$^\ddagger$Weierstrass Institute, Mohrenstr. 39, 10117 Berlin, Germany.}
\begin{abstract}
The signature of a path can be described as its full non-commutative exponential. Following T. Lyons we regard its expectation, the {\em expected signature}, as path space analogue of the classical moment generating function. The logarithm thereof, taken in the tensor algebra, defines the {\em signature cumulant}.  We establish a universal functional relation in a general semimartingale context. Our work exhibits the importance of Magnus expansions in the algorithmic problem of computing expected signature cumulants, and further offers a far-reaching generalization of recent results on characteristic exponents dubbed diamond and cumulant expansions;
with motivation ranging from financial mathematics to statistical physics. From an affine process perspective, the functional relation may be interpreted as infinite-dimensional, non-commutative (``Hausdorff'') variation of Riccati's equation. Many examples are given.
\end{abstract}
\maketitle

\tableofcontents

\section{Introduction and main results}

Write $\T \coloneq 
T\dparl\mathbb{R}^d\dparr = \Pi_{k \ge 0} (\R^d)^{\otimes k}$ for the tensor series over $\R^d$, equipped with concatenation product, elements of which are written indifferently as
$$\bx = (\bx^{(0)},\bx^{(1)},\bx^{(2)},\dotsc) \equiv
\bx^{(0)}+ \bx^{(1)}+ \bx^{(2)}+\dotsb.
 $$
The affine subspace $\T_0$ (resp. $\T_1$) with scalar component $\bx^{(0)} = 0$ (resp. $ = 1$) has a natural Lie algebra (resp. formal Lie group) structure.

Let further $\Se = \Se (\R^d)$, resp. $\Sec= \Sec(\R^d)$, denote the class of \cadlag, resp. continuous, $d$-dimensional semimartingales on some filtered probability space $(\Omega,
(\mathcal{F}_t)_{t \ge 0}, \mathbb{P})$. The formal sum of iterated Stratonovich-integrals, the
\emph{signature} of $X \in \Sec$
\[
\Sig (X)_{s, t} = 1 + X_{s, t} + \int_s^t X_{s, u}\,{\circ\mathrm d}X_u + \int_s^t
   \left( \int_s^{u_1} X_{s, u_2}\,{\circ \mathrm d}X_{u_2} \right)\,{\circ\mathrm d}
   X_{u_1} + \cdots
\]
for $0 \le s \le t$ defines a random element in $\T_1$ and, as a process, a formal $\T_1$-valued
semimartingale. By regarding the $d$-dimensional semimartingale $X$ as $\tf$-valued semimartingale
($X \leftrightarrow \bX = (0,X,0,\dots$)), we see that the signature of $X$ satisfies the Stratonovich stochastic differential equation
\begin{equation}
     \mathrm d S = S\,{\circ\mathrm d} \bX \label{equ:gSig}.
\end{equation}
The solution is a.k.a. the Lie group valued stochastic exponential (or development) of 
$\bX \in \Se(\tf)$, with classical references \cite{mckean1969stochastic,hakim1986exponentielle};
the càdlàg case \cite{estrade1992exponentielle} is consistent with the geometric or Marcus
\cite{marcus1978modeling,marcus1981modeling, kurtz1995strtonovich, applebaum2009levy
, friz2017general} interpretation of
\eqref{equ:gSig}\footnote{Diamond notation for Marcus SDEs, $\mathrm d S = S\,{\diamond \mathrm d} \bX$, cf. \cite{applebaum2009levy}, will not be used here to avoid notational clash with \cite{alos2018exponentiation, friz2020cumulants}.}
with jump behavior $S_t = e^{\Delta \bX_t} S_{t-}$. From a stochastic differential
geometry point of view, one aims for an intrinsic understanding of \eqref{equ:gSig} valid for arbitrary Lie groups. For instance, if $\bX$ takes values in any sub Lie algebra $\mathcal{L} \subset \T_0$, then $S$ takes values in the group $\mathcal{G} = \exp \mathcal{L}$.
In case of a $d$-dimensional semimartingale $X$, the minimal choice is
$\mathrm{Lie}\dparl\R^d\dparr$, see
e.g. \cite{reutenauer2003free}, the resulting log-Lie structure of iterated integrals (both in the
smooth and Stratonovich semimartingale case) is well-known. The extrinsic linear ambient space $\T
\supset \exp{\mathcal{L}}$ will be important to us. Indeed, writing $S_t=\Sig(\bX)_{0,t}$ for the
(unique, global) $\T_1$-valued solution of \eqref{equ:gSig} driven by $\T_0$-valued $\bX$, started at $S_0 = 1$, we define, whenever $\Sig (\bX)_{0, T}$ is (componentwise) integrable, the {\em expected signature} and \emph{signature cumulants} (SigCum)
\[
       \fmu (T) \coloneq  \E (\Sig (\bX)_{0, T})\in\TT_1, \quad \kap (T) \coloneq  \log \fmu (T) \in \tf.
\]
Already when $\bX$ is deterministic, and sufficiently regular to make \eqref{equ:gSig} meaningful, this leads to an interesting (ordinary differential) equation for $\kap$ with accompanying (Magnus) expansion, well understood as effective computational tool \cite{iserles2005lie,blanes2009magnus}. The importance of the stochastic case $\bX = \bX(\omega)$, with expectation and logarithm thereof, was developed by Lyons and coworkers; see \cite{LyonsICM} and references therein, with a  variety of applications, ranging from machine learning to numerical algorithms on Wiener space known as cubature, see e.g. \cite{lyons2004cubature}.  In case of $d=1$ and $\bX=(0,X,0,\dots)$ with a single scalar semimartingale $X$, this is nothing but the sequence of moments and cumulants of the real valued random variable $X_T-X_0$. When $d > 1$, expected signature / cumulants provides an effective way
to describe the process $X$ on $[0,T]$, see \cite{LJZ11, LyonsICM, chevyrev2016characteristic}. The question arises how to compute. If one takes $\bX$ as $d$-dimensional Brownian motion, the signature cumulant $\kap(T)$ equals $(T/2) \I_d$, where $\I_d$ is the identity $2$-tensor over $\R^d$. This is known as {\em Fawcett's formula}, \cite{lyons2004cubature,friz2020course}. 
 Loosely speaking, and postponing precise definitions, our main result is a vast generalization of Fawcett's formula.
 \begin{theorem}[FunctEqu $\mathcal{S}$-SigCum]\label{thm:main_result} For sufficiently integrable $\bX\in \Se(\tf)$, the (time-$t$) conditional signature cumulants 
 $ \kap_t (T) \equiv \kap_t \coloneq  \log  \E_t (\Sig (\bX)_{t, T})$, is the unique solution of the functional equation
\begin{equation}
\begin{split}
        \kap_t (T) & = \E_t\bigg\{ \int_{(t,T]}H(\ad{\kap_{u-}})(\mathrm d\bX_{u})
+ \frac{1}{2}\int_t^{T}H(\ad{\kap_{u-}})(\mathrm d\QV{\bX^{c}}_{u}) \label{eq:mainthm} \\
& + \frac{1}{2} \int_t^{T}H(\ad{\kap_{u-}}) \circ Q(\ad{\kap_{u-}})(\mathrm d\outerbracket{\kap}{\kap}^{c}_u)
+\int_t^{T}H(\ad{\kap_{u-}})\circ(\Id\odot G(\ad{\kap_{u-}}))(\mathrm d\outerbracket{\bX}{\kap}^{c}_u)
\\
&\qquad + \sum_{t < u \le T}\bigg(H(\ad{\kap_{u-}})\Big(\exp(\Delta \bX_u)\exp(\kap_u)\exp(-\kap_{u-}) - 1 -\Delta \bX_u\Big) - \Delta \kap_u \bigg)\bigg\},
\end{split}
\end{equation}
where all integrals are understood in It\^o- and Riemann--Stieltjes sense
respectively.\footnote{Here $\circ$ denotes composition, not to be confused with Stratonovich
integration ${\circ\dd} \bX$. } The functions $H,G,Q$ are defined in \eqref{eq:GHQ_def} below, cf. also Section \ref{sec:preliminaries} for further notation.
\end{theorem}
As displayed in Figures 1 and 2, this theorem has an avalanche of consequences on which we now comment.
\begin{itemize}
    \item \Cref{eq:mainthm} allows to compute $\kap^{(n)} = \pi^{(n)}(\kap) \in (\R^d)^{\otimes n}$
        as function of $\kap^{(1)},\dotsc,\kap^{(n-1)}$.
    (This remark
    applies \textit{mutatis mutandis} to all special cases seen as vertices in Figure 1.) The resulting expansions, displayed in \Cref{fig:cube.recur}, are of computational interest.
    \item The most classical consequence of \eqref{eq:mainthm} appears when $\bX$ is a deterministic continuous semimartingale, i.e. $\bX \in FV^c (\tf)$, which also covers the absolutely continuous case with, $\dot \bX \in L^1_{\mathrm{loc}} (\tf)$.  In this case all bracket terms and the final jump-sum disappear. What remains is a classical differential equation due to \cite{hausdorff1906symbolische}, here in backward form
    \begin{equation} \label{eqHDODE}
        - \dd \kap_t (T) = H(\ad{\kap_{t}})\dd\bX_t,\qquad - \dot \kap_t (T) = H(\ad{\kap_{t}})\dot{\bX}_t,
     \end{equation} 
    the accompanying expansions is then precisely Magnus expansion \cite{magnus1954exponential, iserles1999solution, iserles2005lie, blanes2009magnus}.
    By taking $\bX$ continuous and piecewise linear 
  on two adjacent intervals, say $[0,1)\cup[1,2)$, one obtains the
    Baker--Campbell--Hausdorff formula (see e.g. 
    \cite[Theorem 5.5]{Miller})
    \begin{equation} \label{equ:BCH}
        \begin{split}
             \kap_0(2)=\log\bigl(\exp(\bx_1)\exp(\bx_2) \bigr) &\eqcolon \operatorname{BCH}(\bx_1,\bx_2)\\
              &=\bx_2+\int_0^1\Psi(\exp(\ad t\bx_1)\circ\exp(\ad\bx_2))(\bx_1)\,\mathrm dt,
        \end{split}
    \end{equation}
    with
    \[
        \Psi(z)\coloneq\frac{\ln(z)}{z-1}=\sum_{n\ge 0}\frac{(-1)^n}{n+1}(z-1)^n
    \]
        It is also instructive to let $\bX$ piecewise constant on these intervals, with $\Delta \bX_1 =\bx_1, \Delta \bX_2 = \bx_2$, in which case
    \eqref{eq:mainthm} reduces to the first equality in \eqref{equ:BCH}.
    Such jump variations of the Magnus expansion are discussed in \Cref{sec:hausmag}.
    \item Writing $\bx \mapsto \hat \bx$ for the projection from $\T$ to the symmetric algebra $\Sy$ as the
linear space identified with symmetric tensor series, equation \eqref{eq:mainthm}, in its projected and commutative form becomes
    \begin{equation} \label{eq:maincom}
    \begin{split}
    \qquad \text{FunctEqu $\Se$-Cum:} \quad
        \hat \kap_t (T) & = \E_t\bigg\{ \hat \bX_{t,T}  + \frac{1}{2} \QV{(\hat \bX+ \hat \kap)^{c}}_{t,T}\\
&\qquad + \sum_{t < u \le T}\bigg(\exp \Big( \Delta \hat \bX_u + \Delta \hat \kap_u \Big)
- 1 - (\Delta \hat \bX_u + \Delta \hat \kap_u ) \bigg) \bigg\}
\end{split}
\end{equation}
where $\hat \bX$ is a $\Sy_0$-valued semimartingale, and $\exp\colon \Sy_0 \mapsto \Sy_1$ defined by
the usual power series. This includes of course semimartingales with values in $\R^d$, canonically
embedded in $\Sy_0$.
More interestingly, the case
$\hat \bX = (0,aX,b\langle X \rangle, 0, \dotsc)$, for a $d$-dimensional continuous martingale $X$ can be seen to underlie the expansions of \cite{friz2020cumulants}, which improves and unifies previous results \cite{lacoin2019probabilistic, alos2018exponentiation}, treating $(a,b)=(1,0)$ and $(a,b)=(1,-1/2)$, with motivation from QFT and mathematical finance, respectively. Following Gatheral and coworkers, \eqref{eq:maincom} and subsequent expansions involve ``diamond'' products of semimartingales, given, whenever well-defined, by
$$
          (A \diamond B)_t(T) \coloneq \E_t  \big( \QV{A^c, B^c}_{t,T} \big).
$$
All this is discussed in Section \ref{sub:diamond}. With regard to the existing (commutative) literature, our algebraic setup is ideally suited to work under finite moment assumptions, we are able to deal with jumps, not treated in \cite{lacoin2019probabilistic, alos2018exponentiation}.  Equation \eqref{eq:maincom} has a remarkable interpretation in that it can be viewed as (with jumps: generalized) infinite-dimensional Riccati differential equation and indeed reduces to the finite-dimensional equation when specialized to (sufficiently integrable) ``affine'' continuous (resp. general) semimartingales \cite{duffie2003affine,cuchiero2011affine, keller2011affine}.
    Of recent interest, explicit diamond expansions have been obtained for
   ``rough affine'' processes, non-Markov by nature, with cumulant generating function characterized by Riccati Volterra equations, see \cite{abijaber2019affine, gatheral2019affine,friz2020cumulants}. It is remarkable that analytic tractability remains intact when one passes to path space and considers signature cumulants, Section \ref{sec:AVp}.
\end{itemize}

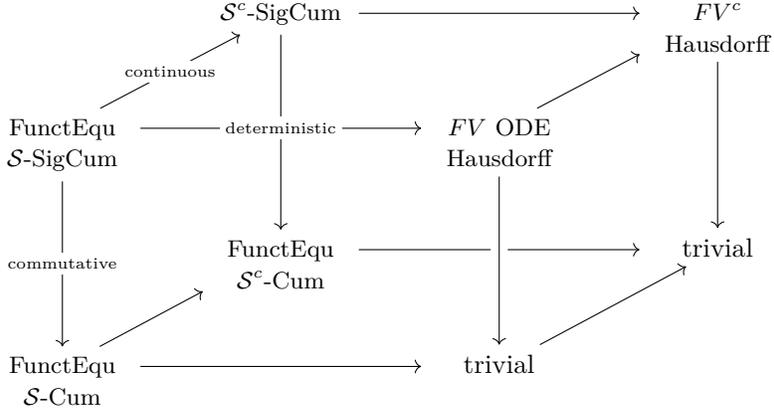
\begin{figure}[!ht]
    \centering
\begin{tikzcd}[math mode=false, cells={align=center,text width=5em}]
    &{\small $\mathcal{S}^c$-SigCum}\ar{rr}\ar{dd}&&{\small $FV^c$ Hausdorff}\ar{dd}\\
    {\small FunctEqu $\mathcal{S}$-SigCum} \ar["{\tiny continuous}"description]{ur}\ar["{\tiny commutative}"description]{dd}\ar[crossing over,"{\tiny deterministic}"description]{rr}&&{\small  $FV$ ODE Hausdorff}\ar[crossing over]{dd}\ar{ur}&\\
    &{\small FunctEqu $\mathcal{S}^c$-Cum}\ar{rr}&&trivial\\
    {\small FunctEqu $\mathcal{S}$-Cum} \ar{rr}\ar{ur}&&trivial\ar{ur}\ar[from=uu,crossing over]&
\end{tikzcd}
    \caption{FunctEqu $\mathcal{S}$-SigCum (\Cref{thm:main_with_jumps}) and implications} 
    \label{fig:theorems}
\end{figure}
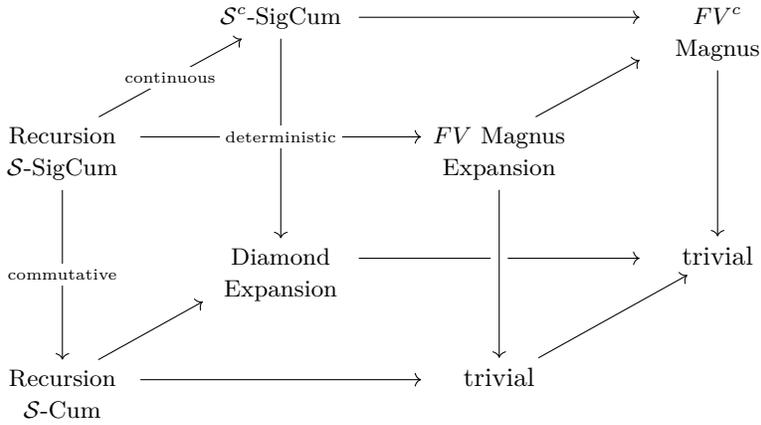
\begin{figure}[!ht]
    \centering
\begin{tikzcd}[math mode=false, cells={align=center,text width=5em}]
    &{\small $\mathcal{S}^c$-SigCum}\ar{rr}\ar{dd}&&{\small $FV^c$ Magnus}\ar{dd}\\
    {\small Recursion $\mathcal{S}$-SigCum} \ar["{\tiny continuous}"description]{ur}\ar["{\tiny commutative}"description]{dd}\ar[crossing over,"{\tiny deterministic}"description]{rr}&&{\small  $FV$ Magnus Expansion}\ar[crossing over]{dd}\ar{ur}&\\
    &{\small Diamond Expansion}\ar{rr}&&trivial\\
    {\small Recursion $\mathcal{S}$-Cum} \ar{rr}\ar{ur}&&trivial\ar{ur}\ar[from=uu,crossing over]&
\end{tikzcd}
    \caption{Computational consequence:  accompanying recursions}
    \label{fig:cube.recur}
\end{figure}

{\bf Acknowledgment.} PKF has received funding from the European Research Council (ERC) under the European Union's Horizon 2020 research and innovation program (grant agreement No. 683164) and the DFG Research Unit FOR 2402. PH and NT are supported by the DFG MATH$^+$ Excellence Cluster.

\section{Preliminaries}\label{sec:preliminaries}
\subsection{The tensor algebra and tensor series}\label{sec:tensor_series}
Denote by $T(\Rd)$ the tensor algebra over $\Rd$, i.e.
\begin{align*}
T(\Rd)\coloneq \bigoplus_{k=0}^\infty (\Rd)^{\otimes k},
\end{align*}
elements of which are {\it finite} sums (a.k.a. tensor polynomials) of the form
\begin{equation} \label{equ:bxform}
\bx = \sum_{k \ge 0} \bx^{(k)} 
= \sum_{w \in \W^d} \bx^w e_w
\end{equation}
with $\bx^{(k)} \in (\Rd)^{\otimes k}, \bx^w \in \R$ and linear basis vectors $e_w \coloneq e_{i_1}\dotsm e_{i_k}\in(\Rd)^{\otimes k}$ where $w$ ranges over all words $w=i_1\dotsm i_k\in\W_d$ over the alphabet $\{1,\dots,d\}$. Note $\bx^{(k)} =  \sum_{|w|=k} \bx^w e_w$ where $|w|$ denotes the length a word $w$. The element $e_\emptyset = 1 \in (\Rd)^{\otimes 0} \cong \R$ is neutral element of the concatenation (a.k.a. tensor) product, is obtained by linear extension of \(e_we_{w'}=e_{ww'}\) where $ww' \in \W_d$ denotes concatenation of two words.
We thus have, for $\bx,\by \in T(\Rd)$,
$$
\bx\by = \sum_{k \ge 0} \sum_{\ell =0}^k \bx^{(\ell)} \by^{(k-\ell)}  = \sum_{w \in \W^d} \left( \sum_{w_1w_2 = w} \bx^{w_1}\by^{w_2} \right) e_w \in T(\Rd).
$$
This extends naturally to {\em infinite} sums, a.k.a  tensor series, elements of the ``completed'' tensor algebra
\begin{align*}
  \TT \coloneq T\dparl\R^d\dparr\coloneq \prod_{k=0}^\infty (\Rd)^{\otimes k},
\end{align*}
which are written as in (\ref{equ:bxform}), but now as formal infinite sums with identical notation and multiplication rules; the resulting algebra $\TT$ obviously extends $T(\R^d)$.
For any $n\in\None$ define the projection to tensor levels by $$\pi_n: \TT \to (\Rd)^{\otimes n}, \quad \bx \mapsto \bx^{(n)}.$$
Denote by $\TT_0$ and $\TT_1$ the subspaces of tensor series starting with $0$ and $1$ respectively;
that is, \(\bx \in\tf\) (resp. \(\TT_1\)) if and only if \(\bx^\emptyset=0\) (resp. \(\bx^\emptyset=1\)).
Restricted to $\tf$ and $\TT_1$ respectively, the exponential and logarithm in $\TT$, defined by the usual series,
\begin{align*}
\exp\colon\tf \to \TT_1,& \quad \bx \mapsto \exp(\bx) \coloneq   1 + \sum_{k=1}^\infty \frac{1}{k!}(\bx)^k, \\
\log\colon\TT_1 \to \tf,& \quad 1 + \bx \mapsto \log( 1 + \bx) \coloneq \sum_{k=1}^\infty \frac{(-1)^{k+1}}{k}(\bx)^k,
\end{align*}
are globally defined and inverse to each other. The vector space $\tf$ becomes a {\it Lie algebra} with
$$\Lie{\bx}{\by} \coloneq \bx\by-\by\bx, \qquad \ad{\by}\colon \tf \to \tf, \ \bx \mapsto \Lie{\by}{\bx}.$$
Its exponential image $\TT_1=\exp(\tf)$ is a Lie group, at least formally so. We refrain from equipping the infinite-dimensional $\TT_1$ with a differentiable structure, not necessary in view of the ``locally finite'' nature of the group law $(\bx,\by) \mapsto \bx \by$.

Let $(a_k)_{k\ge1}$ be a sequence of real numbers then we can always define a linear operator on $\tf$ by
\begin{align*}
\left[ \sum_{k \ge 0}a_k(\ad\bx)^{k} \right]: \tf \to \tf, \quad \by\mapsto \sum_{k \ge 0}a_k(\ad\bx)^{k}(\by),
\end{align*}
where $(\ad\bx)^{0} = \Id$ is the identity operator and $(\ad \bx)^{n} = \ad\bx \circ (\ad \bx)^{n-1}$ for any $n\in\None$.
Indeed, there is no convergence issue due to the graded structure as can be seen by projecting to some tensor level $n\in \None$
\begin{align}\label{eq:projection_adjoint_powerseries}
\begin{split}
\pi_n \left( \sum_{n \ge 0}a_k(\ad\bx)^{k}(\by) \right)
=&\; \sum_{k = 0}^{n-1}a_k \pi_n\big( (\ad\bx)^{k}(\by) \big) \\
=&\; a_0 \by^{(n)} + \sum_{k = 1}^{n-1}a_{k} \sum_{\Abs{\ell} = n, \abs{\ell} = k + 1} (\ad\bx^{(l_2)}\circ \cdots \circ\ad\bx^{(l_{k+1})})(\by^{(l_1)}),
\end{split}
\end{align}
where the inner summation in the right-hand side is over a finite set of multi-indices $\ell = (l_1,
\dotsc, l_{k+1})\in (\None)^{k+1}$ where $\abs{\ell} \coloneq k+1$ and $\Abs{\ell} \coloneq l_1 + \dots + l_{k+1}$.
In the following we will simply write $(\ad\bx\ad\by) \equiv (\ad\bx \circ \ad\by)$ for the composition of adjoint operators. Further, when $\ell = (l_1)$ is a multi-index of length one, we will use the notation $(\ad\bx^{(l_2)} \cdots \ad\bx^{(l_{k+1})}) \equiv \Id$.
Note also that the iteration of adjoint operations can be explicitly expanded in terms of left- and right-multiplication as follows
\begin{align}\label{eq:integral_adjoint_power}
\ad\by^{(l_2)}\cdots\ad\by^{(l_{k+1})}(\bx^{(l_1)}_u)
&=\sum_{I\dot\cup J = \{1, \dots, k\}} (-1)^{\abs{J}}  \left(\prod_{i\in I} \by^{(l_{i+1})}
\right)\bx^{(l_1)}_u  \left(\prod_{j\in J} \by^{(l_{j+1})}\right).
\end{align}

For a word $w \in \W_d$ with $|w|>0$ we define the directional derivative for a function $f\colon \TT \to \R$ by
\begin{align*}
(\partial_w f)(a)\coloneq\partial_t (f(a + t e_w))\big\vert_{t=0},
\end{align*}
for any $a \in \TT$ such that the right-hand derivative exists.

Write $\mul\colon\TT \otimes \TT \to \TT$ for
multiplication (concatenation) map, i.e. \(\mul (a \otimes b) = ab\), in general different from $ba$, extended by linearity. For linear maps $g, f\colon\TT\to\TT$ we define $g
\odot f = \mul\circ(g \otimes f)$, i.e.
\begin{align*}
(g \odot f) (a \otimes b) = g(a)f(b), \quad a,b\in\TT,
\end{align*}
extended by linearity.

\subsection{Some quotients of the tensor algebra}\label{sec:trunc_tensor}
The {\it symmetric algebra} over \(\Rd\), denoted by \(S(\Rd)\) is the quotient of \(T(\Rd)\) by the
two-sided ideal \(I\) generated by \(\{xy-yx:x,y\in\Rd\}\). The canonical projection
\(T(\Rd)\twoheadrightarrow S(\Rd), \bx \mapsto \hat \bx \), is an algebra epimorphism.
A linear basis of $S(\Rd)$ is then given by $\{ \hat e_w \}$ over non-decreasing words,
$w=(i_1,\dotsc,i_n) \in \widehat \W_d$, with $1 \le i_1 \le \dots \le i_n \le d, n \ge 0$. Every $\tilde \bx \in S(\R^d)$ can be written as finite sum,
$$
     \tilde \bx = \sum_{w \in \widehat \W_d} \tilde \bx^w \hat e_w ,
$$
and we have an immediate identification with polynomials in $d$ commuting indeterminates. The canonical projection map extends to
an epimorphism \(\TT\twoheadrightarrow\Sy\) where $\TT = T \dparl \R^d\dparr$ and $\Sy = S \dparl
\R^d\dparr$ are the respective completions, identifiable as formal series in $d$ non-commuting (resp. commuting) indeterminates. As a vector space, $\Sy$ 
can be identified with {\it symmetric} formal tensor series. Denote by $ \Sy_0$ and $\Sy_1$ the affine space determined by \(\tilde \bx^\emptyset=0\) (resp. \( \tilde \bx^\emptyset=1\)).
The usual power series in $\Sy$ define $\hatexp{}\colon \Sy_0 \to \Sy_1$ with inverse
$\hatlog{}\colon \Sy_1 \to \Sy_0$ and we have
\begin{align*}
\widehat{\exp{(\bx + \by)}} &= \hatexp{}(\hat \bx)\hatexp{}(\hat \by), \quad \bx, \by \in \tf\\
\widehat{\log{(\bx \by)}} &= \hatlog{}(\hat \bx) + \hatlog{}(\hat \by), \quad \bx,\by \in \TT_1.
\end{align*}
We shall abuse notation in what follows and write $\exp$ (resp. $\log$), instead of $\hatexp$ (resp. $ \hatlog$).

\subsubsection{The (step-$n$) truncated tensor algebra}\label{sec:truncated_tensor_algebra}
For \(n\in\N\), the subspace \[ \ideal_n\coloneq\prod_{k=n+1}^\infty(\Rd)^{\otimes k} \] is a two sided
ideal of \(\TT\).
Therefore, the quotient space \(\TT/\ideal_n\) has a natural algebra structure. We denote the \emph{projection map} by \(\pi_{(0,n)}\).
We can identify \(\TT/\ideal_n\) with %
\[
    \TT^n \coloneq \bigoplus_{k=0}^n(\Rd)^{\otimes k},
\]
equipped with truncated tensor product,
$$
\bx\by = \sum_{k=0}^{n} \sum_{\ell_1 + \ell_2=k} \bx^{(\ell_1)} \by^{(\ell_2)}  = \sum_{w \in \W^d,|w|\le n} \left( \sum_{w_1w_2 = w} \bx^{w_1}\by^{w_2} \right) e_w \in \TT^n.
$$
The sequence of algebras \((\TT^n:n\ge 0)\) forms an inverse system with limit \(\TT\).
There are also canonical inclusions \(\TT^k\hookrightarrow\TT^{n}\) for \(k\le n\); in fact, this
forms a direct system with limit \(T(\Rd)\). The usual power series in $\TT^{n}$ define $\exp_n\colon
\TT^n_0 \to \TT^n_1$ with inverse $\log_n\colon \TT^n_1 \to \TT^n_0$, we may again abuse notation and write $\exp$ and $\log$ when no confusion arises.
As before, $\TT^n_0$ has a natural Lie algebra structure, and $\TT^n_1$ (now finite dimensional) is
a \emph{bona fide} Lie group.

We equip $T(\R^d)$ with the norm
\begin{align*}
    |a|_{T(\R^d)} \coloneq  \max_{k\in\N}|a^{(n)}|_{(\Rd)^{\otimes k}},
\end{align*}
where $|\cdot|_{(\Rd)^{\otimes k}}$ is the euclidean norm on $(\Rd)^{\otimes k}\cong\R^{d^k}$, which makes it a Banach space.
The same norm makes sense in \(\TT^n\), and since the definition is consistent in the sense that $|a|_{\TT_k} = |a|_{\TT_n}$ for any $a \in \TT^{n}$ and $k \ge n$ and $|a|_{\TT_n} = |a|_{(\Rd)^{\otimes n}}$ for any $a \in (\Rd)^{\otimes n}$. We will drop the index whenever it is possible and write simply $|a|$.

\subsection{Semimartingales}\label{sec:tensor_semimartingales}
Let $\D$ be the space of adapted càdlàg process $X\colon\Omega \times [0,T) \to \mathbb{R}$ with $T\in(0,\infty]$ defined on some
filtered probability space $(\Omega, (\F_t)_{0 \le t \le T}, \mathbb{P})$.
The space of {\em semimartingales} $\Se$ is given by the processes $X\in\D$ that can be decomposed as
$$
     X_{t}=X_0+M_{t}+A_{t},
$$
where $M \in \Ma_\loc$ is a càdlàg local martingale, and $A\in\Fv$ is a càdlàg adapted process of locally bounded variation, both started at zero.
Recall that every $X \in \Se$ has a well-defined continuous local martingale part denoted by $X^c\in \Mac_\loc$.
The quadratic variation process of $X$ is then given by
\begin{align*}
[X]_t = \QV{X^{c}}_t + \sum_{0 < u \le t} (\Delta X_u)^{2}, \quad 0 \le t \le T,
\end{align*}
where $\QV{\cdot}$ denotes the (predictable) quadratic variation of a continuous semimartingale. Covariation square resp. angle brackets $[X,Y]$ and $\CV{X^{c}}{Y^{c}}$, for another real-valued semimartingale $Y$, are defined by polarization.
For $q \in [1, \infty)$, write $\Lcal^{q} = L^{q}(\Omega, \F, \PM)$, then a Banach space $\HSe^q \subset \Se$ is given by those $X\in\Se$ with $X_0 = 0$ and
$$
 \| X \|_{\HSe^q} \coloneq \inf_{X = M + A} \bigg\Vert \GQV{M}^{1 / 2}_{T} + \int_0^{T} |\dd A_s | \bigg\Vert_{\Lcal^{q}}< \infty.
$$
Note that for local martingale $M \in \Ma_\loc$ it holds (see \cite[Ch. V, p. 245]{protter2005stochastic})$$\Abs{M}_{\HSe^{q}} = \Abs{\GQV{M}^{1/2}_T}_{\Lcal^{q}}.$$
For a process $X\in\D$ we define
\begin{align*}
\Vert{X}\Vert_{\Sesup^q} \coloneq \Big\Vert{\sup_{0 \le t \le T}\abs{X_t}}\Big\Vert_{\Lcal^{q}}
\end{align*}
and define the space $\Se^q\subset\Se$ of semimartingales $X\in\Se$ such that $\Vert{X}\Vert_{\Sesup^q}< \infty$.
Note that there exits a constant $c_q>0$ depending on $q$ such that (see \cite[Ch. V, Theorem 2]{protter2005stochastic})
\begin{align}\label{eq:Sesup_vs_HSe}
\Vert{X}\Vert_{\Sesup^q} \le c_q  \| X \|_{\HSe^q}.
\end{align}

We view $d$-dimensional semimartingales, $X= \sum_{i=1}^d X^i e_i \in \Se (\R^d)$, as special cases of tensor series valued semimartingales $\Se (\T)$ of the form
$$\bX = \sum_{w \in \W_d} \bX^w e_w$$
with each component $\bX^w$ a real-valued semimartingale. (This extends mutatis mutandis to the spaces $\D, \Ma, \Fv$.
Note also that we typically deal with $\TT_0$-valued semimartingales which amounts to have only words with length $|w| \ge 1$.)
Standard notions such as continuous local martingale $\bX^{c}$ and jump process $\Delta \bX_t = \bX_t - \bX_{t^-}$ are defined componentwise.

{\bf Brackets}: Now let $\bX$ and $\bY$ be $\TT$-valued semimartingales.
We define the (non-commutative) \emph{outer quadratic covariation bracket} of $\bX$ and $\bY$ by
\[
\outerbracket{\bX}{\bY}_t
\coloneq \sum_{w_1, w_2\in\W_d}[\bX^{w_1}, \bY^{w_2}]_te_{w_1} \otimes e_{w_2}\in \TT \otimes \TT.
\]
Similarly, define the (non-commutative) {\em inner quadratic covariation bracket} by 
\begin{align*}
    [\bX,\bY]_t \coloneq \mul(\outerbracket{\bX}{\bY}) = \sum_{w\in\W_d
    }\left(\sum_{w_1 w_2 = w}
   [\bX^{w_1},\bY^{w_2}]_t \right)e_w \in \TT;
\end{align*}
for {\em continuous} $\TT$-valued semimartingales $\bX, \bY$, this coincides with the predictable quadratic covariation
\[
\CV{\bX^{}}{\bY^{}}_t \coloneq \sum_{w\in\W_d
}\left(\sum_{w_1 w_2 = w}
   \innerbracket{\bX^{w_1}}{\bY^{w_2}}_t \right)e_w \in \TT.
\]
As usual, we may write ${\left\llbracket \bX \right\rrbracket} \equiv \outerbracket{\bX}{\bX}$ and $\QV{\bX} \equiv   \CV{\bX}{\bX}$.

{\bf $\HSe$-spaces}: The definition of $\HSe^{q}$-norm naturally extends to tensor valued martingales.
More precisely, for $\bX^{(n)} \in \Se((\Rd)^{\otimes n})$ with $n\in\None$ and $q\in[1,\infty)$ we define
\begin{equation*}
\Abss{\bX^{(n)}}_{\HSe^{q}} \coloneq \Abss{\bX^{(n)}}_{\HSe^{q}((\Rd)^{\otimes n})} \coloneq
\inf_{\bX^{(n)}= \bM + \bA}\Abs{\abs{\GQV{\bM}}_T^{1/2} + \abs{\bA}_{1-\mathrm{var};[0;T]}}_{\Lcal^{q}},
\end{equation*}
where the infimum is taken over all possible decompositions $\bX^{(n)} = \bM + \bA$ with $\bM\in\Ma_\loc((\Rd)^{\otimes n})$ and $\bA\in\Fv((\Rd)^{\otimes n})$, where
\begin{equation*}
\abs{\bA}_{1-\var;[0;T]}\coloneq\sup_{0 \le t_1 \le \dotsb \le t_k \le T}
\sum_{t_i}\abs{\bA_{t_{i+1}} - \bA_{t_{i}}} \le \sum_{w\in\W_d, |w|=n} \int_0^{T}\abs{\dd A^{w}_s},
\end{equation*}
with the supremum taken over all partitions of the interval $[0,T]$.
One may readily check that
\begin{equation*}
\Abss{\bX^{(n)}}_{\HSe^{q}} \le \sum_{w\in\W_d, |w|=n} \Abs{X^{w}}_{\HSe^{q}};\quad \text{ and for }\bX^{(n)}\in\Ma_\loc:\; \Abss{\bX^{(n)}}_{\HSe^{q}} = \Abss{\abss{\GQVsmall{\bX^{(n)}}}_T}_{\Lcal^{q}}.
\end{equation*}
Further define the following subspace $\HSehom^{q,N} \subset \Se(\tf^{N})$ of homogeneously integrable semimartingales
\begin{equation*}
\HSehom^{q,N} \coloneq\left\{ \bX \in \Se(\tf^{N}) \;\Big\vert\; \bX_0 = 0,\; \homnqN{\bX} < \infty \right\},
\end{equation*}
where for any $\bX \in \Se(\TT^{N})$ we define
\begin{equation*}
\homnqN{\bX} \coloneq \sum_{n=1}^{N} \big(\Abss{\bX^{(n)}}_{\HSe^{qN/n}}\big)^{1/n}.
\end{equation*}
Note that $\homnqNs{\cdot}$ is sub-additive and positive definite on $\HSehom^{q, N}$ and it is homogeneous under dilation in the sense that
\begin{equation*}
\homnqN{\delta_{\lambda}\bX} = \abs{\lambda}\,\homnqN{\bX}, \quad \delta_\lambda \bX \coloneq
(\bX^{(0)}, \lambda \bX^{(1)}, \dotsc, \lambda^{N}\bX^{(N)}), \quad \lambda \in \R.
\end{equation*}
We also introduce the following subspace of $\Se(\TT)$
\begin{equation*}
\HSe^{\infty-}(\TT) \coloneq \left\{ \bX \in \Se(\TT):\; \bX^w\in\HSe^q, \;\forall\, 1 \le q < \infty,\;w\in\W_d\right\}.
\end{equation*}
Note that if $\bX\in\Se(\TT)$ such that $\homns{\bX^{(0,N)}} < \infty$ for all $N\in\None$ then it also holds $\bX\in\HSe^{\infty-}(\TT)$.

{\bf Stochastic integrals}: We are now going to introduce a notation for the stochastic integration with respect to tensor valued semimartingales.
Let $\bF: \Omega \times [0,T] \to \mathcal{L}(\TT; \TT)$ with $(t, \omega) \mapsto \bF_t(\omega; \cdot)$ such that it holds
\begin{align}\label{cond:tensor_integrator_one}
&(\bF_t(\bx))_{0 \le t \le T} \in \D(\TT), \quad \text{for all } \bx\in\TT\\\label{cond:tensor_integrator_two}
\text{and}\quad &\bF_t(\omega; \ideal_n) \subset \ideal_n,  \quad \text{for all } n\in\N, \; (\omega, t)\in\Omega\times[0,T],
\end{align}
where $\ideal_n\subset\TT$ was introduced in \Cref{sec:truncated_tensor_algebra}, consisting of series with tensors of level $n$ and higher.
In this case, we can define the stochastic Itô-integral (and then analogously the Stratonovich/Marcus integral) of $\bF$ with respect to $\bX\in\Se(\TT)$ by
\begin{align}\label{eq:definition_tensor_sotch_integral}
\int_{(0, \cdot]} \mathbf F_{t-}(\dd\bX_t) :=\sum_{w\in\W_d}\;\sum_{v\in\W_d,\,\abs{v} \le \abs{w}}\; \int_{(0, \cdot]}  \mathbf  F_{t-}(e_v)^{w}\dd\bX_t^{v} e_w \in \Se(\TT).
\end{align}
For example, let $\bY, \bZ \in \D(\TT)$ and define $\bF := \bY\,\Id\,\bZ$, i.e. $\bF_t(\bx) = \bY_t \, \bx \, \bZ_t$ for all $\bx\in\TT$.
Then we see that $\bF$ indeed satisfies the conditions \eqref{cond:tensor_integrator_one} and \eqref{cond:tensor_integrator_two} and we have
\begin{equation}\label{eq:def_tensor_ito_integral}
\int_{(0, \cdot]} (\bY_{t-}\,\Id\,\bZ_{t-})(\dd \bX_t)= \int_{(0, \cdot]} \bY_{t-}\dd \bX_t \bZ_{t-}= \sum_{w\in\W_d}\left( \sum_{w_1 w_2w_3 = w} \int_{(0, \cdot]}  \mathbf  Z^{w_1}_{t-} \bY^{w_3}_{t-}\,\mathrm d \bX_t^{w_2} \right)e_w.
\end{equation}
Another important example is given by $\bF = (\ad \bY)^{k}$ for any $\bY\in\D(\tf)$ and $k\in\N$.
Indeed, we immediately see $\bF$ satisfies the condition \eqref{cond:tensor_integrator_two} and recalling from \eqref{eq:integral_adjoint_power} that the iteration of adjoint operations can be expanded in terms of left- and right-multiplication, we also see that $\bF$ satisfies \eqref{cond:tensor_integrator_one}.
More generally, let $(a_k)_{k=0}^{\infty}\subset\R$ and let $\bX\in\Se(\tf)$, then the following integral
\begin{align}\label{eq:power_series_integral}
\int_{(0, \cdot]}\left[ \sum_{k=0}^{\infty} a_k (\ad\bY_{t-})^{k} \right](\dd\bX_t) = \sum_{n= 1}^{\infty}\sum_{k=0}^{n-1}\; \sum_{\Abs{\ell}=n,\, \abs{\ell} = k+1} \int_{(0,\cdot]}
\ad\bY^{(l_2)}_{t-}\cdots\ad\bY^{(l_{k+1})}_{t-}(\dd\bX^{(l_1)}_t)
\end{align}
is well define in the sense \eqref{eq:def_tensor_ito_integral}.
The definition of the integral with integrands of the form $\bF: \Omega \times [0,T] \to \mathcal{L}(\TT \otimes \TT; \TT)$ with respect to processes $\bX \in \Se(\TT \otimes \TT)$ is completely analogous.

{\bf Quotient algebras:} All of this extends in a straight forward way to the case of semimartingales in the quotient algebra of Section \ref{sec:trunc_tensor}, i.e. symmetric and truncated algebra.
In particular, given 
\(\bX\) and \(\bY\) in \(\Se(\Sy)\) have well-defined continuous local martingale parts denoted by \(\bX^c,\bY^c\) respectively, with {\em inner} (predictable) quadratic covariation given by
\[
  \langle\bX^c,\bY^c\rangle
  =\sum_{w_1,w_2\in \widehat \W_d}\langle \bX^{w_1,c},\bY^{w_2,c}\rangle\hat e_{w_1}\hat e_{w_2} .
\]
Write $\Sy^N$ for the truncated symmetric algebra, linearly spanned by $\{ \hat e_{w}: w \in \widehat \W_d, |w| \le N\}$
and $\Sy^N_0$ for those
elements with zero scalar entry. In complete analogy with non-commutative setting discussed above, we then write $\HSehomSym^{q,N} \subset \Se(\Sy^N_0)$ for the corresponding space homogeneously $q$-integrable semimartingales.

\subsection{Diamond Products}
We extend the notion of the {\em diamond product} introduced in \cite{alos2018exponentiation} for continuous scalar semimartingales to our setting.
\begin{definition} \label{def:diamondSym}
 For \(\bX\) and \(\bY\) in \(\Se(\TT)\) define
$$
(\bX \diamond \bY)_t(T) \coloneq \E_t  \big( \QV{\bX^c, \bY^c}_{t,T}
\big)=\sum_{w\in\W_d}\left( \sum_{w_1w_2=w}(\bX^{w_1}\diamond\bY^{w_2})_t(T) \right)e_w \in \TT
$$
whenever the $\TT$-valued quadratic covariation which appears  on the right-hand side is integrable.
Similar to the previous section, we also define an \emph{outer diamond}, for \(\bX,\bY\in\TT\), by
\[
  (\bX\blackdiamond\bY)_t(T)\coloneq\E_t(\outerbracket{\bX^c}{\bY^c}_{t,T})=\sum_{w_1,w_2\in\W_d}(\bX^{w_1}\diamond\bY^{w_2})_t(T)e_{w_1}\otimes
  e_{w_2}\in\TT\otimes\TT.
\]
\end{definition}
This definition extends immediately to semimartingales with values in the quotient algebras of Section \ref{sec:trunc_tensor}.
In particular, given \(\tilde\bX\) and \(\tilde\bY\) in \(\Se(\Sy)\), we have
$$
         (\tilde \bX \diamond \tilde \bY)_t(T) \coloneq \E_t \big( \langle \tilde \bX^c, \tilde \bY^c \rangle_{t,T} \big) = \sum_{w_1,w_1\in\mathcal
  W_d}  (\tilde\bX^{w_1} \diamond \tilde\bY^{w_2})_t(T)
  \hat e_{w_1}\hat e_{w_2}  \in \Sy,
$$
where the last expression is given in terms of diamond products of scalar semimartingales.

\begin{lemma} \label{lem:dm}
  Let $p,q,r\in[1,\infty)$ such that $1/p + 1/q + 1/r < 1$ and let $X\in\Ma_{\loc}^{c}((\Rd)^{\otimes l})$, \break $Y\in\Ma_{\loc}^{c}((\Rd)^{\otimes m})$, and $Z\in\D((\Rd)^{\otimes n})$ with $l,m,n\in\N$, such that $\Abs{X}_{\HSe^{p}}, \Abs{Y}_{\HSe^{q}}, \Abs{Z}_{\Se^{r}} <\infty$ then it holds for all $0 \le t \le T$
\begin{align*}
\E_t\left(\int_t^{T}Z_{u-}\dd(X \diamond Y)_u(T)\right) = -\E_t\left(\int_{t}^T Z_{u-}\dd\CV{X}{Y}_u\right).
\end{align*}
\end{lemma}
\begin{proof}
Using the Kunita-Watanabe inequality (\Cref{lem:cauchy_schwartz_tensor_bracket}) we see that the expectation on the right hand side is well defined.
Further note that it follows from Emery's inequality (\Cref{lem:emerys}) and Doob's maximal inequality that the local martingale
\begin{align*}
\int_0^{\cdot}Z_{u-}\dd(\E_u\CV{X}{Y}_T)
\end{align*}
is a true martingale.
Recall the definition of the diamond product and observe that the difference of left- and right-hand side  of the above equation is a conditional expectation of a martingale interment and is hence zero.
\end{proof}

\subsection{Generalized signatures}
\label{sec:gensig}
We now give the precise meaning of \eqref{equ:gSig}, that is \(\mathrm dS=S\,{\circ\mathrm
d}\bX\), or component-wise, for every word \(w\in\W_d\),
\[
    \mathrm dS^w=\sum_{w_1w_2=w}S^{w_1}\,{\circ\mathrm d}\bX^{w_2},
\]
where the driving noise \(\bX\) is a \(\tf\)-valued semimartingale, so that $\bX^{\emptyset} \equiv 0$.
 Following \cite{marcus1978modeling,marcus1981modeling,kurtz1995strtonovich,friz2017general,bruned2020quasi} the integral meaning of this equation, started at time $s$ from $\xi \in \T_1$, for times $t \ge s$, is given by
\begin{equation}\label{eq:marcus_signature_ode3}
S_t = \xi + \int_{(s,t]} S_{u-}\,\mathrm d \bX_u + \frac{1}{2}\int_s^{t} S_{u-}\,\mathrm d\QV{ \bX^{c}}_u\\
+ \sum_{s< u \le t} S_{u-}\big(\exp(\Delta \bX_u)-1-\Delta \bX_u\big),
\end{equation}
leaving the component-wise version to the reader. We have
\begin{proposition}
    Let \(\mathbf\xi\in\TT_1\) and suppose \(\bX\) takes values in \(\tf\). For every $s \ge 0$ and $\xi \in \T_1$, equation \eqref{eq:marcus_signature_ode3} has a
    unique global solution on \(\TT_1\) starting from $S_s=\mathbf\xi$.
    \label{prop:sigsol}
\end{proposition}
\begin{proof}
   Note that $S$ solves \eqref{eq:marcus_signature_ode3} iff $\xi^{-1} S$
   solves the same equation started from $1 \in \T_1$. We may thus take $\xi = 1$ without loss of
   generality. The graded structure of our problem, and more precisely that $\bX = (0,X,\mathbb{X},\dots)$
    in \eqref{eq:marcus_signature_ode3} has no scalar component, shows that the (necessarily)
    unique solution is given explicitly by iterated integration, as may be seen explicitly when writing out
    $S^{(0)} \equiv 1$, $S^{(1)}_t = \int_s^t \mathrm d X = X_{s,t} \in \R^d$,
      $$
      S^{(2)}_t = \int_{(s,t]} S^{(1)}_{u-}\,\mathrm d X_u +\mathbb{X}_{t} -\mathbb{X}_{s} + \frac{1}{2} \QV{ X^{c}}_{s,t} +
           \frac{1}{2} \sum_{s< u \le t} (\Delta X_u)^2 \in (\R^d)^{\otimes 2}.
   $$
   and so on. (In particular, we do not need to rely on abstract existence, uniqueness results for Marcus SDEs \cite{kurtz1995strtonovich} or Lie group stochastic exponentials \cite{hakim1986exponentielle}.)
   \end{proof}

\begin{definition}
    Let \(\bX\) be a \(\tf\)-valued semimartingale defined on some interval $[s,t]$. 
    Then
    $$ \Sig (\bX \vert_{[s,t]}) \equiv \Sig(\bX)_{s,t}$$ is defined to be the unique solution to \eqref{eq:marcus_signature_ode3} on \([s,t]\), such
    that \(\Sig(\bX)_{s,s}=1\).
\end{definition}

The following can be seen as a (generalized) Chen relation.

\begin{lemma} \label{lem:Chen}
Let $\bX$ be a  \(\tf\)-valued semimartingales on $[0,T]$ and $0 \le s \le t \le u \le T$. Then the following identity holds with probability one, for all such $s,t,u$,
        \begin{equation}\label{eq:chen_identity}
         \Sig(\bX)_{s,t}\Sig(\bX)_{t,u}=\Sig(\bX)_{s,u}.
         \end{equation}
\end{lemma}
\begin{proof} Call $\Phi_{t \leftarrow s} \xi \coloneq S_t$ the solution to \eqref{eq:marcus_signature_ode3} at time $t \ge s$, started from $S_s = \xi$. By uniqueness of the solution flow, we have
$
             \Phi_{u \leftarrow t} \circ \Phi_{t \leftarrow s} = \Phi_{u \leftarrow s} .
$
It now suffices to remark that, thanks to the multiplicative structure of \eqref{eq:marcus_signature_ode3} we have
$               \Phi_{t \leftarrow s} \xi = \xi  \Sig(\bX)_{s,t}$.
\end{proof}

\section{Expected signatures and signature cumulants}

\subsection{Definitions and existence}
Throughout this section let $\bX \in \Se(\tf)$  be defined on a filtered probability space $(\Omega, \F, (\F_t)_{0 \le t \le T}, \PM)$.
When $\E(|\Sig(\bX)^w_{0,t}|)<\infty$ for all $0 \le t \le T$ and all words $w\in\W_d$, then the \emph{(conditional) expected signature}
\begin{equation*}
\esig_t(T) \coloneq \E_t\left(\Sig(\bX)_{t,T}\right) = \sum_{w\in\W_d}\E_t(\Sig(\bX)^w_{t,T})e_w \in \TT_1, \quad 0 \le t \le T,
\end{equation*}
is well defined with $\E_t$ denoting the conditional expectation with respect to the sigma algebra $\F_t$.
In this case, we can also define the \emph{(conditional) signature cumulant} of $\bX$ by
\begin{align*}
\kapT{t}\coloneq\log\left(\E_t\left(\esig_t(T)\right)\right) \in \tf, \quad 0 \le t \le T.
\end{align*}
An important observation is the following
\begin{lemma}
Given $\E(|\Sig(\bX)^w_{0,t}|)<\infty$ for all $0 \le t \le T$ and words $w\in\W_d$, then $\esig(T) \in \Se(\TT_1)$ and $\kap(T)\in\Se(\tf)$.
\end{lemma}
\begin{proof}
It follows from the relation \eqref{eq:chen_identity} that
\begin{align*}
\esig_t(T) = \E_t\left(\Sig(\bX)_{t,T}\right) = \E_t\left(\Sig(\bX)_{0,t}^{-1}\Sig(\bX)_{0,T}\right) = \Sig(\bX)_{0,t}^{-1}\E_t\left(\Sig(\bX)_{0,T}\right).
\end{align*}
Therefore projecting to the tensor components we have
\begin{align*}
\esig_t(T)^w = \sum_{w_1w_2 = w}(-1)^{|w_1|}S(\bX)^{w_1}_{0,t}\E_t\left(S(\bX)^{w_2}_{0,T}\right), \quad 0 \le t \le T, \quad w \in \W_d.
\end{align*}
Since $(\Sig(\bX)^w_{0, t})_{0 \le t \le T}$ and $(\E_t(\Sig(\bX)^w_{0,T})_{0 \le t \le T}$ are semimartingales (the latter in fact a martingale), it follows from Itô's product rule that $\esig^w(T)$ is also a semimartingale for all words $w\in\W_d$, hence $\esig(T)\in\Se(\TT_1)$.
Further recall that $\kap(T) = \log(\esig(T))$ and therefore it follows from the definition of the logarithm on $\TT_1$ that each component $\kap(T)^w$ with $w\in\W_d$ is a polynomial of $(\esig(T)^{v})_{v\in\W_d, |v|\le|w|}$. Hence it follows again by Itô's product rule that $\kap(T)\in\Se(\tf)$.
\end{proof}
It is of strong interest to have a more explicit necessary condition for the existence of the expected signature.
The following theorem below, the proof of which can be found in \Cref{sec:proof_bdg_signature}, yields such a criterion.
\begin{theorem}\label{thm:bdg_signature}
Let $q\in[1, \infty)$ and $N\in\None$, then there exist two constants $c,C>0$ depending only on $d$, $N$ and $q$, such that for all $\bX \in \HSehom^{q,N}$
\begin{equation*}
  c\homnqNs{\bX} \le \homnqNs{\Sig(\bX)_{0,\cdot}} \le C\homnqNs{\bX}.
\end{equation*}
In particular, if $\bX\in\HSe^{\infty-}(\tf)$ then $\Sig(\bX)_{0,\cdot}\in \HSe^{\infty-}(\TT_1)$ and the expected signature exists.
\end{theorem}
\begin{remark}
Let $\bX = (0, M, 0, \dotsc, 0)$ where $M \in \Ma(\Rd)$ is a martingale, then $$\homnqNs{\bX} = \Abss{M}_{\HSe^{qN}} = \Abss{\abs{\GQV{M}_T}^{1/2}}_{\Lcal^{qN}},$$
and we see that the above estimate implies that
\begin{align*}
\max_{n=1, \dotsc, N} \Abss{\Sig(\bX)^{(n)}_{0, \cdot}}_{\Se^{qN/n}}^{1/n} \le C \Abss{M}_{\HSe^{qN}}.
\end{align*}
This estimate is already known and follows from the Burkholder-Davis-Gundy inequality for enhanced martingales, which was first proved in the continuous case in \cite{friz2006burkholder} and for the general case in \cite{chevyrev2019canonical}.
\end{remark}
\begin{remark}
When $q>1$, the above estimate also holds true when the signature $\Sig(\bX)_{0,\cdot}$ is replaced by the conditional expected signature $\esig(T)$ or the conditional signature cumulant $\kap(T)$. This will be seen in the proof of \Cref{thm:main_with_jumps} below (more precisely in \Cref{claim:proof_L_estimate}).
\end{remark}

\subsection{Moments and cumulants} \label{sec:mc}
We quickly discuss the development of a symmetric algebra valued semimartingale, more precisely \(\tilde\bX  \in\Se(\Sy_0)\), in the group $\Sy_1$.
That is, we consider
\begin{equation}
    \mathrm d\tilde S=\tilde S\,{\circ\mathrm d}\tilde\bX.
    \label{eq:gSigSy}
\end{equation}

It is immediate (validity of chain rule) that the unique solution to this equation, at time $t \ge s$, started at $\tilde S_s = \tilde \xi \in \Sy_1$ is given by
\[
    \tilde S_{t}\coloneq \exp\left( \tilde \bX_t-\tilde\bX_s \right)\tilde\xi \in \Sy_1
\]
and we also write $\tilde S_{s,t} = \exp\left( \tilde \bX_t-\tilde\bX_s \right)$ for this solution
started at time $s$ from \(1\in\Sy_1\). The relation to signatures is as follows. Recall that the hat denotes the canonical projection from $\TT$ to $\Sy$.
\begin{proposition} \label{prop:XandXhat}
    (i) Let \(\bX,\mathbf Y\in\Se(\TT)\) and $\bZ = \int \bX \mathrm d\mathbf Y$ in It\^o sense. Then \(\hat\bX,\hat{\mathbf Y}\in\Se(\Sy)\) and, in the sense of indistinguishable processes,
    \begin{equation}
       \widehat \bZ = \int \hat\bX\,\mathrm d\hat{\mathbf Y}.
        \label{eq:intproj}
    \end{equation}
    (ii) Let \(\bX\in\Se(\TT_0)\). Then \(\widehat{\Sig(\bX)_{s,\cdot}}\) solves \eqref{eq:gSigSy} started at time $s$ from \(1\in\Sy_1\) and driven
    by \(\hat\bX\in\Se(\Sy_0)\).
    In particular \(\widehat{\Sig(\bX)}_{s,t}=\exp(\hat\bX_t-\hat\bX_s)\).
\end{proposition}
\begin{proof}
    (i) That the projections \(\hat\bX,\hat{\mathbf Y}\) define \(\Sy\)-valued semimartingales follows
    from the componentwise definition and the fact that the canonical projection is linear.
    In particular, the right-hand side of \cref{eq:intproj} is well defined.
    Now, \cref{eq:intproj} is true whenever \(\bX\) is piece-wise constant.
    By a limiting procedure, we immediately see that it is also true for general
    semimartingales.
    Part (ii) is then immediate.
\end{proof}

Assuming componentwise integrability, we then define {\it symmetric moments} and {\it cumulants} by
\begin{align*}
\tilde\fmu_t(T) & \coloneq \E_t\exp\left( \tilde\bX_T-\tilde\bX_t\right) = \sum_w \E_t \left( \exp\left( \tilde\bX_T-\tilde\bX_t\right)^w_{t,T}\right )\hat e_w
  \in \Sy_1,\\
\tilde\kap_t(T) & \coloneq \log \tilde\fmu_t(T) \in \Sy_0,  \quad 0 \le t \le T.
\end{align*}
If $\tilde \bX = \hat \bX$, for $ \bX \in \Se(\TT)$, with expected signature and signature cumulants $\fmu$ and $\kap$, it is then clear that the symmetric moments and cumulants of $\hat \bX$ are obtained by projection, $$\fmu \mapsto \hat \fmu, \quad \kap \mapsto \hat \kap.$$

\begin{example}    Let \( X\) be an \(\R^d\)-valued martingale in \(\mathscr H^{\infty-}\), and $\tilde\bX_t\coloneq\sum_{i=1}^dX^i_t\hat e_i$. Then
    \[
        \tilde \fmu_t(T)=\sum_{n=0}^\infty\frac{1}{n!}\E_t(X_T-X_t)^n,
            \]
   consists of the (time-\(t\) conditional) multivariate moments of $X_T-X_t \in \R^d$. And it readily follows, also noted in \cite[Example 3.3]{bonnier2019signature},  that
   $\tilde \kap_t (T) = \log \fmu_t(T)$ consists precisely of the multivariate cumulants of
   $X_T-X_t$. Note that the symmetric moments and cumulants of the scaled process $a X$, $a \in \R$,
   is precisely given by $\delta_a \fmu$ and $\delta_a \kap$ where the linear dilation map is
   defined by $\delta_a\colon \hat e_w \mapsto a^{|w|} \hat e_w$. The situation is similar for $a
   \cdot X$, $a \in \R^d$, but now with $\delta_a\colon\hat e_w \mapsto a^w \hat e_1^{|w|}$ with $a^w = a_1^{n_1} \cdots a_d^{n_d}$ where $n_i$ denotes the multiplicity of the letter $i \in \{1,\dots, d\}$ in the word $w$.
\end{example}
We next consider linear combinations, $\tilde\bX = a X + b \langle X \rangle $, for general pairs $a,b \in \R$, having already dealt with $b=0$. The special case $b = - a^2/2$, by scaling there is no loss in generality to take $(a,b) = (1,-1/2)$, yields a (at least formally) familiar exponential martingale identity.
\begin{example}
    Let \( X\) be an \(\R^d\)-valued martingale in \(\mathscr H^{\infty-}\), and define
    \[
        \tilde\bX_t\coloneq\sum_{i=1}^dX^i_t\hat e_i-\frac12\sum_{1\le i\le j\le d}\langle
        X^i,X^j\rangle_t\hat e_{ij}.
    \]
    In this case we have trivial symmetric cumulants, \(\tilde\kap_t(T)=0\) for all \(0\le t\le T\).
    Indeed, It\^o's formula shows that \(t\mapsto\exp\left(\tilde\bX_t\right)\) is an
    \(\Sy_1\)-valued martingale, so that
    \[
    \tilde\fmu_t(T)=\E_t\exp(\tilde\bX_T-\tilde\bX_t)=\exp(-\tilde\bX_t)\E_t\exp(\tilde\bX_T)=1.\qedhere \]
    \end{example}
While the symmetric cumulants of the last example carries no information, it suffices to work with
$$
        \tilde\bX = \sum_{i=1}^d a^i X  + \sum_{j,k=1}^d b_{jk} \langle X^j,X^k \rangle
$$
in which case $\fmu = \fmu (a,b), \kap = \kap(a,b)$ contains full information of the joint moments of $X$ and its quadratic variation process. A recursion of these was constructed as diamond expansion in \cite{friz2020cumulants}.

\section{Main Results}
\subsection{Functional equation for signature cumulants}\label{sec:master_equation}
Let $\bX\in\Se(\tf)$ defined one a filtered probability space $(\Omega, \F, (\F_t)_{0\le t\le
T<\infty},\PM)$ satisfying the usual conditions.
For all $\bx\in\tf$ (or $\tf^N$) define the following operators, with Bernoulli numbers $(B_k)_{k\ge 0} = (1, -\frac{1}{2}, \frac{1}{6}\dotsc)$,
\begin{equation}\label{eq:GHQ_def}
    \begin{split}
    G(\ad{\bx}) = \sum_{k = 0}^{\infty} \frac{(\ad{\bx})^k}{(k + 1) !}, \quad
    &Q(\ad{\bx}) = \sum_{m, n = 0}^{\infty}2\frac{(\ad{\bx})^n\odot (\ad{\bx})^m}{(n + 1) ! (m) ! (n + m + 2)},\\
    H(\ad{\bx}) &\coloneq \sum_{k=0}^{\infty}\frac{B_k}{k!}(\ad{\bx})^{k},
    \end{split}
\end{equation}
noting $G(z) = (\exp(z)-1)/z$, $H(z) = G^{-1}(z) = z/(\exp(z)-1)$.
Our main result is the following
\begin{theorem}\label{thm:main_with_jumps}
Let $\bX \in\HSe^{\infty-}(\tf)$, then the signature cumulant $\kap =\kap (T) = (\log\E_t(\Sig(\bX)_{t,T}))_{0 \le t \le T}$ is the unique solution (up to indistinguishably) of the following functional equation: for all $0 \le t \le T$
\begin{align}\label{eq:master_with_jumps}
\begin{split}
0 = \E_t\bigg\{& \bX_{t,T} + \frac{1}{2}\QV{\bX^{c}}_{t,T} +
    \int_{(t,T]}G(\ad{\kap_{u-}})(\dd\kap_u) + \frac{1}{2}\int_t^{T}Q(\ad{\kap_{u-}})(\dd\outerbracket{\kap^{c}}{\kap^{c}}_u) \\
&+\int_t^{T}(\Id\odot G(\ad{\kap_{u-}}))(\dd\outerbracket{\bX^{c}}{\kap^{c}}_u) \\
&+ \sum_{t < u \le T}\Big(\exp(\Delta \bX_u)\exp(\kap_u)\exp(-\kap_{u-}) - 1 -\Delta \bX_u - G(\ad{\kap_{u-}})(\Delta \kap_u) \Big)\bigg\}.
\end{split}
  \end{align}
Equivalently, $\kap =\kap (T) $ is the unique solution to
\begin{equation}\label{eq:master_with_jumps_h_form}
\begin{split}
\kap_t = \E_t\bigg\{& \int_{(t,T]}H(\ad{\kap_{u-}})(\mathrm d\bX_{u})
+ \frac{1}{2}\int_t^{T}H(\ad{\kap_{u-}})(\mathrm d\QV{\bX^{c}}_{u})\\
& + \frac{1}{2} \int_t^{T}H(\ad{\kap_{u-}}) \circ Q(\ad{\kap_{u-}})(\mathrm d\outerbracket{\kap^{c}}{\kap^{c}}_u)\\
&+\int_t^{T}H(\ad{\kap_{u-}})\circ(\Id\odot G(\ad{\kap_{u-}}))(\mathrm d\outerbracket{\bX^{c}}{\kap^{c}}_u)\\
&+ \sum_{t < u \le T}\bigg(H(\ad{\kap_{u-}})\Big(\exp(\Delta \bX_u)\exp(\kap_u)\exp(-\kap_{u-}) - 1 -\Delta \bX_u\Big) - \Delta \kap_u \bigg)\bigg\}.
\end{split}
\end{equation}

Furthermore, if $\bX\in\HSehom^{1,N}$ for some $N\in\None$, then the identities
\eqref{eq:master_with_jumps} and \eqref{eq:master_with_jumps_h_form} still hold true for the
truncated signature cumulant $\kap \coloneq (\log\E_t(\Sig(\bX^{(0,N)})_{t,T}))_{0\le t \le T}$.
\end{theorem}

\begin{proof}
We postpone the proof for the fact that $\kap$ satisfies the equations \eqref{eq:master_with_jumps} and \eqref{eq:master_with_jumps_h_form} to section \Cref{sec:proof_main_theorem}.
The uniqueness part of the statement can be easily seen as follows:
Regarding equation \eqref{eq:master_with_jumps} we first note that it holds
\begin{align*}
\E_t \left\{ \int_{(t,T]} G(\ad \kap_{u-})(\dd \kap_u)\right\} = \E_t \left\{ \int_{(t,T]} (G(\ad \kap_{u-})-\Id)(\dd \kap_u) \right\} - \kap_t, \quad 0 \le t \le T,
\end{align*}
where we have used that $\kap_T \equiv 0$ (and the fact that the conditional expectation is well defined, which is shown in the first part of the proof).
Hence, after separating the identity from $G$, we can bring $\kap_t$ to the left-hand side in \eqref{eq:master_with_jumps}.
This identity is an equality of tensor series in $\tf$ and can be projected to yield an equality for each tensor level of the series.
As presented in more detail in the following subsection, we see that projecting the latter equation to tensor level say $n\in\None$, the right-hand side only depends on $\kap^{(k)}$ for $k < n$, hence giving an explicit representation $\kap^{(n)}$ in terms of $\bX$ and strictly lower tensor levels of $\kap$.
Therefore the equation \eqref{eq:master_with_jumps} characterizes $\kap$ up to a modification and then due to right-continuity up to indistinguishably.
The same argument applies to the equation \eqref{eq:master_with_jumps_h_form}, referring to the following subsections for details on the recursion.
\end{proof}

{\bf Diamond formulation:} The functional equations given in \Cref{thm:main_with_jumps} above, can be phrased in terms of the diamond product between $\tf$-valued semimartingales.
Writing $\mathbf{J}_t(T) = \sum_{t < u \le T} (\dots)$ for the last (jump) sum in
\eqref{eq:master_with_jumps}, this equation can be written, thanks to Lemma \ref{lem:dm}, which
applies just the same with outer diamonds, 
\begin{align*}
\frac{1}{2}  (\bX \diamond \bX)_t(T) & + \E_t \bigg\{ \bX_{t,T} +
\int_{(t,T]}G(\ad{\kap_{u-}})(\dd\kap_u) + 
\mathbf{J}_t(T)  \bigg\} \\
=
& \E_t \bigg\{  \frac{1}{2}\int_t^{T}Q(\ad{\kap_{u-}}) \dd  (\kap \blackdiamond \kap)_u(T)       
+\int_t^{T}(\Id\odot G(\ad{\kap_{u-}})) \dd   (\bX \blackdiamond \kap)_u(T)        
   \bigg\}
  \end{align*}
  and a similar form may be given for (\ref{eq:master_with_jumps_h_form}). While one may, or may not, prefer this equation to \eqref{eq:master_with_jumps}, diamonds become very natural in $d=1$ (or upon projection
  to the symmetric algebra, cf. \Cref{sub:diamond}). In this case $G = \Id$, $Q = \Id \odot \Id$ and with identities of the form
  $$
          \int_t^{T}(\Id\odot \Id) \dd   (\bX \blackdiamond \bY)_u(T)   = (\bX \diamond \bY)_u(T)|_{u=t}^T = - (\bX \diamond \bY)_t(T)
  $$
  some simple rearrangement, using bilinearity of the diamond product, gives 
\begin{equation} \label{equ:Diamond_dis1}
     \kap_t (T) = \E_t \{ \bX_{t,T} \} + \frac{1}{2} ((\bX + \kap) \diamond  (\bX + \kap))_t(T)  + \E_t \{ \mathbf{J}_t(T) \}.
\end{equation}
If we further impose martingality and continuity, we arrive at
$$
        \kap_t (T) =  \frac{1}{2} ((\bX + \kap) \diamond  (\bX + \kap))_t(T).
$$

\subsection{Recursive formulas for signature cumulants}\label{sec:reucsrion_with_jumps}

Theorem \ref{thm:main_with_jumps} allows for an iterative computation of signature cumulants, trivially started from
\begin{align*}
    \kap^{(1)}_t & = \fmu^{(1)}_t = \E_t\left(\bX^{(1)}_{t, T}\right).
\end{align*}
The second signature cumulant, obtained from Theorem  \ref{thm:main_with_jumps}, or from first principles, reads
\begin{align*}
\kap^{(2)}_t & = \E_t \bigg\{\bX^{(2)}_{t,T}
    + \frac{1}{2}\QV{\bX^{(1)c}}_{t,T} +\frac12 \int_{(t,T]}\Lie{\kap^{(1)}_{u-}}{\dd\kap^{(1)}_u} + \frac{1}{2}\QV{\kap^{(1)c}}_{t,T} + \CV{\bX^{(1)c}}{\kap^{(1)c}}_{t,T} \\
    &\hspace{3em}+ \sum_{t<u\le T}\bigg(\frac{1}{2}\left(\Delta \bX^{(1)}_u\right)^{2} + \Delta \bX^{(1)}_u\Delta\kap^{(1)}_u + \frac{1}{2}\left(\Delta\kap^{(1)}_u\right)^{2}  \bigg)\bigg\}
    \end{align*}
For instance, consider the special case with vanishing higher order components, $\bX^{(i)} \equiv 0$, for $i \ne 1$,
and $\bX = \bX^{(1)} \equiv M$, a $d$-dimensional continuous square-integrable
martingale. 
In this case, $\kap^{(1)}  = \fmu^{(1)} \equiv 0$
and from the very definition of the logarithm relating $\kap$ and $\fmu$, we have $\kap^{(2)} = \fmu^{(2)} - \frac12 \fmu^{(1)} \fmu^{(1)} = \fmu^{(2)}$.
It then follows from Stratonovich-Ito correction that
$$
         \kap^{(2)}_t =  \E_t \int_t^T (M_u - M_s) \circ \dd M_u = \frac12 \E_t \QV{M}_{t,T} =  \frac12 \E_t \QV{\bX^{(1)}}_{t,T}
$$
which is indeed a (very) special case of the general expression for $\kap^{(2)}$. We now treat general higher order signature cumulants.

\begin{corollary}\label{cor:reucsrion_with_jumps} Let $\bX\in\HSehom^{1,N}$ for some $N\in \mathbb{N}_{\ge 1}$, then we have
\begin{align*}
    \kap^{(1)}_t & = \E_t\left(\bX^{(1)}_{t, T}\right),
\end{align*}
for all $0 \le t \le T$ and for $n \in \{2, \dotsc, N\}$ we have recursively (the r.h.s. only depends on $\kap^{(j)},j<n$)
\begin{multline}\label{eq:recursion_jump}
\kap^{(n)}_t =  \E_t\left(\bX^{(n)}_{t,T }\right) + \frac12\sum_{k = 1}^{n-1}\E_t\left( \CV{\bX^{(k)c}}{\bX^{(n-k)c}}_{t,T}\right) \\
 +\sum_{|\ell|\ge 2, \; \|\ell\|=n}\E_t\Big(\Mag(\kap; \ell)_{t,T} + \Qua(\kap; \ell)_{t,T} + \Cov(\bX, \kap; \ell)_{t,T}  + \Jmp(\bX,\kap; \ell)_{t,T}\Big)
  \end{multline}
with $\ell = (l_1, \dotsc, l_k)$, $l_i \in \None$, $|\ell|\coloneq k\in\None$, $\|\ell\|\coloneq l_1 + \dotsb + l_k$ and
\begin{align*}
\Mag(\kap; l_1, \dotsc, l_{k})_{t,T} &= \frac{1}{k!} \int_{(t,T]} \ad{\kap^{(l_2)}_{u-}} \cdots \ad{\kap^{(l_{k})}_{u-}} (\dd \kap_u^{(l_{1})}) \\
\Qua(\kap; l_1, \dotsc, l_{k})_{t,T} &=\begin{multlined}[t] \frac{1}{k!}\sum_{m =
  2}^{k}\binom{k-1}{m-1}\\\times\int_t^{T}\Big(\ad{\kap^{(l_3)}_{u-}}\cdots
      \ad{\kap^{(l_{m})}_{u-}} \odot \ad{\kap^{(l_{m + 1})}_{u-}} \cdots
        \ad{\kap^{(l_k)}_{u-}} \Big)\Big(\dd \outerbracket{\kap^{(l_{1})c}}{
              \kap^{(l_{2})c}}_u\Big)\end{multlined}\\
\Cov(\bX, \kap; l_1, \dotsc, l_k)_{t,T} &= \frac{1}{(k-1)!} \int_t^T \left( \mathrm{Id} \odot \ad{\kap^{(l_3)}_{u-}} \cdots \ad{\kap^{(l_k)}_{u-}} \right) \left(\dd \outerbracket{\bX^{(l_1)c}}{\kap^{(l_{2})c}}_u\right) \\
\Jmp(\bX,\kap; l_1, \dotsc, l_k)_{t,T} &=\begin{multlined}[t] \sum_{t < u \le T}\sum_{1 \le m \le j \le
  k}\left((-1)^{k-j}\frac{\Delta\bX^{(l_1)}_u \cdots
    \Delta\bX^{(l_m)}_u\kap^{(l_{m+1})}_u\cdots\kap^{(l_j)}_u\kap^{(l_{j+1})}_{u-}\cdots
\kap^{(l_{k})}_{u-}}{m!(m-j)!(k-j)!}\right) \\
-\frac{1}{k!}\ad{\kap^{(l_2)}_{u-}}\cdots\ad{\kap^{(l_{k})}_{u-}}\left(\Delta
\kap^{(l_1)}_u\right).\end{multlined}
\end{align*}
\end{corollary}
\begin{proof}
Recall from \Cref{sec:tensor_semimartingales}, more specifically \eqref{eq:power_series_integral},
the definition of the stochastic Itô integral of a power series of adjoint operations with respect to a tensor valued semimartingale.
As in the proof of \Cref{thm:main_with_jumps} above, in \eqref{eq:master_with_jumps}, we can separate the identity from $G$ and bring the resulting $\kap_t$ to the left-hand side.
The recursion then follows from projecting the resulting form of the equation to tensors of level $n\in\{1, \dots, N\}$.
We demonstrate this projection for the first appearing term, which is the stochastic integral with respect to $\kap$.
It holds
\begin{align*}
\pi_n \E_t \left\{ \int_{(t,T]} (G(\ad \kap_{u-})-\Id)(\dd \kap_u) \right\}
&= \E_t \left\{\sum_{k=1}^{n} \frac{1}{k!}\sum_{\Abs{\ell} = n, \abs{\ell} = k}\int_{(t,T]} \ad{\kap^{(l_2)}_{u-}} \cdots \ad{\kap^{(l_{k})}_{u-}} (\dd \kap^{(l_1)}_u) \right\} \\
&=  \E_t \left\{\sum_{\Abs{\ell}=n} \mathrm{Mag}(\kap; \ell)_{t,T} \right\},
\end{align*}
for all $0 \le t \le T$, where in the first equality we have used the linearity to interchange $\pi_n$ with the expectation and the explicit form of the projection of a power series of adjoint operations given in \eqref{eq:projection_adjoint_powerseries}.
The projection of the remaining terms in equation \eqref{eq:master_with_jumps} follows analogously except for the jump part.
Regarding the latter, we note again that due to the linearity we can interchange the projection $\pi_n$ with the expectation and the sum over the interval $(t, T]$.
The remaining steps in order to arrive at the above form of the $\mathrm{Jmp}(\bX, \kap)$ term are a simple combinatorial exercise.
\end{proof}

We obtain another recursion for the signature cumulants from projecting the functional equation \eqref{eq:master_with_jumps_h_form}.
Note that, apart from the first two levels, it is far from trivial to see that the following recursion is equivalent to the recursion in \Cref{cor:reucsrion_with_jumps}.

\begin{corollary}\label{cor:recursion_h_form} Let $\bX\in\HSehom^{1,N}$ for some $N\in \mathbb{N}_{\ge 1}$, then we have
\begin{multline}\label{eq:recursion_h_form}
\kap^{(n)}_t = \E_t\left(\bX^{(n)}_{t,T}\right) +
\sum_{ |\ell|\ge 2,\; ||\ell||=n} \E_t\bigg( \mathrm{HMag}^{1}(\bX, \kap; \ell)_{t,T} + \frac{1}{2}\mathrm{HMag}^{2}(\bX, \kap; \ell)_{t,T} +  \mathrm{HQua}(\kap; \ell)_{t,T}\\
+ \mathrm{HCov}(\bX, \kap; \ell)_{t,T}+ \mathrm{HJmp}(\bX, \kap; \ell)_{t,T}\bigg)
\end{multline}
with $\ell = (l_1, \dotsc, l_k)$, $l_i \ge 1$, $|\ell|=k$, $||\ell||=l_1 + \dotsb + l_k$ and
\begin{align*}
\mathrm{HMag}^1(\bX, \kap; l_1, \dotsc, l_{k})_{t,T} &= \frac{B_{k-1}}{(k-1)!} \int_{(t,T]} \ad{\kap^{(l_2)}_{u-}} \cdots \ad{\kap^{(l_{k})}_{u-}} \left(\dd\bX^{(l_1)}_u\right) \\
\mathrm{HMag}^2(\bX, \kap; l_1, \dotsc, l_{k})_{t,T} &= \frac{B_{k-2}}{(k-2)!} \int_t^{T} \ad{\kap^{(l_3)}_{u-}} \cdots \ad{\kap^{(l_{k})}_{u-}} \left(\dd\CV{\bX^{(l_1)c}}{\bX^{(l_2)c}}_u\right) \\
\mathrm{HQua}(\kap; l_1, \dotsc, l_{k})_{t,T} &= \int_t^{T} \sum_{j=2}^{k}\frac{B_{k-j}}{(k-j)!}\ad{\kap^{(l_{j+1})}_{u-}}\cdots \ad{\kap^{(l_k)}_{u-}}\left(\dd\Qua(\kap; l_{1}, \dotsc, l_{j})_{u}\right)\\
\mathrm{HCov}(\bX, \kap; l_1, \dotsc, l_{k+1})_{t,T} &= \int_{t}^{T}\sum_{j=1}^{k}\frac{B_{k-j}}{{(k-j)}!}\ad{\kap^{(l_{j+1})}_{u-}}\cdots \ad{\kap^{(l_k)}_{u-}}\big(\dd\Cov(\bX, \kap; l_{1}, \dotsc, l_{j})_u\big) \\
\mathrm{HJmp}(\bX,\kap; l_1, \dotsc, l_k)_{t,T} &=\begin{multlined}[t]\sum_{t < u \le T}\sum_{1\le m
  \le j \le i \le k}(-1)^{k-j}\Bigg(\frac{B_{k-i}}{(k-i)!}\\
    \hspace{-3em}\times\ad{\kap^{(l_{i+1})}_{u-}}\cdots\ad{\kap^{(l_{k})}_{u-}}\Bigg(
\frac{\Delta\bX^{(l_1)}_u \cdots
\Delta\bX^{(l_m)}_u\kap^{(l_{m+1})}_u\cdots\kap^{(l_j)}_u\kap^{(l_{j+1})}_{u-}\cdots
\kap^{(l_{i})}_{u-}}{m!(m-j)!(k-j)!}\Bigg)\Bigg).\end{multlined}
\end{align*}
\end{corollary}
\begin{proof}
The recursion follows from projecting the equation \eqref{eq:master_with_jumps_h_form} to each tensor level, analogously to the way that the recursion of \Cref{cor:reucsrion_with_jumps} follows from \eqref{eq:master_with_jumps} (see the proof of \Cref{cor:reucsrion_with_jumps}).
\end{proof}

{\bf Diamonds. } All recursions here can be rewritten in terms of diamonds.
In a first step, by definition the second term in \Cref{cor:reucsrion_with_jumps} can be rewritten
as
\[
  \frac12\sum_{k=1}^n(\bX^{(k)}\diamond\bX^{(n-k)})_t(T).
\]
Thanks to \Cref{lem:dm} we may also write
\begin{equation*}
  \begin{split}
    \MoveEqLeft\E_t\Qua(\kap;\ell)_{t,T}\\
    &=-\E_t\Biggl\{\frac1{k!}\sum_{m=2}^k\binom{k-1}{m-1}\int_t^T\left(
      \ad\kap_{u-}^{(\ell_3)}\dotsb\ad\kap_{u-}^{(\ell_m)}\odot\ad\kap_{u-}^{(\ell_{m+1})}\dotsm\ad\kap_{u-}^{(\ell_k)}
  \right)\left( \dd(\kap^{(\ell_1)}\blackdiamond\kap^{(\ell_2)})_u(T) \right)\Biggr\}.
\end{split}
\end{equation*}
Similarly,
\[
  \E_t\Cov(\bX,\kap;\ell)_{t,T}=-\E_t\left\{
  \frac1{(k-1)!}\int_t^T\left( \Id\odot\ad\kap_{u-}^{(\ell_3)}\dotsm\ad\kap_{u-}^{(\ell_k)}
\right)\left( \dd(\bX^{(\ell_1)}\blackdiamond\kap^{(\ell_2)})_u(T) \right) \right\}.
\]
Inserting these expressions into \Cref{eq:recursion_h_form} we may obtain a ``diamond'' form of the
recursions in H form.

When \(d=1\) (or in the projection onto the symmetric algebra, c.f. \Cref{sub:diamond}) the
recursions take a particularly simple form, since \(\ad\bx\equiv 0\) for all \(\bx\in\tf\), for $d=1$ a commutative algebra.
\Cref{eq:recursion_jump} then becomes
$$
  \kap^{(n)}_t (T) =
  \E_t\left(\bX_{t,T}^{(n)}\right)+\frac12\sum_{k=1}^{n-1} ( (\bX^{(k)} + \kap^{(k)}  ) \diamond (\bX^{(n-k)} + \kap^{(n-k)}))_t(T)
  + \E_t\left( \mathbf{J}^{(n)}_t (T) \right)
  $$
where $\mathbf{J}^{(n)}_t(T) = \sum_{|\ell|\ge 2, \; \|\ell\|=n}  \Jmp(X,\kap; \ell)_{t,T}$ contains
the $n$-th tensor component of the jump contribution. The above diamond recursion can also be obtained
by projecting the functional relation (\ref{equ:Diamond_dis1}) to the $n$-th tensor level. We shall revisit
this in a multivariate setting and comment on related works in Section \ref{sub:diamond}.

\section{Two special cases}
and application of the Lie bracket, coming from the \(\ad\) operator.
\subsection{Variations on Hausdorff, Magnus and Baker--Campbell--Hausdorff}
\label{sec:hausmag}
We now consider a deterministic driver $\bX$ of finite variation. This includes the case when $\bX$ is absolutely continuous, in which case we recover, up to a harmless time reversal, $t \leftrightarrow T-t$ , Hausdorff's ODE and the classical Magnus expansion for the solution to a linear ODE in a Lie group \cite{hausdorff1906symbolische, magnus1954exponential,chen1954iterated,iserles1999solution}. Our extension with regard to discontinuities seems to be new and somewhat unifies Hausdorff's equation with multivariate Baker--Campbell--Hausdorff integral formulas.

\begin{theorem} \label{thm:magnus_expansion}
Let $\bX \in \mathscr{V} (\tf)$, and more specifically $\bX\colon [0,T] \to \tf$ deterministic, \cadlag{}
of bounded variation. The log-signature $\Omega_t=\Omega_{t}(T)\coloneq  \log(\Sig(\bX)_{t,T})$ satisfies the integral equation
\begin{align}
  \Omega_{t}(T) &=
    \int_t^{T}H(\ad{\Omega_{u-}})(\dd\bX^c_u)
  + \sum_{t<u\le T} \int_0^1\Psi(\exp(\ad \theta \Delta \bX_u)\circ\exp(\ad\Omega_u))(\Delta
  \bX_u)\,\mathrm d\theta,   \label{equ:HauJum}
  \end{align}
with \(\Psi(z)\coloneq H(\log z)={\log z}/{(z-1)}\) as in the introduction. The sum in (\ref{equ:HauJum}) is absolutely convergent, over (at most countably many) jump times of $\bX$, vanishes when $\bX \equiv \bX^c$, in which case \cref{equ:BCH} reduces to Hausdorff's ODE.

(i) The accompanying {\em Jump Magnus expansion} becomes $\Omega^{(1)}_{t}(T) = \bX^{(1)}_{t,T}$ followed by
$$
          \Omega^{(n)}_{t}(T) = \bX^{(n)}_{t,T} + \sum_{|\ell|\ge 2,
          \Vert\ell\Vert=n}\left(\mathrm{HMag}^{1}(\bX, \Omega; \ell)_{t,T} + \mathrm{HJmp}(\bX,
          \Omega; \ell)_{t,T}\right)
$$
where the right-hand side only depends on $\Omega^{(k)}, k<n$.

(ii) If $\bX \in \mathscr{V} (V)$ for some linear subspace $V \subset \T_0 = T_0\dparl\R^d\dparr$, it follows that, for all
\(t\in[0,T]\),
$$\Omega_{t}(T) \in \mathcal{L} \coloneq \mathrm{Lie}\dparl V\dparr \subset \tf, \qquad  \Sig(\bX)_{t,T} \in \exp (\mathcal{L}) \subset  \TT_1,
$$
we say that $\Omega_{t}(T)$ is Lie in $V$. In case $V=\R^d$ one speaks of (free) Lie series, cf.  \cite[Def. 6.2]{LyonsICM}.
\end{theorem}
\begin{proof} Since we are in a purely deterministic setting the signature cumulant coincides with
  the log-signature $\kap_t(T) = \Omega_{t}(T)$ and Theorem \ref{thm:main_with_jumps} applies without any expectation and angle brackets.

Using $\Delta \Omega_u = \Omega_{u} - \Omega_{u-} = \Omega_u - \log (\mathrm e^{\Delta \bX_u}\mathrm
e^{\Omega_u})$ we see that
\begin{align*}
  \Omega_{t}(T) &=
  \int_t^{T}H(\ad{\Omega_{u-}})(\dd\bX^c_u)
  - \sum_{t<u\le T}
        \Delta\Omega_u
         \\
  &= \int_t^{T}H(\ad{\Omega_{u-}})(\dd\bX^c_u)
  - \sum_{t<u\le T}
  \left(
        \Omega_u - \operatorname{BCH}(\Delta \bX_u,\Omega_u)
  \right) \\
  & = \int_t^{T}H(\ad{\Omega_{u-}})(\dd\bX^c_u)
  + \sum_{t<u\le T} \int_0^1\Psi(\exp(\theta \ad\Delta \bX_u)\circ\exp(\ad\Omega_u))(\Delta\bX_u)\,\mathrm
  d\theta,
  \end{align*}
  where we used the identity
 \begin{equation}
 \operatorname{BCH}(\bx_1,\bx_2) - \bx_2 =  \log\bigl(\exp(\bx_1)\exp(\bx_2) \bigr) - \bx_2 =
              \int_0^1\Psi(\exp(\theta\ad\bx_1)\circ\exp(\ad\bx_2))(\bx_1)\,\mathrm d\theta.\qedhere  \label{equ:BCHSec5}
            \end{equation}
\end{proof}

\begin{remark}[Baker--Campbell--Hausdorff] The identity \eqref{equ:BCHSec5} is well-known, but also easy to obtain {\em en passant}, thereby rendering the above proof self-contained.  We treat directly the $n$-fold case.
Given \(\bx_1,\dotsc,\bx_n\in\tf\) one defines a continuous {\em piecewise affine linear} path $(\bX_t: 0 \le t \le n)$ with $\bX_i - \bX_{i-1} = \bx_i$.
Then $\Sig ( \bX |_{[i-1,i]})=\Sig ( \bX)_{i-1,i}=\exp(\bx_i)$ and by Lemma \ref{lem:Chen} have
\(\Sig(\bX)_{0,n}=\exp(\bx_1)\dotsm\exp(\bx_n)\) and therefore
\[
    \Omega_{0}=\log\left( \exp(\bx_1)\dotsm\exp(\bx_n)
    \right)\eqcolon\operatorname{BCH}(\bx_1,\dotsc,\bx_n).
\]
A computation based on Theorem \ref{thm:magnus_expansion}, but now applied without jumps,
reveals the general form
\begin{align*}
  \operatorname{BCH}(\bx_1,\dotsc,\bx_n) &=  \bx_n + \sum_{k=1}^{n-1}\int_0^1\Psi(\exp(\theta\ad\bx_k)\circ\exp(\ad\bx_{k+1})\circ\dotsm\circ\exp(\ad\bx_n))(\bx_k)\,\mathrm
  d\theta  \\
  &=\begin{multlined}[t] \sum_i\bx_i+\frac12\sum_{i<j}[\bx_i,\bx_j]+\frac{1}{12}\sum_{i<j}([\bx_i,[\bx_i,\bx_j]]+[\bx_j,[\bx_j,\bx_i]])\\
  -\frac16\sum_{i<j<k}[\bx_j,[\bx_i,\bx_k]]-\frac1{24}\sum_{i<j}[\bx_i,[\bx_j,[\bx_i,\bx_j]]]\dotsb\end{multlined}
\end{align*}
The flexibility of our Theorem \ref{thm:magnus_expansion} is then nicely illustrated by the fact that this $n$-fold BCH formula is an immediate consequence of  \eqref{equ:HauJum}, applied to a {\em piecewise constant} \cadlag{}  path $(\bX_t: 0 \le t \le n)$ with $\bX_\cdot - \bX_{i-1} \equiv \bx_i$ on $[i-1,i)$.
\end{remark}

\subsection{Diamond relations for multivariate cumulants}
\label{sub:diamond}
As in Section \ref{sec:trunc_tensor} we write $\Sy$ for the symmetric algebra over $\R^d$, and $\Sy_0,\Sy_1$ for those elements with scalar component $0,1$, respectively. Recall the exponential map $\exp: \Sy_0\to \Sy_1$ with global defined inverse $\log$.
Following Definition \ref{def:diamondSym} the diamond product for $\Sy_0$-valued semimartingales $\tilde \bX, \tilde \bY$ is another $\Sy_0$-valued semimartingale given by
$$
         (\tilde \bX \diamond \tilde \bY)_t(T) = \E_t \big( \langle \tilde \bX^c, \tilde \bY^c \rangle_{t,T} \big)
         =
          \sum  ( \E_t \langle \tilde \bX^{w_1}, \tilde \bY^{w_2} \rangle_{t,T})  \hat e_{w_1} \hat e_{w_2},
$$
with summation over all $w_1,w_2 \in \widehat \W_d$, provided all brackets are integrable. This trivially adapts to
$\Sy^N$-valued semimartingales, $N\in \mathbb{N}_{\ge 1}$, in which case all words have length less equal $N$, the summation is restricted accordingly to $|w_1+|w_2| \le N$.

\begin{theorem}\label{thm:main_with_jumps_sym}
(i) Let $\Xi= (0, \Xi^{(1)},\Xi^{(2)},...)$ be an $\F_T$-measurable random variable with values in $\Sy_0 (\R^d)$, componentwise in $\mathcal{L}^{\infty-}$.
Then
\[
  \KK_t(T) \coloneq  \log \E_t \exp (\Xi)
\]
satisfy the following functional equation, for all $0 \le t \le T$,
\begin{equation}\label{eq:main_with_jumps_sym}
   \KK_t(T)  = \E_t \Xi + \frac{1}{2} (\KK \diamond \KK)_t(T) +  \mathbb{J}_t(T)
\end{equation}
with jump component,
 \begin{equation*}
   \mathbb{J}_t(T) = \E_t \left( \sum_{t < u \le T} \left( e^{\Delta  \KK_u} - 1 - \Delta \KK_u \right)\right)
   =\E_t\left( \sum_{t < u \le T} \left( \frac{1}{2!}(\Delta  \KK_u)^2 + \frac{1}{3!} (\Delta\KK_u)^3 + \dotsb \right)\right).
\end{equation*}
Furthermore, if $N\in \mathbb{N}_{\ge 1}$, and $\Xi=(\Xi^{(1)},...,\Xi^{(N)})$ is $\F_T$-measurable with graded integrability condition
       \begin{equation} \label{ref:Ncond}
\Abss{\Xi^{(n)}}_{\Lcal^{N/n}} < \infty, \qquad n=1,...,N,
\end{equation}
then the identity (\ref{eq:main_with_jumps_sym})
holds for the
truncated signature cumulant $\KK^{(0,N)} \coloneq (\log\E_t(\Sig(\bX^{(0,N)})_{t,T}))_{0\le t \le T}$
with values in $\Sy^{(N)}_0 (\R^d)$.
\end{theorem}
\begin{remark} Identity (\ref{eq:main_with_jumps_sym}) is reminiscent of generalized Riccati equations for affine jump diffusions. The relation is, in a nutshell, that (\ref{eq:main_with_jumps_sym}) reduces to a PIDE system when the involved processes have a Markov structure. (We will make this point explicit in Section \ref{sec:Ladp} below, even in the fully non-commutative setting.) These PIDEs reduce to generalized Riccati under appropriate (affine linear) structure of the characteristics. The framework described here however requires neither Markov nor affine structure. We will show in Section \ref{sec:AVp} that such computations also possible in the fully non-commutative setting, i.e. to obtain signature cumulants. 
\end{remark}
\begin{proof}
  We first observe that since \(\Xi\in\mathcal L^{\infty-}\), by Doob's maximal inequality and the
  BDG inequality, we have that \(\tilde\bX_t\coloneq\E_t\Xi\) is a martingale in
  \(\HSe^{\infty-}(\Sy_0)\).
  In particular, thanks to \Cref{thm:bdg_signature}, the signature moments are well defined.
  According to \Cref{sec:mc}, the signature is then given by
  \[
    \Sig(\tilde\bX)_{t,T}=\exp(\Xi-\E_t\Xi),
  \]
  hence \(\kap_t(T)=\KK_t(T)-\tilde\bX_t\).

  Projecting \Cref{eq:master_with_jumps_h_form} onto the symmetric algebra yields
  \begin{align*}
      \kap_t(T)&=
    \begin{multlined}[t]
      \E_t\Bigg\{\tilde\bX_{t,T}+\frac12\langle\tilde\bX^c\rangle_{t,T}+\frac12\langle\kap(T)^c\rangle_{t,T}+\langle\tilde\bX^c,\kap(T)^c\rangle_{t,T}\Bigg.\\
        \Bigg.\quad+\sum_{t<u\le T}\left(
      e^{\Delta\tilde\bX_u+\Delta\kap_u(T)}-1-\Delta\tilde\bX_u-\Delta\kap_u(T) \right)\Bigg\}
    \end{multlined}\\
    &= \E_t\left\{\Xi+ \frac12\langle\KK(T)^c\rangle_{t,T} +\sum_{t<u\le T}\left(
    e^{\Delta\KK_u(T)}-1-\Delta\KK_u(T) \right) \right\}-\tilde\bX_t,
  \end{align*}
  and \cref{eq:main_with_jumps_sym} follows upon recalling that
  \((\KK\diamond\KK)_t(T)=\E_t\langle\KK(T)^c\rangle_{t,T}\).
The proof of the truncated version is left to the reader.
\end{proof}

As a corollary, we provide a general view on recent results of \cite{alos2018exponentiation,lacoin2019probabilistic,friz2020cumulants}.
Note that we also include jump terms in our recursion.
\begin{corollary}
\label{thm:diamond_recursion}
The conditional multivariate cumulants $(\KK_t)_{0\le t\le T}$ of a random variable \(\Xi\) with values in
\(\Sy_0(\R^d)\), componentwise in \(\mathcal L^{\infty-}\) satisfy the recursion
\begin{equation}\label{eq:diamond_recursion}
  \KK^{(1)}_t = \E_t(\Xi^{(1)}) \quad \text{and} \quad \KK^{(n)}_t =
  \E_t(\Xi^{(n)})+\frac{1}{2}\sum_{k=1}^{n}\left( \KK^{(k)} \diamond \KK^{(n-k)}\right)_t(T)+\mathbb J^{(n)}_t(T) \quad \text{ for } \quad n \ge 2,
\end{equation}
with
\[
  \mathbb J^{(n)}_t(T)=\E_t\left( \sum_{t<u\le
  T}\sum_{k=2}^n\frac{1}{k!}\sum_{\|\ell\|=n,|\ell|=k}\Delta\KK^{(\ell_1)}_u(T)\dotsm\Delta\KK_u^{(\ell_k)}(T)
\right).
\]
The analogous statement holds true in the $N$-truncated setting, i.e. as recursion for $n=1,..,N$ under the condition (\ref{ref:Ncond}).
\end{corollary}

\begin{example}[Continuous setting] In case of absence of jumps and higher order information (i.e. $\mathbb{J} \equiv 0,\Xi^{(2)} = \Xi^{(3)} = ... \equiv 0$, this type of cumulant recursion appears in \cite{lacoin2019probabilistic} and under optimal integrability conditions $\Xi^{(1)}$ with finite $N$.th moments, \cite{friz2020cumulants}. (This requires a localization argument which is avoided here by directly working in the correct algebraic structure.)
\end{example}

\begin{example}[Discrete filtration]
As opposite of the previous continuous example, we consider a purely discrete situation, starting from a discretely filtered probability space
with filtration $(\mathcal{F}_t\colon t = 0,1,\dotsc,T \in \N)$. For $\Xi$ as in Corollary \ref{thm:diamond_recursion}, a discrete martingale is defined by $\E_t \exp (\Xi)$, which may regard as c\'adl\'ag  semimartingale with respect to $\mathcal{F}_t \coloneq \mathcal{F}_{[t]}$, and similar for $\KK_t(T) = \log \E_t \exp (\Xi) \in\Sy_0$, i.e. the conditional cumulants of $\Xi$. Clearly, the continuous martingale part of $\KK_(T)$ vanishes, as does any diamond product with $\KK_(T)$. What remains is the functional equation
$$
      \KK_t(T) = \E_t (\Xi) + \mathbb{J}_t(T) =  \E_t (\Xi) + \E_t \bigg( \sum_{u=t+1}^T
       \big( e^{\Delta  \KK_u}
- 1 - \Delta \KK_u \big)\bigg)
$$
As before, the resulting expansions are of interest. On the first level, trivially, $\KK^{(1)}_t = \E_t(\Xi^{(1)})$,  whereas on the second level we see
$$
        \KK_t^{(2)}(T) = \E_t(\Xi^{(2)})+ \E_t \bigg( \sum_{u=t+1}^T
        (\E_u (\Xi^{(1)})-\E_{u-1} (\Xi^{(1)}))^2 \bigg)
$$
which one can recognize, in case $\Xi^{(2)} = 0$ as energy identity for the discrete square-integrable martingale $\ell_u := \E_u \Xi^{(1)}$.
Going further in the recursion yields increasingly non-obvious relations. Taking $\Xi^{(2)} = \Xi^{(3)} = ... \equiv 0$ for notational simplicity gives
$$
 \KK_t^{(3)}(T)  =   \mathbb{E}_t \left( \sum_{u = t +
  1}^T (\ell_u - \ell_{u - 1})^3 + 3 (\ell_u - \ell_{u - 1}) \{ \E_u
  \kappa (\ell, \ell)_{u, T} -\mathbb{E}_{u - 1} \kappa (\ell, \ell)_{u - 1,
  T} \} \right)
 $$
 It is interesting to note that related identities have appeared in the statistics literature under the name
{\em Bartlett identities}, cf. Mykland \cite{Myk1994} and the references therein.

\end{example}

\subsection{Remark on tree representation}
As illustrated in the previous section, in the case where \(d=1\), or when projecting onto the symmetric algebra,
our functional equation takes a particularly simple form (see \Cref{thm:main_with_jumps_sym}). If one further
specializes the situation, in particular discards all jump, we are from an algebraic perspective in the setting of
 Friz, Gatheral and Radoi\c{c}i\'c \cite{friz2020cumulants} which give a tree series expansion of cumulants using binary trees.
 This representation follows from the fact that the diamond product of semimartingales is commutative
but not associative. As an example (with notations taken from \Cref{sub:diamond}), in case of a one-dimensional continuous martingale, the first terms are
\[
	\mathbb
	K_t(T)=\Forest{[]}+\frac12\Forest{[[][]]}+\frac12\Forest{[[[][]][]]}+\frac12\Forest{[[[[][]][]][]]}+\frac18\Forest{[[[][]][[][]]]}+\dotsb
\]
This expansion is organized (graded) in terms of the number of leaves in each tree, and each leaf
represents the underlying martingale.

In the deterministic case, tree expansions are also known for the Magnus expansion 
\cite{iserles1999solution} and the BCH formula \cite{CM2009}.
These expansions also in terms of binary trees, but this time they are also required to be
non-planar to account for the non-commutativity of the Lie algebra.
As an example (with the notations of \Cref{sec:hausmag}), we have
\[
	\Omega_t(T) =
	\Forest{[]}+\frac12\Forest{[[[]][]]}+\frac{1}{12}\Forest{[[][[[]][]]]}+\frac{1}{4}\Forest{[[[[[]][]]][]]}+\dotsb
\]
In this expansion, the nodes represent the underlying vector field and edges represent integration
and application of the Lie bracket, coming from the \(\ad\) operator.

Since our functional equation and the associated recursion puts both contexts into a single common
framework. We suspect that our general recursion, Corollary \ref{cor:reucsrion_with_jumps} and thereafter, allows
for a sophisticated tree representation, at least in absence of jumps, and propose to return to this question in future work.


\section{Applications}
\subsection{Brownian and stopped Brownian signature cumulants}
\subsubsection{Time dependent Brownian motion}
Let $B$ be a $m$-dimensional standard Brownian motion defined on a portability space $(\Omega, \F, \PM)$ with the canonical filtration $(\F_t)_{t\ge0}$ and define the continuous (Gaussian) martingale $X = (X_t)_{0\le t\le T}$ by
\begin{align*}
    X_t = \int_0^{t} \sigma(u)\,\dd B_u, \quad 0 \le t \le T,
\end{align*}
with $\sigma \in L^2 ([0, T],\R^{m\times d})$. An immediate application of \Cref{thm:main_with_jumps} shows that
the integrability condition $\bX = (0, X, 0, \dots)\in\HSe^{\infty-}$ is trivially satisfied.
The {\em Brownian signature cumulants}
$\kap_t(T) =
\log(\E_t(\Sig(\bX)_{t,T}))$ satisfies the functional equation, with $a(t) \coloneq \sigma(t)\sigma(t)^T \in \mathrm{Sym}(\Rd \otimes \Rd),$
\begin{align}
\kap_t(T) = \int_t^{T} H(\ad{\kap_u(T)})(a(u)) \dd u, \quad 0 \le t \le T. \label{equ:BrowSigCum}
\end{align}
Therefore the tensor levels are precisely given by the Magnus expansion, starting with $$ \kap^{(1)}_t(T) = 0,\quad  \kap^{(2)}_t(T) = \tfrac{1}{2}\int_t^T a (u) \dd u, $$ and the general term
\begin{align*}
   \kap^{(2n-1)}_t(T) \equiv 0, \quad \kap^{(2n)}_t(T)
   &= \sum_{|\ell|\ge2, \Vert\ell\Vert=2n} \mathrm{HMag}^{2}(\bX, \kap; \ell)_{t,T} \\
   &= \sum_{\Vert\ell\Vert=n-1} \frac{B_{k}}{k!} \int_t^{T} \ad{\kap^{(2\cdot l_1)}_{u}} \cdots \ad{\kap^{(2\cdot l_{k})}_{u}} \left(a(u)\right)\dd u.
\end{align*}

Note that $\kap_t(T)$ is Lie in $\mathrm{Sym}(\Rd \otimes \Rd) \subset \tf$, but, in general, not a Lie series. In the special case $X=B$, i.e. $m=d$ and identity matrix $\sigma = \I_d= \sum_{i=1}^{d}e_{ii}\in \mathrm{Sym}(\Rd\otimes\Rd)$, all commutators vanish and we obtain what is known as \emph{Fawcett's formula} \cite{fawcett2002problems,friz2020course}.
\begin{align*}
\kap_t(T) = \tfrac{1}{2} (T-t)\I_d \, .
\end{align*}
\begin{example}
  Consider \(B^1,B^2\) two Brownian motions on the filtered space \((\Omega,\mathcal F,\mathbb P)\),
  with correlation \(\dd\langle B^1,B^2\rangle_t=\rho\,\dd t\) for some fixed constant \(\rho\in[-1,1]\).
  Suppose that \(K^1,K^2\colon[0,\infty)^2\to\R\) are two kernels such that \(K^i(t,\cdot)\in
  L^2([0,t])\)
  for all \(t\in [0,T]\), and set
  \[
    X^i_t\coloneq X_0^i+\int_0^tK^i(t,s)\,\dd B^i_s,\quad i=1,2
  \]
  for some fixed initial values \(X^1_0,X^2_0\).
  Note that neither process is a semimartingale in general. However, for each \(T>0\), the process
  \(\xi^i_t(T)\coloneq\E_t[X^i_T]\) is a martingale and we have
  \[
    \xi^i_t(T)=X^i_0+\int_0^tK^i(T,s)\,\dd B^i_s,
  \]
  that is, \((\xi^1,\xi^2)\) is a time-dependent Brownian motion as defined
  above.
  In particular, one sees that
  \[
    a(t)=\begin{pmatrix}\int_0^tK^1(T,u)^2\,\dd u&\rho\int_0^tK^1(T,u)K^2(T,u)\,\dd
    u\\\rho\int_0^tK^1(T,u)K^2(T,u)\,\dd u&\int_0^tK^2(T,u)^2\,\dd u\end{pmatrix}.
  \]
  \Cref{equ:BrowSigCum} and the paragraph below it then give an explicit recursive formula for the
  signature cumulants, the first of which are given by
  \begin{align*}
    \kap_t^{(1)}(T)&= 0,\\
    \kap_t^{(2)}(T)&=  \frac12\begin{pmatrix}\int_t^T\int_0^uK^1(T,r)^2\,\dd r\dd
      u&\rho\int_t^T\int_0^uK^1(T,r)K^2(T,r)\,\dd r\dd u\\[1ex]\rho\int_t^T\int_0^uK^1(T,r)K^2(T,r)\,\dd
    r\dd u&\int_t^T\int_0^uK^2(T,r)^2\,\dd r\dd u\end{pmatrix},\\
    \kap_t^{(3)}(T)&= 0,\\
    \kap_t^{(4)}(T)&= \frac1{2}\sum_{i,j,i',j'=1}^2\left[\int_t^T\int_u^T\left( a_{ij}(u)a_{i'j'}(r)-a_{i'j'}(u)a_{ij}(r)\right)\,\dd r\dd u\right]e_{iji'j'}.
  \end{align*}
  We notice that in the particular case when \(K^1=K^2\equiv K\), the matrix \(a\) has the form
  \[
    a(t)=\int_0^tK(T,u)^2\,\dd u\times\begin{pmatrix}1&\rho\\\rho&1\end{pmatrix}.
  \]
  Therefore, we have \(a(t)\otimes a(t')-a(t')\otimes a(t)=0\) for any \(t,t'\in[0,T]\).
  Hence, in this case, our recursion shows that for any \(\rho\in[-1,1]\),
  \[
    \kap_t^{(1)}(T)=0,\quad\kap_t^{(2)}(T)=\frac12\int_t^T\int_0^uK(T,r)^2\,\dd r\,\dd
    u\times\begin{pmatrix}1&\rho\\\rho&1\end{pmatrix},
  \]
  and \(\kap_t^{(n)}(T)=0\) for all \(0\le t\le T\) and \(n\ge 3\).
\end{example}

\subsubsection{Brownian motion up to the first exit time from a domain}

Let $B=(B_t)_{t\ge 0}$ be a $d$-dimensional Brownian motion defined on a filtered probability space $(\Omega, \F, \PM)$ with the canonical filtration $(\F_t)_{t\ge0}$ and a possibly random starting value $B_0$.
Assume that there is a family of probability measures $\{\PM^x\}_{x\in\Rd}$ on $(\Omega, \F)$ such that $\PM^x(B_0 = x) = 1$ and denote by $\E^x$ the expectation with respect to $\PM^x$.
Further let $\Gamma \subset \Rd$ be a bounded domain and define the stopping time $\tau_\Gamma$ of the first exit of $B$ from the domain $\Gamma$, i.e.
\begin{align*}
\tg = \inf\{t\ge0 \;\vert\;  B_t \in \Gamma^{c}\}.
\end{align*}
In \cite{lyons2015expected} Lyons--Ni exhibit an infinite system of partial differential equations for the expected signature of the Brownian motion until the exit time as a functional of the starting point.
The following result can be seen as the corresponding result for the signature cumulant, which follows directly
from the expansion in \Cref{thm:main_result}.
Recall that a boundary point $x \in \partial\Gamma$ is called \emph{regular} if and only if
\begin{align}
\PM^x\big( \inf\{t > 0 \;\vert\; B_t \in \Gamma^{c}\} = 0\big) = 1.
\end{align}
The domain $\Gamma$ is called regular if all points on the boundary are regular.
For example domains with smooth boundary are regular and see \cite[Section 4.2.C]{karatzas1991brownian} for a further characterization of regularity.

\begin{corollary}\label{cor:stopped_bm}
Let $\Gamma \subset \Rd$ be a regular domain, such that
\begin{align}\label{eq:tau_integrabiliy_assumption}
\sup_{x\in\Gamma}\E^x(\tau_\Gamma^n)<\infty, \quad n\in\None.
\end{align}
The signature cumulant $\kap_t = \log(\E(\Sig(B)_{t\wedge\tg, \tg}))$ of the Brownian motion $B$ up to the first exit from the domain $\Gamma$ has the following form
\begin{align*}
\kap_t = \tltg \mathbf{F}(B_t), \quad t\ge0,
\end{align*}
where $\mathbf{F} = \sum_{|w|\ge 2} e_w F^{w}$ with $F^{w}\in C^{0}(\ol{\Gamma},\R)\cap C^{2}(\Gamma,\R)$ is the unique bounded classical solution to the elliptic PDE
\begin{align}\label{eq:pde_domain_exit_cumulant}
-\Delta \mathbf{F} (x) &= \sum_{i=1}^{d}H(\ad{\mathbf{F} (x)})\Big(e_{ii} + Q(\ad{\mathbf{F} (x)})(\partial_i \mathbf{F} (x)^{\otimes 2}) + 2e_i G(\ad{\mathbf{F} (x)})(\partial_i \mathbf{F} (x)) \Big),
\end{align}
for all $x\in\Gamma$ with the boundary condition
$\mathbf{F}\vert_{\partial\Gamma} \equiv 0$.
\end{corollary}
\begin{proof}
Define the martingale $\bX = ((0, B_{t\wedge\tg}, 0, \dotsc) )_{t\ge0}\in \Se(\tf)$ and note that
$\abs{\QV{\bX}_\infty} = \tg$.
It then follows from the integrability of $\tg$ that $\bX \in \HSe^{\infty-}(\TT_0)$ and thus by \Cref{thm:bdg_signature} that $(\Sig(\bX)_{0,t})_{t\ge0} \in \HSe(\TT_1)^{\infty-}$.
This implies that the signature cumulant $\kap_t(T)\coloneq \log(\E_t(\Sig(\bX)_{t,T}))$ is well defined for all $0 \le t \le T < \infty$ and furthermore under (component-wise) application of the dominated convergence theorem that it holds
\begin{align*}
\kap_t = \lim_{T\to\infty} \kap_t(T) = \lim_{T\to\infty}\log(\E_t(\Sig(\bX)_{t,T})) = \log(\E_t (\Sig(B)_{t\wedge \tg, \tg})), \quad t\ge 0.
\end{align*}

Again by $\bX \in \HSe^{\infty-}(\TT_0)$ it follows that \Cref{thm:main_result} applies to the
martingale $(\bX_t)_{0\le t \le T}$ for any $T>0$ and therefore $\kap(T)$ satisfies the functional
equation \eqref{eq:master_with_jumps_h_form}.
It is well known that all martingales with respect to  the filtration $(\F_t)_{0 \le t \le T}$ are continuous, and therefore it is easy to see that also $\kap(T)\in\Sec(\tf)$.
Therefore \eqref{eq:master_with_jumps_h_form} simplifies to the following equation
\begin{multline}\label{eq:stopped_master}
    \kap_t(T) = \tltg \E_t \Biggl\{\frac{1}{2}\int_t^{\tg \wedge T}H(\ad{\kap_u})(\I_d)\,\dd u
    + \int_t^{\tg\wedge T} \frac{1}{2} H(\ad{\kap_u})\circ Q(\ad{\kap_u})(\dd\outerbracket{\kap}{\kap}_{u})\\
    + \int_t^{\tg\wedge T} H(\ad{\kap_u}) \circ (\mathrm{Id} \odot G(\ad{\kap_u}))
  (\dd\outerbracket{X}{\kap}_u)\Biggr\},
\end{multline}
where we have already used the martingality of $\bX$ and the explicit form of the quadratic variation $\QV{\bX}_t = \I_d(t \wedge \tg)$ with $\I_d = \sum_{i=1}^{d}e_{ii} \in (\Rd)^{\otimes 2}$.
It follows that $\kap^{(1)} \equiv \kap(T)^{(1)} \equiv 0$ and for the second level we have from the integrability of $\tg$ and the strong Markov property of Brownian motion that
\begin{align*}
\kap_t^{(2)} =\frac{1}{2}\I_d  \lim_{T\to\infty}\E_t\left(  \tltg(\tg\wedge T -t) \right)= \frac{1}{2}\I_d \tltg\left.\E^{x}(\tg)\right\vert_{x=B_t}, \quad t\ge0.
\end{align*}
Now note that
the function $u(x) \coloneq  \E^x(\tg)$ for $x \in \Gamma$ is in $C^{0}(\overline\Gamma, \R)\cap C^{2}(\Gamma, \R)$ and solves the Poisson equation
$ -(1/2)\Delta u = g$ with boundary condition $u\vert_{\partial \Gamma} = 0$ and data $g \equiv 1$.
Indeed, since $\Gamma$ is regular and $g$ is bounded and differentiable, this follows from Theorem 9.3.3 (and the remark thereafter) in \cite{oksendal2014stochastic}.
Moreover from the assumption \eqref{eq:tau_integrabiliy_assumption} we immediately see that $u$ is bounded on $\ol\Gamma$ and it follows from Theorem 9.3.2 in \cite{oksendal2014stochastic} that $u$ is the unique bounded classical such solution.
Thus we have shown that the statement holds true up to the second tensor level with $\mathbf{F}^{(1)}\equiv 0$ and $\mathbf{F}^{(2)} = \I_d u$ under the usual notation $\mathbf{F}^{(n)} = \sum_{|w|=n}e_wF^{w}$.

Now assume that the statement of the corollary holds true up to the tensor level $(N-1)$ for some $N \ge 3$.
Then, for any $n, k < N$ we have by applying  Itô's formula
\begin{align*}
\outerbracket{\kap^{(n)}}{\kap^{(k)}}_t = \sum_{i=1}^{d}\int_0^{t\wedge \tg}(\partial_i
\mathbf{F}^{(n)}(B_u)) \otimes (\partial_i \mathbf{F}^{(k)}(B_u))\,\dd u, \quad t\ge 0,
\end{align*}
and
\begin{align*}
\outerbracket{\bX}{\kap^{(n)}}_t = \sum_{i=1}^{d}\int_0^{t\wedge \tg} e_i \otimes (\partial_i
\mathbf{F}^{(n)}(B_u))\,\dd u, \quad t\ge 0.
\end{align*}

Further define the function $\mathbf{G}^{(N)}$ by the projection under $\pi_N$ of the right hand side of \eqref{eq:pde_domain_exit_cumulant} multiplied by the factor $1/2$.
Then applying \Cref{thm:main_with_jumps} to $\bX^{(0,N)}$ on the probability space $(\Omega, \F, \PM^{x})$ we see that it follows from the estimate \eqref{eq:l_tilde_estimate} that there exists a constant $c>0$ such that
\begin{align*}
\sup_{x\in\Gamma}\E^x \left\{ \int_0^\tg \big\vert \mathbf{G}^{(N)}(B_u) \big\vert\,\dd u\right\} \le c \sup_{x\in\Gamma}\Aabss{\bX^{(0,N)}}_{\HSehom^{1,N}(\mathbb{P}^{x})} = c \sup_{x\in\Gamma} \E_x(\tau_\Gamma^{N}) <\infty
\end{align*}
Therefore it follows, from projecting \eqref{eq:stopped_master} to level $N$ and using
the dominated convergence theorem to pass to the $T\to\infty$ limit, that $\kap^{(N)}$ is of the form
\begin{align*}
    \kap_t^{(N)} =  \tltg \mathbf{F}^{(N)}(B_t)
    \quad\text{with}\quad
    \mathbf{F}^{(N)}(x)\coloneq \E^x \left\{  \int_0^{\tg} \mathbf{G}^{(N)}(B_u)\,\dd u\right\},\quad x\in\ol{\Gamma}.
\end{align*}
Furthermore, by the assumption it also holds that $G^{w} \in C^1(\Gamma)$ for all $w\in\W_d$, $|w|=N$.
Therefore we can conclude again with Theorem 9.3.3 in \cite{oksendal2014stochastic} that $F^{w} \in C^{0}(\ol\Gamma, \R)\cap C^{2}(\ol\Gamma, \R)$ solves the Poisson equation with data $g=G^{w}$ for all words $w$ with $|w|=N$.
The statement then follows by induction.
\end{proof}

\begin{example} For $n \in
  \{1,\dotsc,d\}$, let $\mathbb{D}^n $ be the open unit ball in $\R^n$ 
and define the (regular) domain $\Gamma = \mathbb{D}^n \times \R^{d-n}\subset \R^d$.
Further note that it holds
\begin{align*}
\tg = \inf\{t \ge 0\;\vert\; B_t \notin \Gamma\} = \inf\{t \ge 0\;\vert\; \vert(B^1_t, \dots, B^n_t)\vert \ge 1\}.
\end{align*}
Hence we readily see that $\tg$ satisfies the condition \eqref{eq:tau_integrabiliy_assumption}.
Applying \Cref{cor:stopped_bm} it follows that the signature cumulant of the Brownian motion $B$ up to the exit of the domain $\Gamma$ is of the form $\kap_t = \tltg \mathbf{F}(B_t)$, where $\mathbf{F}$ satisfies the PDE \eqref{eq:pde_domain_exit_cumulant}.
Recall that $\mathbf{F}^{(1)}\equiv 0$ and projecting to the second level we see that
\begin{align*}
-\Delta\mathbf{F}^{(2)}(x)= \I_d, \quad x\in\Gamma; \qquad \mathbf{F}^{(2)}\vert_{\partial\Gamma} \equiv 0.
\end{align*}
The unique bounded solution the above Poisson equation is given by
$$\mathbf{F}^{(2)}(x) = \frac{1}{2}\I_d \left(1-\sum_{i=1}^n x_i^2\right), \quad x \in \Gamma.$$
More generally, we see that the Poisson equation $\Delta u = -g$ on $\Gamma$ with zero boundary
condition, where $g\colon\Gamma \to \R$ is a polynomial in the first $n$-variables, has a unique bounded solution $u$ which is also a polynomial of the first  $n$-variables of degree $\mathrm{deg}(u) = \mathrm{deg}(g)+2$ and has the factor $(1-\sum_{i=1}^n  x_i^2)$ (see Lemma 3.10 in \cite{lyons2015expected}).
Hence it follows inductively that each component of $\mathbf{F}^{(n)}$ is a polynomial of degree $n$ with the factor $(1-\sum_{i=1}^n x_i^2)$.
The precise coefficients of the polynomial can be obtained as the solution to a system of linear equations recursively derived from the forcing term in \eqref{eq:pde_domain_exit_cumulant}.
This is similar to \cite[Theorem 3.5]{lyons2015expected}, however we note that a direct conversion of the latter result for the expected signature to signature cumulants is not trivially seen to yield the same recursion and requires combinatorial relations as studied in \cite{bonnier2019signature}.
\end{example}

\subsection{Lévy and diffusion processes} \label{sec:Ladp}

Let $X\in\Se(\R^d)$ and throughout this section assume that the filtration $(\F_t)_{0 \le t \le T}$ is generated by $X$.
Denote by $\varepsilon_a$ the Dirac measure at point $a\in\Rd$, the random measure $\mu^{X}$ associated to the jumps of $X$ is an integer-valued random measure of the form
$$
\mu^{X}(\omega; \dd t, \dd x) \coloneq \sum_{s\geq0} \indik_{\{\Delta X_{s}(\omega) \neq 0\}}
\varepsilon_{(s, \Delta X_{s}(\omega))} (\dd t, \dd x).
$$
There is a version of the predictable compensator of \(\mu^X\), denoted by \(\nu\), such that the
\(\mathbb{R}^d\)-valued semimartingale \(X\) is quasi-left continuous if and only if \(\nu(\omega,
\{t\} \times \mathbb{R}^d) = 0\) for all \(\omega \in \Omega\), see \cite[Corollary II.1.19]{jacod2003limit}.
In general, \(\nu\) satisfies
$(|x|^{2} \wedge 1) \ast \nu \in \mathscr{A}_{\loc}$, i.e. locally of integrable variation. The semimartingale  $X$ admits a \emph{canonical representation} (using the usual notation for stochastic integrals with respect to random measures as introduced e.g. in \cite[II.1]{jacod2003limit})
\begin{align}\label{eq:representation_from_characteristics}
X=X_{0}+B(h)+X^{c}+(x-h(x)) \ast \mu^{X} + h(x) \ast (\mu^{X}-\nu),
\end{align}
where $h(x) = x 1_{|x| \le1}$ is a truncation function (other choice are possible.)
Here $B(h)$ is a predictable $\R^d$-valued process with components in $\mathscr{V}$ and $X^{c}$ is the continuous martingale part of $X$.

Denote by $C$ the predictable $\R^{d} \otimes \mathbb{R}^d$-valued covariation process defined as
$C^{ij}\coloneq\langle X^{i,c}, X^{j,c} \rangle$. 
Then the triplet $(B(h), C, \nu)$ is called the \emph{triplet of predictable characteristics} of $X$ (or simply the \emph{characteristics} of $X$). 
In many cases of interest, including the case of Lévy and diffusion processes discussed in the
subsection below, we have differential characteristics $(b,c,K)$ such that $$\dd B_t = b_t (\omega)
\dd t, \ \dd C_t = c_t (\omega) \dd t,\  \nu(\dd t,\dd x) = K_{t}(\dd x;\omega) \dd t,$$
where $b$ is a $d$-dimensional predictable process, $c$ is a predictable process taking values in the set of symmetric non-negative definite $d\times d$-matrices and $K$ is a transition kernel from $(\Omega \times \R_{+}, \mathcal{B}^d)$ into $(\R^d, \mathcal{B}^d)$. We call such a process \textit{It\^o semimartingale} and the triplet $(b,c,K)$ its {\em differential} (or \emph{local}) \emph{characteristics}. This extends mutatis mutandis to an $\tf^N$ (and then $\tf$) valued semimartingale $\bX$, with local characteristics $(\mathbf{b},\mathbf{c},\mathbf{K})$.

While every It\^o semimartingale is quasi-left continuous it is in general not true that $\kap$ is continuous (with the notable exception of time-inhomogeneous Lévy processes discussed below) and therefore there is no significant simplification of the functional equation \eqref{eq:master_with_jumps_h_form} in these general terms.
The following example illustrates this point in more detail.

\begin{example} Take $X \in \Se(\R^d)$ and then $d=1$, so that we are effectively in the symmetric setting. In this case $\hatexp(\kap_t (T)) = \E_t(\hatexp({X_T -X_t}))$, in the power series sense of enlisting all moments with factorial factors. These can also be obtained by taking higher order derivatives at $u=0$ of $\E_t (e^{u (X_T -X_t)})$, now with the classical calculus interpretation of the exponential.
The important class of affine models satisfies
$$
      \E_t (e^{u X_T-u X_t}) = \exp ( \phi (T-t,u) +( \Psi (T-t,u)- u)X_t )
$$
In the Levy-case, we have the trivial situation $\Psi(\cdot,u) \equiv u$, but otherwise $(\phi,\Psi)$ solve (generalized) Riccati equations and are in particular continuous in $T-t$. We see that, in non-trivial situations, the log of $ \E_t(e^{u X_T-u X_t})$ and any of its derivatives will jump when $X$ jumps. In particular, $\kappa_t(T)$ will not be continuous in $t$, even if $X$ is quasi-left continuous.
Let us note in this context that, in the general non-commutative setting and directly from definition of $\kap$,
$$
       \exp(\kap_{t-}) = \E_{t-} ( \exp ( \Delta X_t)  \exp(\kap_{t}) ) = \E_{t-} ( \exp \kap_{t} )
$$
where the second equality holds true under the assumption of quasi-left continuity of $X$. If we assume for a moment $\F_{t-} = \F_t$, then we could conclude that $\kap_{t-} = \kap_t$ and hence (right-continuity is clear) that $\kap_t$ is continuous in $t$. Since we know that this fails beyond L\'evy processes, if follows that such left continuity of filtrations is not a good assumption, at least not beyond L\'evy processes.
\end{example}

\subsubsection{The case of time-inhomogeneous L\'evy processes}

We consider now a $d$-dimensional time-inhomogeneous L\'evy processes of the form
\begin{align}\label{eq:levy_decomp}
    X_t = \int_0^{t}b(u)\,\dd u + \int_0^{t}\sigma(u)\,\dd B_u + \int_{(0,t]} \int_{|x|\le 1} x \;
    (\mu^X-\nu)(\dd s, \dd x) + \int_{(0,t]} \int_{|x| > 1} x \; \mu^X(\dd s, \dd x),
\end{align}
for all $0\le t\le T$, with $b \in L^{1}([0,T],R^d)$, $\sigma \in L^2 ([0, T],\R^{m\times d})$, $B$
a $d$-dimensional Brownian motion, $\mu^X$ is an independent inhomogeneous Poisson random measure
with the intensity measure $\nu$ on $[0,T] \times\R^d$, such that $\nu(\dd t, \dd x) = K_t(\dd x)\dd t$ with L\'evy measures $K_t$, i.e. $K_t(\{0\})=0$, and
\begin{align*}
\int_0^T\int_{\Rd} ({|x|}^2 \wedge 1) K_t(\dd x)\dd t < \infty,
\end{align*}
and measurability of $t \mapsto K_t (A) \in [0,\infty]$, any measurable $A \subset \R^d$.
Consider further the condition
\begin{align}\label{cond:moment_condition_levy_measure}
    \int_0^T\int_{\Rd}\abs{x}^N \xgone K_t(\dd x) < \infty, 
\end{align}
for some integer $N\in\None$.
The Brownian case (\ref{equ:BrowSigCum}) then generalizes as follows.
\begin{corollary}\label{cor:inhom_levy}
Let $X$ be an inhomogenous L\'evy process of the form \eqref{eq:levy_decomp}, such that the family of L\'evy measures $\{K_t\}_{t>0}$ satisfy the moment condition \eqref{cond:moment_condition_levy_measure} for all $N\in\None$.
Then $X\in\HSe^{\infty-}(\Rd)$and the signature cumulant $\kap_t \coloneq  \log( \E_t(\Sig(X)_{t,T}))$ satisfies the following integral equation
\begin{align}\label{eq:inhom_levy_kintchin}
    \kap_t = \int_t^{T}H(\ad{\kap_{u}})(\etfrak(u))\,\dd u,\quad 0 < t \le T,
\end{align}
where $a(t) = \sigma(t) \sigma(t)^{T}\in 
\Rd\otimes\Rd \subset \tf$ and
\begin{align}\label{def:levy_kintchin_exponent}
\etfrak(t)\coloneqq  b(t) + \frac{1}{2}a(t) + \int_{\Rd}(\exp(x)-1-x\xlone) K_t(\dd x) \in \tf.
\end{align}
In case the L\'evy measures $\{K_t\}_{t>0}$ satisfy the condition \eqref{cond:moment_condition_levy_measure} only up to some finite level $N\in\None$, we have $X\in\HSe^{N}$ and the identity \eqref{eq:inhom_levy_kintchin} holds for the truncated signature cumulant in $\tf^N$.
\end{corollary}
\begin{remark} Corollary \ref{cor:inhom_levy} extends a main result of \cite{friz2017general}, where a L\'evy-Kintchin type formula was obtained for the expected signature of  L\'evy processes with triplet $(b,a,K)$. Now this is an immediate consequence of \eqref{eq:inhom_levy_kintchin}, with all commutators vanishing in time-homogeneous case, and explicit solution
\begin{align*}
\kap_t(T) = (T-t)\left(b+\frac{1}{2}a + \int_{\Rd}(\exp(x)-1-x\xlone) K(\dd x)\right).
\end{align*}
\end{remark}

\begin{proof}
Assume that the L\'evy measures $\{ K_t \}_{t>0}$ satisfy the condition \eqref{cond:moment_condition_levy_measure} for some $N\in\None$.
We will first show that $X\in\HSe^{N}(\Rd)$.
Note that the decomposition \eqref{eq:levy_decomp} naturally yields a semimartingale decomposition $X = M + A$, where the local martingale $M$ and the adapted bounded variation process $A$ are defined by
\begin{align*}
     M = \int_0^\cdot\sigma(u)\,\dd B_u + (x\indik_{\abs{x}\le    1})\ast(\mu^X - \nu),
     \qquad A = \int_0^\cdot b(u)\,\dd u + (x\indik_{\abs{x}>1})\ast\mu^X.
\end{align*}

Regarding the integrability of the $1$-variation of $A$ we have first note that it holds
\begin{align*}
    \abs{A}_{1-\var;[0,T]} = \int_0^T|b(u)|\,\dd u + (\abs{x}\mathbf{1}_{\abs{x}>1}) \ast \mu^X_T.
\end{align*}
Define the increasing, piecewise constant process $V \coloneq (\abs{x}\mathbf{1}_{\abs{x}>1}) \ast \mu^X$.
Since $b$ is deterministic and integrable over the interval $[0,T]$ it suffices to show that $V_T$ has finite $N$th moment.
To this end, note that it holds
\begin{align*}
\E(V_T) = \int_0^T \int_{\abs{x}> 1} \abs{x} K_t(\dd x) \dd t <\infty.
\end{align*}
Further it holds for any $n \in \{1, \dotsc,N\}$ that
\begin{align*}
V_T^n = \sum_{0 < t \le T}\left(V_t^n - V_{t-}^n\right)
= \sum_{0 < t \le T}\sum_{k=0}^{n-1} \binom{n}{k} V_{s-}^k(\Delta V_{s})^{n-k}
\end{align*}
and by definition $\Delta V_t = |\Delta X_t|\indik_{\abs{\Delta X_t}>1}$.
Now let $n = 2$ and $k \in \{0, \dotsc, n-1\}$ then we have
\begin{align*}
\E\left( \sum_{0 < t \le T} V_{s-}^k (\Delta V_{s})^{n-k} \right)
&= \E\left( \int_{0}^T \int_{\abs{x}> 1} V_{s-}^k \abs{x}^{n-k} K_t(\dd x) \dd t\right) \\
&\le  \E\left(V_{T}^k\right) \int_{0}^T \int_{\abs{x}> 1}  \abs{x}^{n-k} K_t(\dd x) \dd t < \infty.
\end{align*}
It then follows inductively that $\E(V_T^n)$ is finite for all $n = 1, \dotsc, N$ and hence that the
1-variation of $A$ has finite $N$-th moment.

Concerning the integrability of the quadratic variation of $M$, let $w\in\{1, \dots, d\}$, then it is well known that (see e.g. \cite[Ch. II Theorem 1.33]{jacod2003limit})
\begin{align*}
\QV{(x^w\indik_{\abs{x}\le 1}) \ast (\mu^X- \nu)}_T =(x^w)^2\indik_{\abs{x}\le1} * \nu_T,
\end{align*}
where $\QV{M}$ denotes the dual predictable projection (or compensator) of $\GQV{M}$.
Further using that the compensated martingale $(x\mathbf{1}_{\abs{x}\le1})\ast(\mu^X - \nu)$ is orthogonal to continuous martingales, we have
\begin{align*}
   \QV{M^w}_T = \int_0^T a^{ww}(t) \dd t + \int_0^T \int_{\abs{x}\le 1}(x^w)^2 K_t(\dd x) \dd t < \infty.
\end{align*}
Now let $q \in [1, \infty)$, then from Theorem 8.2.20 in \cite{cohen2015stochastic} we have the following estimation
\begin{align*}
    \E\left(\GQV{M^w}^q_T\right) \le c\, \E\left(\QV{M^w}_T^q + \sup_{0 \le t \le T}(\Delta M^w_t)^{2q}\right) \le c \left( \QV{M^w}^q_T + 1 \right) <\infty,
\end{align*}
where $c>0$ is a constant depending on $q$.

We have shown that $\bX = (0, X, 0, \dotsc, 0) \in\HSehom^{1,N}$ and it follows from \Cref{thm:main_with_jumps} that the signature cumulant $\kap_t = \log(\E_t(\Sig(\bX)_{t,T}))$ satisfies the functional equation \eqref{eq:master_with_jumps_h_form}.
On the other hand, it follows from the condition \eqref{cond:moment_condition_levy_measure} that $\mathfrak{y}$ in \eqref{def:levy_kintchin_exponent} is well defined.
Now define $\wt{\kap} = (\wt\kap_t)_{0 \le t \le T}$ by the identity \eqref{eq:inhom_levy_kintchin}.
Noting that $\wt\kap$ is deterministic and has absolutely continuous components it is easy to see that $\wt\kap$ also satisfies the functional equation \eqref{eq:master_with_jumps_h_form} for the semimartingale $X$.
It thus follows that $\kap$ and $\wt\kap$ are identical.
\end{proof}

\subsubsection{Markov jump diffusions}
The generator of a general Markov jump diffusion $X$ is given by
\begin{align}\label{eq:diffusion_generator}
  \mathcal{L}f(x) = \sum_{i}b^{i}(x)\partial_i f(x) + \sum_{i,j}a^{ij}(x)\partial_{i}\partial_j f(x)+ \int_{\Rd}\bigg( f(x+y)-f(x)-\mathbf{1}_{\abs{y}\le 1} \sum_{i}y^{i}\partial_i f(x)\bigg)K(x, \dd y),
\end{align}
where the summations are over $i,j\in\{1, \dots, d\}$, $\bb\colon \Rd \to \Rd$ and $\ba\colon \Rd \to \Rd \otimes \Rd$ (symmetric, positive definite)  are bounded Lipschitz, $K$ is a Borel transition kernel from $\Rd$ into $\Rd$ with $K(\cdot, \{0\}) \equiv 0$ and
\begin{align*}
\sup_{x\in\Rd}\int_{\Rd}(\abs{y}^{2}\wedge 1)K(x,\dd y) <\infty.
\end{align*}
Note that (the law of) $X$ is the unique solution to the martingale problem associated to $\Lcal$.
That said, the extensions to Markov processes with differential characteristics
$(\bb(t,x),\ba(t,x),K(t,x,\dd y))$, with associated local L\'evy generators \cite{stroock1975diffusion} is mostly notational.
For the construction of general jump diffusions and their semimartingale characteristics see e.g. \cite[Ch. III.2.c]{jacod2003limit} and \cite[XIII.3]{jacod1979calcul}.

The expected signature of $X$ was seen in \cite{friz2017general} (in \cite{ni2012expected} for the continuous case) to satisfy a system of (linear) partial integro-differential equations (PIDEs). Passage to signatures cumulants amounts to take the logarithm, which represents a non-commutative Cole--Hopf transform, with resulting quadratic non-linearity, if viewed as $\TT_1$-valued PIDE, resolved thanks to the graded structure so that again a system of (linear) PIDEs arises.
In the proof of the following corollary we will show how this PIDE can be derived from \Cref{thm:main_with_jumps}.
\begin{corollary}
Let $X$ be a $d$-dimensional Markov diffusion with generator given by \eqref{eq:diffusion_generator}, where the transition kernel $K$ have uniformly bounded moments of all orders, i.e.
\begin{align}\label{eq:markov_kernel_moment_condition}
\sup_{x\in\Rd} \left(\int_\Rd \abs{y}^n K(x, \dd y) \right)< \infty, \quad n \in\N.
\end{align}
Then $X\in\HSe^{\infty-}$ and the signature cumulant is of the form
\begin{align*}
\kap_t(T) = \bv(t, X_t; T) = \bv(t,X_t),
\end{align*}
where $\bv=\sum_{w} \bv^{w} e_w$ is the unique solution with $\bv^{w} \in C^{1,2}_b([0,T] \times \Rd; \R)$ for all $w\in\W_d$ of the following partial integro-differential equation
\begin{equation}\label{eq:jmp_diff_PIDE}
  -[\partial_t + \Lcal]\bv = \begin{multlined}[t]H(\ad \bv)\bigg\{\bb + \frac12\ba + \frac12\sum_{i,j} \ba^{ij}Q(\ad \bv)(\partial_i \bv \otimes \partial_j \bv) +  \sum_{i,j} \ba^{ij}e_j G(\ad \bv)\partial_i \bv \bigg\}\\
+ \int_{\Rd} \bigg\{H(\ad \bv)\Big(\exp(y)\exp(\bv\circ \tau_y)\exp(-\bv) - 1 -
\mathbf{1}_{\abs{y}\le1}y\Big) - \big( \bv \circ \tau_y - \bv \big) \bigg\}K(\cdot, \dd
y),\end{multlined}
\end{equation}
on $[0,T]\times\Rd$ with terminal condition $\bv(T, \cdot) \equiv 0$, where $\tau_y(t, x) = (t, x + y)$.
\begin{proof}
First note that $X$ has the semimartingale characteristics $(B, C, \nu)$ where (see \cite[XIII.3]{jacod1979calcul})
\begin{align*}
\dd B_t = \bb(X_{t-}) \dd t, \quad \dd C_t =\ba(X_{t-})\dd t, \quad \nu(\dd t, \dd x) = \dd t K(X_{t-}, \dd x),
\end{align*}
with respect to the truncation function $h(x)=\mathbf{1}_{\abs{x} \le 1}$.
Further denote by $\mu^{X}$ the random measure associated with the jumps of $X$ and recall the canonical representation \eqref{eq:representation_from_characteristics}.
We can easily verify that the boundedness of $\bb$ and $\ba$, and the moment condition \eqref{eq:markov_kernel_moment_condition} implies that $\bX = (0, X, 0, \dots) \in \HSe^{\infty -}$ (compare also with the proof of \Cref{cor:inhom_levy}).
It then follows from \Cref{thm:main_with_jumps} that $\kap(T) = (\E_t(\Sig(\bX)_{t,T}))_{0 \le t \le T}$ is the unique solution to the functional equation \eqref{eq:master_with_jumps_h_form}.

Now assume that $\bv$ is the (unique) solution to the above PIDE with $\bv^{w}\in
C^{1,2}_b([0,T]\times\Rd;\R)$ for all $w\in\W_d$ (this is really an infinite-dimensional system of
linear PIDEs, solved inductively upon projection to the linear span of $e_w$ with $|w| \le \ell$, for $\ell \in \None$, see that standard results as found in \cite[Section 12.2]{cont2004financial} and references therein apply).
Then define $\tkap\in\Se(\tf)$ by $\tkap_t \coloneq \bv(t,X_t)$ for all $0 \le t \le T$ and note that $\tkap_{t-} = \bv(t, X_{t-})$.
We are going to show that also $\tkap$ also satisfies the functional equation \eqref{eq:master_with_jumps_h_form}.

Since $X$ solves the martingale problem with generator $\Lcal$ and $\bv$ is sufficiently regular it holds
\begin{align}\label{eq:mtg_prob_on_v}
\tkap_t = -\E_t\big(\bv(T, X_T)- \bv(t, X_t) \big) = \E_t\bigg(-\int_t^{T}[\partial_t + \Lcal] \bv (u,X_{u-}) \dd u \bigg).
\end{align}
On the other hand, we can plug in $\tkap$ into the right-hand side of \eqref{eq:master_with_jumps_h_form}.
We then obtain for the first integral inside the conditional expectation
\begin{align*}
\int_{(0,t]}H(\ad \tkap_{u-})(\dd\bX_u)
= \int_{(0,t]}H(\ad \tkap_{u-})(\dd B_u + \dd X^{c}_u) + \bW\ast(\mu^{X} - \nu)_t + \ol \bW\ast\mu^{X}_t,
\end{align*}
where
\begin{equation*}
\bW_t(y) \coloneq H(\ad \tkap_{t-})(h(y)), \quad\text{and}\quad \ol \bW_t(y) \coloneq H(\ad \tkap_{t-})(y - h(y)),
\end{equation*}
for all $0 \le t \le T$ and $y\in\Rd$.
Similarly we have
\begin{align*}
\sum_{0 < u \le t}\bigg\{H(\ad{\tkap_{u-}})\Big(\exp(\Delta \bX_u)\exp(\tkap_u)\exp(-\tkap_{u-}) - 1 -\Delta \bX_u\Big) - \Delta \tkap_u \bigg\} = \bJ \ast \mu^{X}_t,
\end{align*}
where $0 \le t \le T$ and $y\in\Rd$
\begin{align*}
\bJ_t(y) \coloneq \bigg\{H(\ad\bv)\Big(\exp(y)\exp(\bv\circ\tau_y)\exp(-\bv) - 1 - y\Big) - (\bv\circ\tau_y - \bv) \bigg\}(t, X_{t-}).
\end{align*}
Finally for the quadratic variation terms with respect to continuous parts we have
\begin{align*}
\bU_t &\coloneqq \int_0^{t}H(\ad \tkap_{u-})\Big\{\dd\QV{\bX}_u + \big(\Id\odot G(\ad \tkap_{u-})\big)\big(\dd \outerbracket{\bX^{c}}{\tkap^{c}}_u\big)  + Q(\ad \tkap_{u-})\big(\dd \outerbracket{\tkap^{c}}{\tkap^{c}}_u\big)\Big\} \\
&= \sum_{i,j}\int_0^{t}\ba^{ij}H(\ad \bv)\bigg\{e_{ij} + (\Id\odot G(\ad \bv))(e_i\otimes\partial_j \bv) + Q(\ad \bv)(\partial_i \bv \otimes\partial_j\bv) \bigg\} (t,X_{u-}) \dd u \\
&\eqqcolon \int_{0}^{t} H(\ad \bv)(\bu(u, X_{u-}))\dd u.
\end{align*}
Provided that we can show the following integrability property holds for all words $w\in\W_d$
\begin{equation}\label{eq:proof_diff_integrability}
 \E\bigg( \int_0^{T}\abs{\big\{H(\ad \tkap_{u-})(\bb(X_{u-}) + \bu(u, X_{u-})) \big\}^{w}}\dd u + \big(|\bW^w|^2 + |\ol \bW^w| + |\bJ^w|\big)\ast\nu_T\bigg) < \infty,
\end{equation}
it follows that
\begin{align*}
\MoveEqLeft\E_t\bigg\{\int_{(t,T]}H(\ad \tkap_{u-})(\dd\bX_u) + \bU_{t,T} +  \bJ \ast \mu^{X}_{t,T}\bigg\} \\
&= \E_t\bigg\{ \int_t^{T} H(\ad \tkap_{u-})(\dd B_u) +  \bU_{t,T} +  (\bJ - \ol \bW) \ast \nu_{t,T}\bigg\} \\
&= \E_t\bigg\{ \int_{(t,T]}\bigg( H(\ad \bv)\big(\bb(X_{u-}) + \bu(u, X_{u-})\big) + \int_{\Rd}(\bJ_{u}(y) - \ol\bW_{u}(y)) K(X_{u-}, \dd y)\bigg) \dd u\bigg\} \\
&= \E_t\bigg(-\int_t^{T}[\partial_t + \Lcal] \bv (u,X_{u-}) \dd u \bigg).
\end{align*}
where in the last line we have used $\bv$ satisfies the PIDE.
Since the above left-hand side is precisely the right-hand side of the functional equation \eqref{eq:master_with_jumps_h_form}, it follows together with \eqref{eq:mtg_prob_on_v} that $\tkap$ satisfies the functional equation \eqref{eq:master_with_jumps_h_form}.

Note that in case the integrability condition  \eqref{eq:proof_diff_integrability} is satisfied for all words $w\in\W_d$  with $\abs{w} \le n$ for some length $n\in\None$ it follows that the above equality holds up to the projection with $\pi_{(0,n)}$.
For words with $\abs{w} = 1$ the condition \eqref{eq:proof_diff_integrability} is an immediate consequence of $\bX\in\HSe^{\infty-}$.
It then follows inductively, by the same arguments as in the proof of \Cref{claim:proof_L_estimate} that \eqref{eq:proof_diff_integrability} is indeed satisfied for all words $w\in\W_d$.

Since $\kap(T)$ is the unique solution to \eqref{eq:master_with_jumps_h_form} it then follows that $\tkap \equiv \kap(T)$.

\end{proof}
\end{corollary}

\subsection{Affine Volterra processes} \label{sec:AVp}
For $i =1, 2$ let $K^{i}$ be an integration kernel such that $K^{i}(t, \cdot) \in L^{2}([0, t])$ for all $0 \le t \le T$ and let $V^{i}$ be the solution to the Volterra integral equation
\begin{align*}
V^{i}_t = V^{i}_0 + \int_0^{t} K^{i}(t,s)\sqrt{V^{i}_s}\dd W^{i}_s, \quad 0 \le t\le T,
\end{align*}
with $V^{i}_0 > 0$, where $W^{1}$ and $W^{2}$ are uncorrelated standard Brownian motions which generate the filtration $(\F_t)_{0 \le t \le T}$.
Note that in general $V^{i}$ is not a semimartingale.
In particular this is not the case when $K^{i}$ is a power-law kernel of the form $K(t,s) \sim (t-s)^{H-1/2}$ for some $H\in (0, 1/2)$, which is the prototype of a \emph{rough affine volatility model} (see e.g. \cite{keller2018affinerough}).
However, a martingale $\xi^{i}(T)$ is naturally  associated to $V^{i}$ by
$$\xi^{i}_t(T) = \E_t(V^{i}_T), \quad 0 \le t \le T.$$
In the financial context, $\xi^{i}(T)$ is the central object of a \emph{forward variance model} (see e.g. \cite{gatheral2019affine}).
It was seen in \cite{friz2020cumulants} that the iterated diamond products of $\xi^{1}(T)$ are of a particularly simple form and easily translated to a system of convolutional Riccati equations of the type studied in \cite{abijaber2019affine}, \cite{gatheral2019affine} for the cumulant generating function.
We are interested in the signature cumulant of the two dimensional martingale $X = (\xi^{1}(T), \xi^{2}(T))$.

\begin{corollary}
It holds that $\bX = (0,\, \xi^{1}(T)e_1 + \xi^{2}(T)e_2,\, 0,\, \dots)\in\HSe^{\infty-}$ and the signature cumulant $\kap_t(T) = \log \E_t(\Sig(\bX)_{t,T})$ is the unique solution to the functional equation: for all $0 \le t \le T$
\begin{align*}
\kap_t(T) = -\E_t\Bigg(\frac{1}{2}\sum_{i=1,2}\int_t^{T} H(\ad\kap_u)(e_{ii})K^{i}(T,u)^{2}V_u^{i} \dd u + \frac{1}{2}\int_t^{T}H(\ad \kap_u)\circ Q(\ad\kap_u)(\dd (\kap \blackdiamond \kap)_u(T))& \\
+\sum_{i=1,2} \int_t^{T}H(\ad\kap_u)\left\{e_i G(\ad\kap_u)\big(\dd (\xi^{i}(T) \diamond \kap)_u(T)\big) \right\} &\Bigg).
\end{align*}
\end{corollary}
\begin{proof}
Regarding the integrability statement it suffices to check that $V^{i}_T$ has moments of all order for $i=1, 2$.
This is indeed the case and we refer to \cite[Lemma 3.1]{abijaber2019affine} for a proof.
Hence we can apply \Cref{thm:main_with_jumps} and we see that $\kap$ satisfies the functional equation \eqref{eq:master_with_jumps_h_form}.
As described in \Cref{sec:master_equation} this equation can be reformulated with brackets replaced by diamonds.
Further note that, due to the continuity, jump terms vanish and, due to the martingality, the Itô integrals with respect to $\bX$ have zero expectation.
The final step to arrive at the above form of the functional equation is to calculated the brackets $\langle \xi^{i}(T), \xi^{j}(T) \rangle$.
From the definition $\xi^{i}(T)$ and $V^{i}$ we have for all $0 \le t \le T$
\begin{align*}
\xi^{i}_t(T) &= \E_t \left( V^{i}_0 +\int_0^{t}K^{i}(T,s)\sqrt{V^{i}_s}  \dd W^{i}_s+  \int_t^{T}K^{i}(t,s)\sqrt{V^{i}_s}  \dd W^{i}_s\right) = V^{i}_0 + \int_0^{t}K^{i}(T,s)\sqrt{V^{i}_s}  \dd W^{i}_s.
\end{align*}
Therefore and due to the independence we have $\langle \xi^{1}(T), \xi^{2}(T) \rangle = 0$  and for the square bracket we have
$\dd \langle \xi^{i}(T), \xi^{i}(T)\rangle_t = K^{i}(T,t)^{2}V^{i}_t \dd t$.
\end{proof}

The recursion for the signature cumulants from \Cref{cor:recursion_h_form} are easily simplified in analogy to the above corollary.
In the rest of this section we are going to demonstrate explicit  calculations for the first four levels.
Clearly, due to the martingality the first level signature cumulants are identically zero $\kap^{(1)}(T) \equiv 0$.
In the second level we start to observe the type of simplifications that appear due to the affine structure
\begin{align*}
\kap^{(2)}_t(T) &= \frac{1}{2}\sum_{i=1,2}e_{ii}(\xi^{i}(T) \diamond \xi^{i}(T))_t(T) = \frac{1}{2}\sum_{i=1,2}e_{ii}\E_t\left( \int_t^{T}K^{i}(T,u)^{2}V_u^{i}\dd u\right) \\
&=\frac{1}{2}\sum_{i=1,2}e_{ii}  \int_t^{T}K^{i}(T,u)^{2} \xi^{i}_t(u) \dd u,
\end{align*}
where $\xi^{i}_t(u) = \E_t(V^{i}_u)$ for all $0 \le t \le u \le T$.
The third level is of the same form
\begin{align*}
\kap^{(3)}_t(T) &= \frac{1}{2}\sum_{i=1,2}e_{i}(\xi^{i}(T) \diamond \kap^{(2)}(T))_t(T) \\
&=\frac{1}{2}\sum_{i=1,2}e_{iii}\int_t^{T}\left(\int_u^{T}K^{i}(T,s)^{2}K^{i}(T, u)K^{i}(s,u) \dd s\right)\xi^{i}(u) \dd u,
\end{align*}
where we have used that for any suitable $h:[0, T]\to \R$ it holds for all $0 \le t \le T$
\begin{align*}
\int_t^{T}h(u)\xi^{i}_t(u)\dd u = \int_0^{T}h(u) V^{i}_0 \dd u - \int_0^{t}h(u)V^{i}_u\dd u + \int_0^{t}\left(\int_u^{T} h(s)K^{i}(s, u)\dd s \right)\sqrt{V_u^{i}}\dd W^{i}_u.
\end{align*}
The fourth level starts to reveal some of the structure that is not visible in the commutative setting
\begin{align*}
\kap^{(4)}_t(T) =& \sum_{i=1,2}\left\{\frac{1}{8}[e_{\bar i \bar i}, e_{ii}]\int_t^{T}\left(\int_u^{T}K^{\bar i}(T,s)\xi^{\bar i}\dd s\right) K^{i}(T, u)^{2} \xi_t^{i}(u) \dd u  + e_{iiii}\int_t^{T}h^{i}(T,u) \xi^{i}_t(u)\dd u\right\},
\end{align*}
where $\{i, \bar{i}\} = \{1,2\}$ and $h^{i}$ is defined by
\begin{align*}
h^{i}(T, u) =& \frac{1}{8}\left(\int_u^{T}K^{i}(T,s)^{2}K^{i}(u,s)\dd s\right)^{2} \\
&+ \frac{1}{2}\int_u^{T}\left(\int_s^{T}K^{i}(T,r)^{2}K^{i}(T, s)K^{i}(r,s) \dd r\right)K^{i}(T,s)K^{i}(s,u)ds, \quad 0 \le u \le T.
\end{align*}

\section{Proofs}
For ease of notation we introduce a norm on the space of tensor valued finite variation process, which could have been introduced already in \Cref{sec:tensor_semimartingales}, was however not needed until now.
Let $q\in[1,\infty)$ and $\bA\in\Fv((\Rd)^{\otimes n})$ for some $n\in\None$ then we define
\begin{align*}
\Abss{\bA}_{\Fv^{q}} \coloneqq \Abss{\bA}_{\Fv^{q}((\Rd)^{\otimes n})} \coloneqq \Abs{\abs{\bA}_{1-\var;[0,T]}}_{\Lcal^{q}}.
\end{align*}
It is easy to see that it holds $\Abss{\bA}_{\HSe^{q}} \le \Abss{\bA}_{\Fv^{q}}$ and this inequality can be strict.

Further for an element $\mathbb{A} \in \TT\otimes\TT$ we introduce the following notation
\[\mathbb{A}= \sum_{w_1, w_2 \in \W_d} \mathbb{A}^{w_1, w_2} \, e_{w_1}\!\otimes e_{w_2}, \quad \mathbb{A}^{w_1, w_2}\in\R,
\]
and for $l_1, l_2 \in \None$
\begin{align*}
    \mathbb{A}^{(l_1, l_2)} = \sum_{|w_1|=l_1, |w_2|=l_2} e_{w_1w_2}\otimes \mathbb{A}^{w_1,w_2} \in (\Rd)^{\otimes l_1} \otimes (\Rd)^{\otimes l_2} \subset \TT \otimes \TT.
\end{align*}

Next we will proof two well known lemmas translated to the setting of tensor valued semimartingales.
\begin{lemma}[Kunita-Watanabe inequality]\label{lem:cauchy_schwartz_tensor_bracket}
Let $\bX \in \Se((\Rd)^{\otimes n})$ and $\bY \in \Se((\Rd)^{\otimes n})$ then the following estimate holds a.s.
\begin{align*}
   \abs{\CV{\bX^c}{\bY^c}}_{1-\var;[0,T]} + \sum_{0<t\le T}\abs{\Delta \bX_t\Delta \bY_t} &\le \sum_{|w_1|=n}\sqrt{\GQV{\bX^{w_1}}_T}\sum_{|w_2|=m}\sqrt{\GQV{\bY^{w_2}}_T} \\
    &\le c \sqrt{\abs{\GQV{\bX}_T}}{\sqrt{\abs{\GQV{\bY}_T}}},
\end{align*}
where $c>0$ is a constant that only depends on $d$, $m$ and $n$.
\end{lemma}
\begin{proof}
From the definition of the quadratic variation of tensor valued semimartingales in \Cref{sec:tensor_semimartingales} we have
\begin{align*}
    \abs{\CV{\bX^c}{\bY^c}}_{1-\var;[0,T]} + \sum_{0<s\le T}\abs{\Delta \bX_s\Delta \bY_s}
    &\le \sum_{|w_1|=n, |w_2|=m}\int_0^T\abs{\dd\CV{\bX^{w_1c}}{\bY^{w_2c}}_s} + \sum_{0<s\le T}\abs{\Delta \bX^{w_1}_s\Delta \bY^{w_2}_s} \\
    &\le \sum_{|w_1|=n, |w_2|=m}\sqrt{\GQV{\bX^{w_1}}_T}\sqrt{\GQV{\bY^{w_2}}_T} \\
    &\le d^{(n+m)/2} \sqrt{\sum_{|w_1|=n}\GQV{\bX^{w_1}}_T}\sqrt{\sum_{|w_2|=m}\GQV{\bY^{w_2}}_T} \\
    &\le d^{n+m}  \abs{\GQV{\bX}_T}\abs{\GQV{\bY}_T},
\end{align*}
where the first estimate follows form the triangle inequality, the second estimate from the (scalar) Kunita-Watanabe inequality \cite[Ch. II, Theorem 25]{protter2005stochastic}  and the last two estimates follow from the standard estimate between the $1$-norm and the $2$-norm on $(\Rd)^{\otimes m}\cong\R^{d^m}$.
\end{proof}

In order to proof the next well known lemma (Emery's inequality) we need the following technical
\begin{lemma}\label{lem:one-variation_tensor_integrals} Let $\bA\in\Fv((\Rd)^{\otimes n})$, $\bY\in\D((\Rd)^{\otimes l})$, $\bZ\in\D((\Rd)^{\otimes m})$ then it holds
\begin{align*}
\abs{\int_{(0,\cdot]}\bY_{s-} \dd\bA_s \bZ_{s-}}_{1-\var;[0,T]} \le \int_{(0,T]}\abs{\bY_{s-}\bZ_{s-}}\abs{\dd \bA_s}
\end{align*}
where the integration with respect to $\abs{\dd \bA}$ denotes the integration with respect to the increasing one-dimensional path $(\abs{\bA}_{1-\var;[0,t]})_{0 \le t \le T}$.
Further, let $\bY^{\prime}\in\D((\Rd)^{\otimes l^{\prime}})$, $\bZ^{\prime}\in\D((\Rd)^{\otimes m^{\prime}})$ and let $(\mathbb{A}_t)_{0 \le t \le T}$ be a process taking values in $(\Rd)^{\otimes n}\otimes (\Rd)^{\otimes n^{\prime}}$ such that $A^{w_1, w_2} \in \Fv$ for all $w_1, w_2\in\W_d$ with $|w_1|= n$ and $|w_2| = n^{\prime}$. Then it holds
\begin{align*}
\abs{\int_{(0,\cdot]} (\bY_{s-} \Id \bY^{\prime}_{s-})\odot (\bZ_{s-} \Id \bZ^{\prime}_{s-})(\dd\mathbb{A}_s)}_{1-\var;[0,T]} \le \int_{(0,T]}\abss{\bY\bY^{\prime}\bZ\bZ^{\prime}}_{s-}\abs{\dd m(\mathbb{A})_s},
\end{align*}
where $(\bY \Id \bY^{\prime})(\bA) = \bY\bA \bY^{\prime}$ is the left- respectively right-multiplication by $\bY$ respectively $\bY^{\prime}$.
\end{lemma}
\begin{proof}
Let $0 \le s \le t \le T$ then it holds
\begin{align*}
\abs{\int_{(s,t]}\bY_{u-} \dd\bA_u \bZ_{u-}} \le \int_{(s,t]}\abs{\bY_{u-}\bZ_{u-}}\abs{\dd \bA_u}.
\end{align*}
Indeed, as it follows e.g. from \cite[Theorem on Stieltjes integrability]{young1936inequality}, we can approximate the integral in the left-hand side by Riemann sums.
Then for a partition $(t_i)_{i=1, \dotsc, k}$ of the interval $[s,t]$ we have
\begin{align*}
\abs{\sum_{i=1}^{k-1} \bY_{t_i-} (\bA_{t_{i+1}} - \bA_{t_i}) \bZ_{t_i-}}  \le \sum_{i=1}^{k-1} \abs{\bY_{t_i-} (\bA_{t_{i+1}} - \bA_{t_i}) \bZ_{t_i-}} \le \sum_{i=1}^{k-1} \abs{\bY_{t_i-} \bZ_{t_i-}}\abs{\bA_{t_{i+1}} - \bA_{t_i}},
\end{align*}
where the last inequality follows from the fact that for homogeneous tensors $\bx\in(\Rd)^{\otimes m}$ and $\by\in(\Rd)^{\otimes n}$ it holds that $\abs{\bx\by} = \abs{\by\bx} \le \abs{\bx}\abs{\by}$.
Regarding the $1$-variation we then have
\begin{align*}
\abs{\int_{(0,\cdot]}\bY_{s-} \dd\bA_s \bZ_{s-}}_{1-\var;[0,T]} &= \sup_{0 \le t_1 \le \cdots \le t_k \le T} \sum_{i=1}^{k} \abs{\int_{(t_{i}, t_{i+1}]}\bY_{s-} \dd\bA_s \bZ_{s-}} \\
&\le \sup_{0 \le t_1 \le \cdots \le t_k \le T} \sum_{i=1}^{k} \int_{(t_{i},t_{i+1}]}\abs{\bY_{s-}\bZ_{s-}}\abs{\dd \bA_s} \\
& = \int_{(0,T]}\abs{\bY_{s-}\bZ_{s-}}\abs{\dd \bA_s}.
\end{align*}

Regarding the second statement we see that for any $0 \le s \le t \le T$ we have
\begin{align*}
\abs{\int_{(s,t]} (\bY_{u-} \Id \bY^{\prime}_{u-})\odot (\bZ_{u-} \Id \bZ^{\prime}_{u-})(\dd\mathbb{A}_u)} \le \int_{(s,t]} \abss{\bY\bY^{\prime}\bZ\bZ^{\prime}}_{u-} \abs{\dd m(\mathbb{A})_u}
\end{align*}
Indeed, we approximate the integral in the right-hand side again by a Riemann sum.
Then for a partition $(t_i)_{i=1, \dotsc, k}$ of the interval $[s,t]$ we have
\begin{align*}
&\abs{\sum_{i=1}^{k-1} (\bY_{t_i-}\Id\bY^{\prime}_{t_i-})\odot(\bZ_{t_i-}\Id\bZ^{\prime}_{t_i-}) (\mathbb{A}_{t_{i+1}} - \mathbb{A}_{t_i})}  \\
&\hspace{5em}\le \sum_{i=1}^{k-1} \abs{\sum_{|w_1|=m, |w_2|=m^{\prime}}\bY_{t_i-}e_{w_1}\bY^{\prime}_{t_i-}\bZ_{t_i-}e_{w_2}\bZ^{\prime}_{t_i-} (\mathbb{A}_{t_{i+1}} - \mathbb{A}_{t_i})} \\
&\hspace{5em} \le \sum_{i=1}^{k-1} \abs{\bY_{t_i-}\bY^{\prime}_{t_i-}\bZ_{t_i-}\bZ^{\prime}_{t_i-}}\abs{ m(\mathbb{A}_{t_{i+1}}) - m(\mathbb{A}_{t_i})},
\end{align*}
where the last inequality follows from the definition of the norm on (homogeneous) tensors and the definition of the multiplication map $m$.
We conclude analogously to the proof of the first statement.
\end{proof}

\begin{lemma}[Emery's inequality]\label{lem:emerys}
Let $\bX\in\Se((\Rd)^{\otimes n})$, $\bY\in\D((\Rd)^{\otimes l})$ and $\bZ\in\D((\Rd)^{\otimes m})$ then for $p, q\in[1,\infty)$ and $1/r = 1/p + 1/q$ it holds
\begin{align*}
\Abs{\int_{(0,\cdot]}\bY_{s-} \dd\bX_s \bZ_{s-}}_{\HSe^{r}((\Rd)^{\otimes (l+n+m)})} \le c \Abss{\bY\bZ}_{\Se^{q}((\Rd)^{\otimes (l+m)})}\Abss{\bX}_{\HSe^{p}((\Rd)^{\otimes n})},
\end{align*}
where $c>0$ is a constant that only depends on $d$ and $m$.
\end{lemma}
\begin{proof}Let $\bX = \bX_0 + \bM + \bA$ be a semimartingale decomposition with $\bM_0 = \bA_0 = 0$.
Then it follows by definition of the $\HSe^{r}$-norm and the above \Cref{lem:one-variation_tensor_integrals}
\begin{align*}
\Abs{\int_{(0,\cdot]}\bY_{s-} \dd\bX_s \bZ_{s-}}_{\HSe^{r}}
&\le \Abs{\abs{\int_{(0,T]}(\bY_{s-} \Id\bZ_{s-})^{\odot 2}\dd\outerbracket{\bM}{\bM}_s}^{1/2} + \abs{\int_{(0,\cdot]}\bY_{s-} \dd\bA_s\bZ_{s-}}_{1-\var;[0,T]}}_{\Lcal^{r}} \\
&\le \Abs{\abs{\int_{(0,T]}\abs{(\bY_{s-}\bZ_{s-})^{2}} \abs{\dd\GQV{\bM}_s}}^{1/2} + \int_{(0,T]}\abs{\bY_{s-}\bZ_{s-}} \abs{\dd\bA_s}}_{\Lcal^{r}} \\
&\le \Abs{\sup_{0 \le s \le T}\abs{\bY_s\bZ_s}\left( \abs{\GQV{\bM}}_{1-\var;[0;T]}^{1/2} + \abs{\bA}_{1-\var;[0,T]}\right)}_{\Lcal^{r}} \\
&\le c \Abs{\bY\bZ}_{\Se^{q}}\Abs{\abs{\GQV{\bM}_T} + \abs{\bA_s}_{1-\var;[0,T]}}_{\Lcal^{p}},
\end{align*}
where we have used the generalized H\"older inequality and the Kunita-Watanabe inequality (\Cref{lem:cauchy_schwartz_tensor_bracket}) to get to the last line.
Taking the infimum of over all semimartingale decomposition $\bM + \bA$ yields the statement.
\end{proof}

The following technical lemma will be used in the proof of both \Cref{thm:bdg_signature} and \Cref{thm:main_with_jumps}.
\begin{lemma}\label{lem:homn_by_product}
Let $\bX, \bY\in\Se(\TT^{N})$, $N\in\None$, $q\in[1,\infty)$ and assume that there exists a constant $c>0$ such that
\begin{align*}
\Abss{\bY^{(n)}}_{\HSe^{qN/n}} \le c \sum_{\Vert\ell\Vert = n}
\Abss{\bX^{(l_1)}}_{\HSe^{qN/l_1}}\cdots\Abss{\bX^{(l_j)}}_{\HSe^{qN/l_j}}, \quad n=1, \dotsc, N,
\end{align*}
where the summation is over $\ell = (l_1, \dots, l_j)\in(\None)^{j}$, $j\in\None$, $\Vert\ell\Vert = l_1 + \dots + l_j$.
Then there exists a constant $C>0$, depending only on $c$ and $N$, such that
\begin{align*}
\homnqNs{\bY} \le C \homnqNs{\bX}.
\end{align*}
\end{lemma}
\begin{proof} Note that for any $n\in\{1, \dotsc, N\}$ it holds
\begin{align*}
\bigg(\sum_{\Vert\ell\Vert = n} \Abss{\bX^{(l_1)}}_{\HSe^{qN/l_1}}\cdots\Abss{\bX^{(l_j)}}_{\HSe^{qN/l_j}}\bigg)^{1/n}
\le& \sum_{\Vert\ell\Vert = n} (\Abss{\bX^{(l_1)}}_{\HSe^{qN/l_1}}^{1/l_1})^{l_1/n}\cdots(\Abss{\bX^{(l_j)}}_{\HSe^{qN/l_j}}^{1/l_j})^{l_j/n} \\
\le& \sum_{\Vert\ell\Vert = n} \left(\frac{l_1}{n} \Abss{\bX^{(l_1)}}_{\HSe^{qN/l_1}}^{1/l_1} + \dots + \frac{l_1}{n}\Abss{\bX^{(l_j)}}_{\HSe^{qN/l_j}}^{1/l_j} \right) \\
\le& c_n \sum^{n}_{i=1} \Abss{\bX^{(i)}}_{\HSe^{qN/i}}^{1/i},
\end{align*}
where $c_n>0$ is a constant depending only on $n$ and the second inequality follows from Young's inequality for products.
Hence by the above estimate and the assumption we have
\begin{align*}
\homnqNs{\bY} = \sum_{n=1}^{N} \Abss{\bY^{(n)}}_{\HSe^{qN/n}}^{1/n} \le c^{1/n} c_n \sum_{n=1}^{N} \sum_{i=1}^{n} \Abss{\bX^{(i)}}_{\HSe^{qN/i}}^{1/i} \le C \homnqNs{\bX},
\end{align*}
where $C>0$ is a constant depending only on $c$ and $N$.
\end{proof}
\subsection{Proof of \texorpdfstring{\Cref{thm:bdg_signature}}{Theorem 3.1}}\label{sec:proof_bdg_signature}
\begin{proof}
Denote by $\bS = (\Sig(\bX)_{0,t})_{0\le t \le T}$ the signature process.
We will first proof the upper inequality, i.e. that there exists a constant $C>0$ depending only on $d$, $N$ and $q$ such that
\begin{align}\label{eq:proof_sigest_up}
\homnqNs{\bS} \le C\homnqNs{\bX}.
\end{align}

According to \Cref{lem:homn_by_product} it is sufficient to show that for all $n \in\{1, \dots, N\}$ it holds
\begin{align}\label{eq:sig_est_induction}
c_n \Abss{\bS^{(n)}}_{\HSe^{qN/n}} \le \sum_{\Vert\ell\Vert = n} \Abss{\bX^{(l_1)}}_{\HSe^{qN/l_1}}\cdots\Abss{\bX^{(l_j)}}_{\HSe^{qN/l_j}} \eqqcolon \rho_\bX^{n}
\end{align}
where $c_n>0$ is a constant (depending only on $q$, $d$ and $n$).
Note that it holds
\begin{align}\label{eq:rho_multiplicativity_est}
\rho_\bX^{n} \le \sum_{\Vert\ell\Vert = n} \rho_\bX^{l_j}\Abss{\bX^{(l_{j-1})}}_{\HSe^{qN/l_{j-1}}} \cdots \Abss{\bX^{(l_1)}}_{\HSe^{qN/l_1}} \le c^{\prime}\rho_\bX^{n}
\end{align}
where $c^{\prime}$ is a constant depending only on $n$.
We are going to proof \eqref{eq:sig_est_induction} inductively.

For $n=1$ we have $\bS^{(1)} = \bX^{(1)} - \bX^{(1)}_0 = \bX^{(1)} \in \HSe^{qN}$ and therefore the estimate follows immediately.
Now, assume that \eqref{eq:sig_est_induction} holds for all tensor levels up to some level $n-1$
with $n \in \{2, \dotsc, N\}$.
We will denote by $c^{\prime}, c^{\prime\prime} > 0$ constants that only depend on $n$, $d$ and $q$.
Then we have from \eqref{eq:marcus_signature_ode3}
\begin{equation*}
\bS^{(n)}_t
=\begin{multlined}[t] \sum_{\Abs{\ell} = n, \abs{\ell}\le 2}\int_0^t \bS^{(l_2)}_{u-} \dd \bX^{(l_1)}_u + \sum_{\Abs{\ell} = n,\, 2 \le \abs{\ell}\le 3}\frac12\int_0^t \bS^{(l_3)}_{u-} \dd\CVsmall{\bX^{(l_1)c}}{\bX^{(l_2)c}}_u \\
+ \sum_{\Abs{\ell} = n,\, \abs{\ell}\ge 2}\sum_{0 \le u \le t}\bS^{(l_j)}_{u-}\frac{\Delta
\bX^{(l_{j-1})}_u\cdots \Delta \bX^{(l_1)}_u}{(j-1)!}.\end{multlined}
\end{equation*}
For the first term in the above right-hand side we have by Emery's inequality (\Cref{lem:emerys}) the following estimate
\begin{align*}
\Abs{\sum_{\Abs{\ell} = n, \abs{\ell}\le 2} \int_0^\cdot \bS^{(l_2)}_{u-} \dd \bX^{(l_1)}_u}_{\HSe^{qN/n}}
&\le c^{\prime}\sum_{\Abs{\ell} = n, \abs{\ell}\le 2} \Abss{\bS^{(l_2)}}_{\Se^{qN/l_2}} \Abss{\bX^{(l_1)}}_{\HSe^{qN/l_2}}
\le c^{\prime\prime}\rho_\bX^{n}
\end{align*}
where the last inequality follows from the induction claim and \eqref{eq:rho_multiplicativity_est}.
Further, from the Kunita-Watanabe inequality (\Cref{lem:cauchy_schwartz_tensor_bracket}) and the generalized H\"older inequality, it follows that for all $l_1, l_2\in\None$ with $l_1 + l_2 \le n$ we have
\begin{align*}
\Abs{ \CVsmall{\bX^{(l_2)c}}{ \bX^{(l_1)c}}}_{\Fv^{qN/(l_1+l_2)}}
\le&c^{\prime} \Abss{\bX^{(l_2)}}_{\HSe^{qN/l_2}}\Abss{\bX^{(l_1)}}_{\HSe^{qN/l_1}}
\end{align*}
Then we have again by Emery's inequality, the induction base and \eqref{eq:rho_multiplicativity_est} that it holds
\begin{align*}
\Abs{\sum_{\Abs{\ell} = n,\, 2 \le \abs{\ell}\le 3}\int_0^t \bS^{(l_3)}_{u-} \dd\CVsmall{\bX^{(l_2)c}}{\bX^{(l_3)c}}_u}_{\HSe^{qN/n}}
\le c^{\prime}\rho_{\bX}^{n}.
\end{align*}
Finally we have for the summation term
\begin{align*}
\MoveEqLeft \Abs{\sum_{\Abs{\ell} = n,\, \abs{\ell}\ge 2}\sum_{0 \le u \le t}\bS^{(l_j)}_{u-}\frac{\Delta \bX^{(l_2)}_u\cdots \Delta \bX^{(l_k)}_u}{(j-1)!}}_{\HSe^{qN/n}} \\
&\le \sum_{\Abs{\ell} = n,\, \abs{\ell}\ge 2}\Abs{\sum_{0 \le u \le t}\abs{\bS^{(l_k)}_{u-}}\abs{\Delta \bX^{(l_{k-1})}_u}\cdots\abs{\Delta \bX^{(l_{3})}_u} \abs{\Delta \bX^{(l_1)}_u\Delta \bX^{(l_2)}_u}}_{\Lcal^{qN/|w|}} \\
&\le c^{\prime} \sum_{\Abs{\ell} = n,\, \abs{\ell}\ge 2} \Abss{\bS^{(l_k)} }_{\Se_\infty^{qN/l_1}}
\Abss{\bX^{(l_{k-1})} }_{\Se_\infty^{qN/l_{k-1}}}\cdots
\Abss{\bX^{(l_{3})} }_{\Se_\infty^{qN/l_{3}}} \Abs{\sum_{0 \le u \le t} \abs{\Delta \bX^{(l_1)}_u\Delta \bX^{(l_2)}_u}}_{\Lcal^{qN/(l_1+l_2)}} \\
&\le c^{\prime\prime} \rho_\bX^{n}
\end{align*}
with the last inequality follows again by the Kunita-Watanabe inequality, the induction basis and \eqref{eq:rho_multiplicativity_est}.
Thus we have shown that \eqref{eq:sig_est_induction} holds for all $n\in\{1, \dots, N\}$.

\newcommand{\olbX}{\bar{\bX}}
Now we will proof the lower inequality, i.e. that there exists a constant $c>0$ depending only on $d$, $N$ and $q$ such that
\begin{align}\label{eq:proof_sigest_low}
c \homnqNs{\bX} \le \homnqNs{\bS}.
\end{align}
Therefore define $\olbX^{n} \coloneq (0, \bX^{(1)}, \dots, \bX^{(n)}, 0, \dots, 0)\in\HSehom^{q,N}$ and note that it holds
\begin{align*}
\homnqNs{\olbX^{1}} = \Abss{\bX^{(1)}}_{\HSe^{qN}} = \Abss{\bS^{(1)}}_{\HSe^{qN}} \le \homnqN{\bS}.
\end{align*}
Now assume that it holds $$\homnqNs{\olbX^{n-1}} \le c^{\prime} \homnqN{\bS}$$ for some $n \in \{1,
\dotsc,N\}$.
It follows from the definition of the signature that 
\begin{align}\label{eq:cut_off_signature}
\bS^{(n)}_t = \bX^{(n)}_{0,t} + \Sig(\olbX^{n-1})^{(n)}_{0,t}
\end{align}
and further we have from the upper bound \eqref{eq:proof_sigest_up}, which was already proven above, that
\begin{align}\label{eq:cut_off_signature_estimate}
\homnqNs{\Sig(\olbX^{n-1})_{0,\cdot}} \le C \homnqNs{\olbX^{n-1}}  \le C c^{\prime} \homnqNs{\bS}.
\end{align}
Then we have
\begin{align*}
\homnqNs{\olbX^{n}} &= \homnqNs{\olbX^{n-1}} +  \Abss{\bX^{(n)}}_{\HSe^{qN/n}}^{1/n} \\
&\le c^{\prime}\homnqNs{\bS} + \Abss{\bS^{(n)}}_{\HSe^{qN/n}}^{1/n} + \Abss{\Sig(\olbX^{n-1})^{(n)}_{0,\cdot}}_{\HSe^{qN/n}}^{1/n} \\
&\le c^{\prime}\homnqNs{\bS} + \homnqN{\bS}+   \homnqNs{\Sig(\olbX^{n-1})_{0,\cdot}} \\
&\le c^{\prime \prime}\homnqNs{\bS},
\end{align*}
where we have used \eqref{eq:cut_off_signature} in the second line and \eqref{eq:cut_off_signature_estimate} in the last line.
Therefore, noting that $\olbX^{N}=\bX$, the inequality \eqref{eq:proof_sigest_low} follows by induction.
\end{proof}

\subsection{Proof of \texorpdfstring{\Cref{thm:main_with_jumps}}{Theorem 4.1}}
\label{sec:proof_main_theorem}
We prepare the proof of \Cref{thm:main_with_jumps} with a few more lemmas.
\begin{lemma}\label{lem:exponential_derivative}
Let $N \in \None$ then we have the following directional derivatives of the truncated
exponential map $\exp_N\colon \tf^{N} \to \TT_1^{N}$
\begin{align*}
(\partial_w \exp_N)(\mathbf{x}) &=G(\ad{\mathbf{x}})(e_w)\expN{\mathbf{x}} = \expN{\mathbf{x}} G(-\ad{\mathbf{x}})(e_w), \quad \mathbf{x}\in \tf^{N}, \\
(\partial_w\partial_{w'} \exp_N)(\mathbf{x}) &= \widetilde{Q}(\ad{\mathbf{x}})(e_w\otimes e_{w'})\exp_N(\mathbf{x}), \quad \mathbf{x}\in\tf^{N},
\end{align*}
for all words $w, w' \in \W_d$ with $1\le |w|, |w'| \le N$, where $G$ is defined in \eqref{eq:GHQ_def} and for
\begin{align*}
    \widetilde Q(\ad{\mathbf{x}})(a\otimes b) &=\; G(\ad{\mathbf{x}})(b) G(\ad{\mathbf{x}})(a) +
    \int_0^1 \tau [G (\tau \ad{\mathbf{x}})(b), e^{\tau \ad{\mathbf{x}}}(a)] \,\dd \tau\\
    &=\; \sum_{n,m = 0}^{N} \frac{(\ad{\mathbf{x}})^n(b)}{(n + 1)
    !} \frac{(\ad{\mathbf{x}})^m(a)}{(m + 1) !} + \sum_{n,m = 0}^{N}
    \frac{[(\ad{\mathbf{x}})^n(b), (\ad{\mathbf{x}})^m(a)]_{}}{(n + m + 2) (n + 1)!\,m!}, \quad \mathbf{x},a,b \in \tf^{N}.
  \end{align*}
\end{lemma}
\begin{proof}
For all $w\in\W_d$ with $0 \le |w| \le N$ and $\bx\in\tf^{N}$ the expression $\exp_N(\bx)^{w}$ is a polynomial in the tensor components $(\bx^{v})_{1 \le |v| \le |w|}$.
Therefore the map $\exp_N: \TT^{N}_1 \to \tf^{N}$ is smooth and in particular the first and second order partial derivatives exist in all directions.
For a proof of the explicit form of the first order partial derivatives we refer to \cite[Theorem 7.23]{friz2010multidimensional}.
For the second order derivatives we follow the proof of \cite[Lemma A.1]{kamm2020stochastic}.
Therefore let $\mathbf{x}\in\tf^{N}$ and $w, \wpr$ arbitrary with $1\le |w|, |\wpr| \le N$.
Then we have by the definition of the partial derivatives in $\tf^{N}$ and the product rule
\begin{align*}
\partial_w(\partial_\wpr\exp_N(\mathbf{x})) &= \frac{\mathrm d}{\mathrm dt}\Big(G(\ad{\mathbf{x} + t e_w})(e_{\wpr})\exp_N(\mathbf{x} +  t e_w)\Big) \Big\vert_{t=0} \\
&= \frac{\mathrm d}{\mathrm dt}G(\ad{\mathbf{x} + t e_w})(e_{\wpr})\Big\vert_{t=0} \exp_N(\mathbf{x}) + G(\ad{\mathbf{x}})(e_{\wpr})G(\ad{\mathbf{x}})(e_{w})\exp_N(\mathbf{x}).
\end{align*}
From \cite[Lemma 7.22]{friz2010multidimensional} it holds $\exp_N(\ad{\mathbf{x}})({y}) = \exp_N(\mathbf{x}) y \exp_N(-\mathbf{x})$ for all $\mathbf{x},y\in\tf^{N}$ and it follows further by representing $G$ in integral form that
\begin{align*}
\frac{\mathrm d}{\mathrm dt}G(\ad{\mathbf{x} + t e_w})(e_{\wpr})\Big\vert_{t=0} &= \frac{\mathrm
d}{\mathrm dt}\bigg(\int_0^{1}\exp_N(\tau \ad{\mathbf{x}+t e_w})(e_\wpr)\,\dd\tau\bigg)\bigg\vert_{t=0} \\
&=\int_0^{1}\frac{\mathrm d}{\mathrm dt}\Big(\exp_N(\tau(\mathbf{x}+t
e_w))e_\wpr\exp_N(-\tau(\mathbf{x}+t e_w))\Big)\Big\vert_{t=0}\,\dd\tau\\
&=\begin{multlined}[t]\int_0^{1} \tau G(\ad{\tau \mathbf{x}})(e_w) \exp_N(\tau \mathbf{x}) e_\wpr
\exp_N(-\tau \mathbf{x})\,\dd\tau \\
- \int_0^{1} \tau \exp_N(\tau \mathbf{x}) e_\wpr \exp_N(-\tau \mathbf{x}) G(\ad{\tau
\mathbf{x}})(e_w)\,\dd\tau \end{multlined}\\
&= \int_0^{1} \tau \Lie{G(\tau\ad{\mathbf{x}})(e_w)}{\exp_N(\tau\ad{\mathbf{x}})(e_\wpr)} \,\dd\tau.
\end{align*}
Then the proof is finished after noting that
\begin{align*}
    \int_0^{1} \tau G(\tau\ad{\mathbf{x}})(e_w)\exp(\tau\ad{\mathbf{x}})(e_\wpr)\,\dd\tau
&= \int_0^{t}\tau\sum_{n,m=0}^{N}\frac{(\tau \ad{\mathbf{x}})^{n}}{(n+1)!}\frac{(\tau
\ad{\mathbf{x}})^{m}}{m!}\,\dd\tau\\
&=
\int_0^{t}\sum_{n,m=0}^{N}\frac{(\ad{\mathbf{x}})^{n}}{(n+1)!}\frac{(\ad{\mathbf{x}})^{m}}{m!}\tau^{1+m+n}\,\dd\tau \\
&= \sum_{n,m=0}^{N}\frac{(\ad{\mathbf{x}})^{n}(\ad{\mathbf{x}})^{m}}{(n+1)!m!(n+m+2)}.
\end{align*}
\end{proof}
Note that the operator $Q$ defined in \eqref{eq:GHQ_def} differs from the operator $\tilde{Q}$ defined above. However we have the following
\begin{lemma}\label{lem:symmetric_Q}
Let $N\in\None$ and $\bx\in\tf^{N}$, then it holds
\[\widetilde Q(\ad{\bx})(\mathbb{A}) = Q(\ad{\bx})(\mathbb{A}),\]
for all $\mathbb{A}\in\tf^{N}\otimes\tf^{N}$ with symmetric coefficients $\mathbb{A}^{w_1, w_2} = \mathbb{A}^{w_2, w_1}$ for all $w_1, w_2 \in \W_d$.
\end{lemma}
\begin{proof}
Let $N\in\None$ and $\bx\in\tf^N$ be arbitrary.
Then from the bilinearity of $\tilde{Q}(\ad{\bx})$ and the symmetry of $\mathbb{A}$ we have, with
summation over all words $w_1, w_2$ with $1 \le |w_1|,|w_2|\le N$,
\begin{align*}
  \widetilde Q(\ad{\bx})(\mathbb{A}) &=\sum_{w_1, w_2} \mathbb{A}^{w_1,
    w_2}\sum_{n,m = 0}^{N}\bigg( \frac{(\ad{\bx})^n(e_{w_2})}{(n + 1)!}
    \frac{(\ad{\bx})^m(e_{w_1})}{(m + 1)!}
+ \frac{[(\ad{\bx})^n(e_{w_2}), (\ad{\bx})^m(e_{w_1})]}{(n + m + 2) (n + 1)!\,m!} \bigg)\\
&=\begin{multlined}[t]\sum_{w_1, w_2} \mathbb{A}^{w_1, w_2}\bigg( \sum_{n,m = 0}^{N}
  \frac{(\ad{\bx})^n(e_{w_2})}{(n + 1)!} \frac{(\ad{\bx})^m(e_{w_1})}{(m + 1) !} \\
+ \frac{(\ad{\bx})^n(e_{w_2})(\ad{\bx})^m(e_{w_1}) - (\ad{\bx})^m(e_{w_1})(\ad{\bx})^n(e_{w_2})}{(n
+ m + 2) (n + 1)!\,m!} \bigg) \end{multlined}\\
&=\begin{multlined}[t]\sum_{w_1, w_2} \mathbb{A}^{w_1, w_2}\bigg( \sum_{n,m = 0}^{N}
  \frac{(\ad{\bx})^n(e_{w_1})}{(n + 1)!} \frac{(\ad{\bx})^m(e_{w_2})}{(m + 1) !} \\
+ \frac{(\ad{\bx})^n(e_{w_1})(\ad{\bx})^m(e_{w_2})}{(n + m + 2) (n + 1)!\,m!}
-\frac{(\ad{\bx})^n(e_{w_1})(\ad{\bx})^m(e_{w_2})}{(m + n + 2) (m + 1)!\,n!} \bigg)\end{multlined} \\
&= \sum_{w_1, w_2} \mathbb{A}^{w_1, w_2} \sum_{n,m = 0}^{N}  (2m+2)\frac{(\ad{\bx})^n(e_{w_1}) (\ad{\bx})^m(e_{w_2})}{(n + 1) ! (m+1) ! (n + m + 2)}(\mathbb{A}) = Q(\ad{\bx})(\mathbb{A}).
\end{align*}
\end{proof}
The following two applications of Itô's formula in the non-commutative setting will be a key ingredient in the proof of \Cref{thm:main_with_jumps}.
\begin{lemma}[Itô's product rule]\label{lem:ito_product_rule} Let $\bX, \bY \in \Se(\TT_1^N)$ for some $N\in\None$, then it holds
\begin{align*}
    \bX_t \bY_t -  \bX_0 \bY_0 = \int_{(0,t]}(\dd\bX_u)\bY_u + \int_{(0,t]}\bX_u(\dd\bY_u) + \mul(\outerbracket{\bX}{\bY}_{0,T}), \quad 0 \le t \le T.
\end{align*}
\end{lemma}
\begin{proof}
The statement is an immediate consequence of the one-dimensional Itô's product rule for \cadlag
semimartingales (e.g. \cite[Ch. II, Corollary 2]{protter2005stochastic }) and the definition of the outer
bracket and the multiplication map in \Cref{sec:tensor_semimartingales}.
\end{proof}

\begin{lemma}
  \label{lem:ito-exp-rule} Let $\bX \in\Se(\tf^{N})$ for some $N\in\None$, then it holds
\begin{equation*}
  \exp_N({\bX_t}) - \exp_N({\bX_0}) = \begin{multlined}[t]\int_{(0,t]} G(\ad{\bX_{u-}})(\dd\bX_u)\exp_N({\bX_{u-}}) + \int_0^{t} Q(\ad{\bX_{u-}})(\dd\outerbracket{\bX^{c}}{\bX^{c}}_u)\expN{\bX_{u-}} \\
+\sum_{0<u\le t}\Big(\exp_N({\bX_u}) - \exp_N({\bX_{u-}}) - G(\ad{\bX_{u-}})(\Delta
\bX_u)\exp_N({\bX_{u-}})\Big),\end{multlined}
\end{equation*}
for all $0 \le t \le T$.
\end{lemma}
\begin{proof}
As discussed in the proof of \Cref{lem:exponential_derivative}, it is clear that the map $\exp_N\colon \tf^{N} \to \TT_1^{N}$ is smooth.
Further $\tf^{N}$ is isomorphic to $\R^{D}$ with $D = d + \dotsb + d^{N}$ and we can apply the multidimensional Itô's formula for càdlàg semimartingales (e.g. \cite[Ch. II, Theorem 33]{protter2005stochastic}) to obtain
\begin{equation*}
  \expN{\bX_t} - \expN{\bX_0} =\begin{multlined}[t] \sum_{1 \le |w|\le N}\int_{(0,t]} (\partial_w
        \exp_N)(\bX_{u-})\,\dd \bX^{w}_u \\
        + \frac{1}{2}\sum_{1 \le |w_1|,|w_2|\le
      N}\int_0^{t}(\partial_{w_1}\partial_{w_2}\exp_N)(\bX_{u-})\,\dd\innerbracketsmall{\bX^{w_1c}}{\bX^{w_2c}}_u
    \\
    + \sum_{0<u\le t}\bigg(\exp_N(\bX_u) - \exp_N(\bX_{u-}) - \sum_{1 \le |w|\le N} (\partial_w
  \exp_N)(\bX_{u-})(\Delta \bX^{w}_u) \bigg) \end{multlined}
\end{equation*}
for all $0 \le t \le T$.
From \Cref{lem:exponential_derivative} we then have for the first integral term
\begin{align*}
    \sum_{1 \le |w|\le N}\int_{(0,t]} (\partial_w \exp_N)(\bX_{u-})\,\dd \bX^{w}_u
    &= \sum_{1 \le |w|\le N}\int_{(0,t]} G(\ad{\bX_{u-}})(e_w)\exp_N(\bX_{u-})\,\dd \bX^{w}_u \\
&= \int_{(0,t]} G(\ad{\bX_{u-}})(\dd\bX_u)\exp_N(\bX_{u-}),
\end{align*}
and analogously
\begin{align*}
\sum_{1 \le |w|\le N} (\partial_w \exp_N)(\bX_{u-})(\Delta \bX^{w}_u)
&= \sum_{1 \le |w|\le N} G(\ad{\bX_{u-}})(\Delta\bX_u)\exp_N(\bX_{u-}).
\end{align*}
Moreover, from \Cref{lem:ito-exp-rule} and the definition of the outer bracket in \Cref{sec:tensor_semimartingales}
\begin{align*}
\sum_{1 \le |w_1|,|w_2|\le N}\int_{(0,t]}(\partial_{w_1}\partial_{w_2}\exp_N)(\bX_{u-})
&\dd\innerbracketsmall{\bX^{w_1c}}{\bX^{w_2c}}_u = \int_0^{t}\tilde{Q}(\ad{\bX_u-})(\dd\outerbracket{\bX^{c}}{\bX^{c}}_u)\exp_N(\bX_{u-}).
\end{align*}
Finally, the outer bracket $\outerbracket{\bX^{c}}{\bX^{c}}_t \in \tf^{N}\otimes\tf^{N}$ is symmetric in
the sense of \Cref{lem:symmetric_Q} and therefore we can replace $\tilde{Q}$ with $Q$ in the above identity.
\end{proof}

\begin{lemma}\label{lem:estimates_g_q}
Let $\bX \in \Se(\tf)$ and let $\bA\in\Fv(\tf)$.
For all $k \in \None$ and $\ell = (l_1, \dotsc, l_k)\in (\None)^k$ it holds
\begin{align*}
\abs{\int_0^t \left(\ad{\bX_u^{(l_2)}}\cdots\ad{\bX_u^{(l_k)}}\right)\left(\dd \bA^{(l_1)}_u \right)} \le 2^{k-1} \int_0^t\abs{\bX_u^{(l_2)}\cdots\bX_u^{(l_k)}}\abs{\dd \bA^{(l_1)}_u},
\end{align*}
for all $0 \le t \le T$.
Furthermore, let $(\mathbb{A}_t)_{0 \le t \le T}$ be a process taking values in $\tf\otimes\tf$ such that $\mathbb{A}^{w_1,w_2}\in\Fv$ for all $w_1, w_2 \in \W_d$.
Then it holds for all $0 \le t \le T$
\begin{align*}
\abs{\int_0^t \left(\ad{\bX_u^{(l_3)}}\cdots\ad{\bX_u^{(l_m)}}\odot \ad{\bX_u^{(l_{m+1})}}\cdots\ad{\bX_u^{(l_k)}}\right)\left(\dd\mathbb{A}^{(l_1, l_2)}_u \right)}
\le 2^{k-2} \int_0^t\abs{\bX_u^{(l_3)}\cdots\bX_u^{(l_k)}}\abs{\dd m\left(\mathbb{A}^{(l_1, l_2)}\right)}.
\end{align*}
\end{lemma}
\begin{proof}
Recall from \eqref{eq:integral_adjoint_power} that we expand iterated adjoined operations into a sum of left- and right tensor multiplications and apply Lemma \ref{lem:one-variation_tensor_integrals}.
Note again that for homogeneous tensors $\bx$ and $\by$ it holds $\abs{\bx\by} = \abs{\by\bx}$.
Therefore the statement follows by counting the terms in the expansion.
\end{proof}

\begin{lemma}\label{lem:exp_taylor_estimate}
Let $N\in\None$, $\Delta \bx,\by,\Delta \by \in \tf^{N}$ and define the function
\begin{align*}
    f\colon[0,1]\times[0,1] \to \TT^N_1, \quad (s,t) \mapsto f(s,t) = \expN{s \Delta \bx}\expN{\by + t\Delta \by}\expN{-\by}.
\end{align*}
Then $f(0,0) = 1$ and the first order partial derivatives of $f$ at $(s,t)=(0,0)$ are given by
\begin{align*}
    (\partial_s f)\vert_{(s,t)=(0,0)}
    = \Delta \bx, \quad
    (\partial_t f)\vert_{(s,t)=(0,0)}
    = G(\ad{\by})(\Delta \by).
\end{align*}
Further the following explicit bound for the second order partial derivatives holds
\begin{align*}
    \sup_{0 \le s, t \le 1}\abs{(\nabla^2f^{(n)})\vert_{(s,t)}} \le c_n \sum_{\Vert\ell\Vert = n, \; |\ell|\ge 2} \left(\abs{\Delta \bx^{(l_1)}} + \abs{\Delta \by^{(l_1)}}\right)\left(\abs{\Delta \bx^{(l_2)}} + \abs{\Delta \by^{(l_2)}}\right)z_{l_3}\cdots z_{l_3},
\end{align*}
for all $n \in \{2, \dotsc, N\}$, where $c_n > 0$ is a constant depending only on $n$, $\ell=(l_1,\dotsc,l_k)\in(\None)^{k}$ with $|\ell| = k$ and $z_l \coloneq  \max\{|\Delta \bx^{(l)}|, |\by^{(l)}|, |(\by+\Delta \by)^{(l)}|\}$ for all $l\in\{1, \dotsc, N-2\}$.
\end{lemma}
\begin{proof}
The tensor components of $f(s,t)$ are polynomial in $s$ and $t$ and it follows that $f$ is smooth.
From \Cref{lem:exponential_derivative} we have that the first order partial derivatives of $f$ are given by
\begin{align*}
    (\partial_s f)\vert_{(s,t)}
    &= G(\ad{s\Delta \bx})(\Delta \bx)\expN{s \Delta \bx}\expN{\by + t\Delta \by}\expN{-\by}, \\
    (\partial_t f)\vert_{(s,t)}
    &= \expN{s \Delta \bx}G(\ad{\by + t\Delta \by})(\Delta \by)\expN{\by + t\Delta \by}\expN{-\by}.
\end{align*}
Evaluating at $s=t=0$ we obtain the first result.
Now let $n \in \{2, \dotsc, N\}$. Then it follows from \Cref{lem:exponential_derivative,lem:estimates_g_q} that we can bound the second order derivatives as follows
\begin{align*}
  \MoveEqLeft\sup_{0 \le s,t \le 1}(\partial_{ss} f^{(n)})\vert_{(s,t)} \\
  &\le \sup_{0 \le s,t \le 1} \abs{\pi_{(n)}\big(Q(\ad{s \Delta \bx})((\Delta \bx)^{\otimes 2})\expN{s\Delta \bx}\expN{\by + t \Delta \by}\expN{-\by}\big)} \\
  &\le c_n^{\prime}\sum_{\Vert\ell\Vert = n,\; |\ell|\ge 2} |\Delta \bx^{(l_1)}||\Delta \bx^{(l_2)}|z_{l_3} \cdots z_{l_k},
\end{align*}
\begin{align*}
  \MoveEqLeft\sup_{0 \le s,t \le 1}(\partial_{tt} f^{(n)})\vert_{(s,t)} \\
  &\le \sup_{0 \le s,t \le 1} \abs{\pi_{(n)}\big(\expN{s\Delta \bx}Q(\ad{\by + t \Delta \by})((\Delta \by)^{\otimes 2})\expN{\by + t \Delta \by}\expN{-\by}\big)} \\
  &\le c_n^{\prime\prime}\sum_{\Vert\ell\Vert = n,\; |\ell|\ge 2} |\Delta \by^{(l_1)}||\Delta \by^{(l_2)}|z_{l_3} \cdots z_{l_k},
\end{align*}
and
\begin{align*}
  \MoveEqLeft\sup_{0 \le s,t \le 1}(\partial_{st} f^{(n)})\vert_{(s,t)} \\
  &\le \sup_{0 \le s,t \le 1} \abs{\pi_{(n)}\big(G(\ad{s \Delta \bx})(\Delta \bx)\expN{s\Delta \bx}G(\ad{\by+t\Delta \by})\expN{\by + t \Delta \by}\expN{-\by}\big)} \\
  &\le c_n^{\prime\prime\prime}\sum_{\Vert\ell\Vert = n,\; |\ell|\ge 2} |\Delta \bx^{(l_1)}||\Delta \by^{(l_2)}|z_{l_3} \cdots z_{l_k},
\end{align*}
where $c_n^{\prime}, c_n^{\prime\prime}, c_n^{\prime\prime\prime}>0$ are constants depending only on $n$ and the second statement of the lemma follows.
\end{proof}

\begin{lemma}\label{lem:g_inversion} For all $N \in \None$ and all $\bx \in \tf^{N}$ it holds
\begin{equation*}
H(\ad \bx) \circ G(\ad \bx) = \Id.
\end{equation*}
where $G$ and $H$ are defined in \eqref{eq:GHQ_def}.
Hence, the identity also holds for all $x\in\tf$.
\begin{proof}
Recall the exponent generating function of the Bernoulli numbers, for $z$ near $0$,
\begin{align*}
\quad H(z) = \sum_{n=0}^{\infty}\frac{B_k}{k!} z^k = \frac{z}{e^z - 1}, \quad G(z) = \sum_{k=0}^{\infty} \frac{1}{k+1!} z^k = \frac{e^z-1}{z}.
\end{align*}
Therefore $H(z)G(z) \equiv 1$ identically for all $z$ in a neighbourhood of zero.
Repeated differentiation in $z$ then yields the following property of the Bernoulli numbers
\begin{align*}
    \sum_{k=0}^{n}\frac{B_k}{k!}\frac{1}{(n-k+1)!} = 0, \quad n\in\None.
\end{align*}
Hence the statement of the lemma follows by projecting $H(\ad \bx) \circ G(\ad \bx)$ to each tensor level.
\end{proof}
\end{lemma}

We are now ready to give the
\newcommand{\trsgpf}{S}
\newcommand{\trkapf}{\kap}
\newcommand{\trkapfn}{\kap^{n}}
\begin{proof}[Proof of \Cref{thm:main_with_jumps}.]
Since $\pi_{(0,N)}\Sig(\bX) =
\Sig(\bX^{(0,N)})$ for any $\bX\in\Se(\tf)$ and all truncation levels $N\in\None$, it suffices to show that the identities \eqref{eq:master_with_jumps} and \eqref{eq:master_with_jumps_h_form} hold for the signature cumulant of an arbitrary $\bX\in\HSehom^{1,N}$.
Recall from \Cref{thm:bdg_signature} that this implies that $\homn{\Sig(\bX)} < \infty$
and thus the truncated signature cumulant $\trkapf = (\E_t(\Sig(\bX)_{t,T}))_{0 \le t \le T} \in \Se(\tf^N)$ is well defined.
Throughout the proof we will use the symbol "$\lesssim$" to denote an inequality that holds up to a multiplication of the right-hand side by a constant that may depend only on $d$ and $N$.

Recall the definition of the signature in the Marcus sense from \Cref{sec:gensig}.
Projecting \eqref{eq:marcus_signature_ode3} to the truncated tensor algebra, we see that the signature process $\trsgpf = (\Sig(\bX)_{0,t})_{0 \le t \le T} \in \Se(\TT^N_1)$ satisfies the integral equation
\begin{align}\label{eq:trunc_signature_marcus}
\trsgpf_{t}
&= 1 + \int_{(0,t]}  \trsgpf_{u-} \dd\bX_u + \frac{1}{2} \int_0^{t} \trsgpf_{u} \dd\QVsmall{\bX^{c}}_u
+ \sum_{0<u\le t} \trsgpf_{u-}\big(\expN{\Delta \bX_u} - 1 - \Delta \bX_u\big),
\end{align}
for $0 \le t \le T$.
Then by Chen's relation \eqref{eq:chen_identity} we have
\begin{align*}
\E_t (\trsgpf_{T} \exp_N(\trkapf_T))
= \E_t (\Sig(\bX)_{0,T})
= \trsgpf_{t} \E_t(\Sig(\bX)_{t,T})
= \trsgpf_{t} \expN{\trkapf_t}, \quad 0 \le t \le T.
\end{align*}
It then follows from the above identity and the integrability of $\trsgpf_{T}$ that the process
$\trsgpf\expN{\trkapf}$ is a $\TT^{N}_1$-valued martingale in the sense of \Cref{sec:tensor_semimartingales}.
On the other hand, we have by applying Itô's product rule in \Cref{lem:ito_product_rule}
\begin{equation*}
  \trsgpf_{t} \expN{\trkapf_t} - 1 =\begin{multlined}[t]\int_{(0,t]}
  (\dd\trsgpf_{u})\expN{\trkapf_{u-}} + \int_{(0,t]} \trsgpf_{u- } (\dd\expN{\trkapf_u}) +
m\left(\outerbracket{\trsgpf^c}{\expN{\trkapf}^c}_{0,t}\right) \\
+ \sum_{0<u\le t}\Delta\trsgpf_{u}\,\Delta\expN{\trkapf_u} \end{multlined}
\end{equation*}
Further, by applying the Itô's rule for the exponential map from \Cref{lem:ito-exp-rule} to the $\tf^{N}$-valued semimartingale $\trkapf$ and using \eqref{eq:trunc_signature_marcus}, we have the following form of the continuous covariation term
\begin{align*}
m\left(\outerbracket{\trsgpf^{c}}{\expN{\trkapf^{c}}}_{0,t}\right)
&= \mul\left(\outerbracket{\int_{(0, \cdot]} \trsgpf_{u-} \dd \bX^{c}_u}{\int_{(0, \cdot]}  G(\ad{\trkapf_{u-}})(\dd\trkapf_u^{c})\expN{\trkapf_{u-}}}_{0,t}\right)\\
&= \int_{(0, t]}  \trsgpf_{u-} \left(\Id \odot G (\ad{\trkapf_{u-}})\right)\left(\dd\outerbracket{\bX^{c}}{\trkapf^{c}}_u\right) \expN{\trkapf_{u-}}
\end{align*}
and for the jump covariation term
\begin{equation*}
\sum_{0<u\le t}\Delta\trsgpf_{u}\,\Delta\expN{\trkapf_u} = \sum_{0<u\le t}\trsgpf_{u-}\big(\expN{\Delta \bX_u} -1 \big)\big(\expN{\trkapf_{u}}\expN{{-}
\trkapf_{u-}} - 1\big)\expN{\trkapf_{u-}}.
  \end{equation*}
From the above identities and again with \Cref{lem:ito-exp-rule} and \eqref{eq:trunc_signature_marcus} we have
\begin{equation}\label{eq:exponential_sde_proof}
\trsgpf_{t} \expN{\trkapf_t} - 1 = \int_{(0,t]} \trsgpf_{u-} \dd(\bL_u + \trkapf_u)\expN{\trkapf_{u-}}, \quad 0 \le t \le T,
\end{equation}
where $\bL \in \Se(\tf^N)$ is defined by
\begin{equation}\label{def:L}
  \begin{split}
      \bL_t &= \begin{multlined}[t] \bX_{t}
      + \frac{1}{2} \QV{\bX^{c}}_{t}
      + \sum_{0<u\le t} \big(\expN{\Delta \bX_u} - 1 - \Delta \bX_u\big)
      +\int_{(0,t]} (G - \Id)(\ad{\trkapf_{u-}})(\dd\trkapf_u) \\
      + \int_0^{t} \frac{1}{2} Q (\ad{\trkapf_{u-}}) (\dd\outerbracket{\trkapf^{c}}{\trkapf^{c}}_u)
      +\sum_{0 < u \le t}\Big(\expN{\trkapf_{u}}\expN{-\trkapf_{u-}} - 1 - G(\ad{\trkapf_{u-}})(\Delta
      \trkapf_u)\Big)\\
      +\int_{(0,t]}(\Id \odot G (\ad{\trkapf_{u-}}))(\dd\outerbracket{\bX^{c}}{\trkapf^{c}}_u)\\
      +\sum_{0<u\le t}\big(\expN{\Delta\bX_u} -1 \big)\big(\expN{\trkapf_{u}}\expN{-\trkapf_{u-}} - 1\big)
    \end{multlined}\\
    &=  \bX_{t} + \frac{1}{2} \QV{\bX^{c}}_{t} + \int_{(0,t]} (G - \Id)(\ad{\trkapf_{u-}})(\dd\trkapf_u) + \bV_t + \bC_t + \bJ_t,
  \end{split}
\end{equation}
with $\bV, \bC, \bJ \in \Fv(\tf)$ given by
\begin{align*}
\bV_t &= \frac{1}{2}\int_0^{t} Q (\ad{\trkapf_{u-}}) (\dd\outerbracket{\trkapf^{c}}{\trkapf^{c}}_u), \\
\bC_t &= \int_{(0,t]}(\Id \odot G (\ad{\trkapf_{u-}}))(\dd\outerbracket{\bX^{c}}{\trkapf^{c}}_u),\\
\bJ_t &= \sum_{0<u\le t}\big(\expN{\Delta\bX_u}\expN{\trkapf_{u}}\expN{-\trkapf_{u-}} - 1 - \Delta\bX_u - G(\ad{\trkapf_{u-}})(\Delta\trkapf_u)\big).
\end{align*}

Note that we have explicitly separated the identity operator $\Id$ from $G$ in the above definition of $\bL$.

Since left-hand side in \eqref{eq:exponential_sde_proof} is a martingale and since $\trsgpf_t$ and $\exp(\trkapf_t)$ have the multiplicative left- respectively right-inverse $\trsgpf_t^{-1}$ and $\exp(-\trkapf_t)$ respectively for all $0\le t \le T$, it follows that $\bL + \trkapf$ is a $\tf^N$-valued local martingale.
Let $(\tau_k)_{k\ge 1}$ be a sequence of increasing stopping times with $\tau_k \to T$ a.s. for $k \to \infty$, such that the stopped process $(\bL_{t \wedge \tau_k} + \trkapf_{t\wedge\tau_k})_{0\le t\le T}$ is a true martingale.
Using further that $\trkapf_T = 0$ we have
\begin{align}\label{eq:proof_stopped_master}
\trkapf_{t \wedge \tau_k} = \E_{t}\big\{ \bL_{T\wedge\tau_k, t\wedge\tau_k}\big\},\quad 0 \le t \le T, \; k\in\None.
\end{align}
The estimate \eqref{eq:proof_L_estimate} below shows that $\bL$ has sufficient integrability in order to use the dominated convergence theorem to pass to the $k \to \infty$ limit in the above identity \eqref{eq:proof_stopped_master}, which yields precisely the identity \eqref{eq:master_with_jumps} and hence concludes the first part of the proof of \Cref{thm:main_with_jumps}.

\begin{claim}\label{claim:proof_L_estimate}
It holds that
\begin{equation}\label{eq:proof_L_estimate}
\homn{\bL} \lesssim \homn{\bX}.
\end{equation}
\end{claim}
\begin{proof}[Proof of \Cref{claim:proof_L_estimate}]\renewcommand\qedsymbol{\(\blacksquare\)}
According to \Cref{lem:homn_by_product} it suffices to show that for all $n\in\{1, \dotsc, N\}$ it holds
\begin{align*}
\Abss{\bL^{(n)}}_{\HSe^{qN/n}} \lesssim \sum_{\Vert\ell\Vert = n} \Abss{\bX^{(l_1)}}_{\HSe^{qN/l_1}}\cdots\Abss{\bX^{(l_j)}}_{\HSe^{qN/l_j}} \eqcolon \rho^{n}_\bX,
\end{align*}
where the summation above (and in the rest of the proof) is over multi-indices $\ell = (l_1, \dotsc, l_j)\in(\None)^j$, with $|\ell|=j$ and $\Vert\ell\Vert = l_1 + \dotsb + l_j$.
For $\bM\in\Ma_\loc(\tf^{N})$ and $\bA\in\Fv(\tf^{N})$ define
\newcommand{\rhox}{{\rho}_{\bM, \bA}}
\begin{align*}
\rhox^{n} \coloneq \sum_{\Vert\ell\Vert = n} \zeta^{l_1/N}(\bM^{(l_1)}, \bA^{(l_1)}) \cdots
\zeta^{l_j/N}(\bM^{(l_j)}, \bA^{(l_j)}) < \infty, \quad n = 1, \dotsc, N,
\end{align*}
where for any $q\in[1,\infty)$
\begin{align*}
\zeta^{q}(\bM^{(l)}, \bA^{(l)}) \coloneqq \Abs{\abss{\GQVsmall{\bM^{(l)}}_T}^{1/2} + \abss{\bA^{(l)}}_{1-\var;[0,T]}}_{\Lcal^{q}}.
\end{align*}
Note that it holds
\begin{align}\label{eq:rhox_multiplicativ_estimate}
\rhox^{n} \;\le\; \sum_{\Vert\ell\Vert = n}\left(\rhox^{l_1} \cdots \rhox^{l_j}\right) \;\lesssim\; \rhox^{n}.
\end{align}
Furthermore, it follows from the definition of the $\HSe^{q}$-norm that
\begin{align}\label{eq:rho_inf_rho_ma}
\rho_{\bX}^{n} = \inf_{\bX = \bM + \bA} \rhox^{n},
\end{align}
where the infimum is taken over all semimartingale decomposition of $\bX$.

Now fix $\bM\in\Ma_\loc(\tf^{N})$ and $\bA\in\Fv(\tf^{N})$ arbitrarily, such that $\bX = \bM + \bA$
and $\rhox^{n}<\infty$ for all $n\in\{1, \dotsc, N\}$ (such a decomposition always exists since $\bX
\in \HSehom^{1,N}$).
In particular it holds that $\bM$ is a true martingale.
Next we will proof the following
\begin{claim}\label{claim:induction} For all  $n \in\{1, \dots, N\}$ it holds
\begin{align}\label{eq:first_induction_claim}
\Abss{\bL^{(n)}}_{\HSe^{N/n}}\lesssim \rhox^{n},
\end{align}
and further there exits a semimartingale decomposition $\kap^{(n)} = \kap_0^{(n)} + \bm^{(n)} + \ba^{(n)}$, with $\bm^{(n)}\in\Ma((\Rd)^{\otimes n})$ and $\ba^{(n)}\in\Fv((\Rd)^{\otimes n})$ such that $\Abss{\ba^{(n)}}_{\Fv^{{N/n}}} < \rhox^{n}$ and in case $n \le N-1$ it holds
\begin{align}\label{eq:second_induction_claim}
\zeta^{n/N}(\bm^{(n)}, \ba^{(n)}) \lesssim \rhox^{n}.
\end{align}
\end{claim}

\begin{proof}[Proof of \Cref{claim:induction}]\renewcommand\qedsymbol{\(\blacksquare\)}
We are going to proof inductively over $n \in\{1, \dots, N\}$. Let $n=1$ and note that we have $\bL^{(1)} = \bX^{(1)}$, and therefore
\begin{align*}
\Abss{\bL^{(1)}}_{\HSe^{N}} \le \zeta^{N}(\bM^{(1)}, \bA^{(1)}) = \rhox^{1}.
\end{align*}
Using that $\bM^{(1)}$ is a martingale we can identify a semimartingale decomposition of $\kap^{(1)}$ by
\begin{align*}
\bm^{(1)}_t \coloneq \E_t\left(\bA^{(1)}_T\right) - \E\left(\bA^{(1)}_T\right), \quad \ba^{(1)}_t
\coloneq -\bA_t, \quad 0 \le t \le T.
\end{align*}
In case $N\ge2$, we further have from the BDG-inequality and the Doob's maximal inequality that
\begin{align*}
\Abs{\bm^{(1)}}_{\HSe^{N}} \lesssim \Abs{\bm^{(1)}}_{\Se^{N}} \lesssim \Abs{\bm^{(1)}_T}_{\Lcal^{N}} = \Abs{\E\left(\bA^{(1)}_T\right) - \bA^{(1)}_T}_{\Lcal^{N}} \lesssim \Abs{\bA^{(1)}}_{\Fv^{N/n}}  \lesssim \rhox^{1}
\end{align*}
and this shows the second part of the induction claim.

Now assume that $N \ge 2$ and that the induction claim \eqref{eq:first_induction_claim} and
\eqref{eq:second_induction_claim} holds true up level $n-1$ for some $n \in \{2, \dotsc, N\}$.
Note that $\bL^{(n)}$ has the following decomposition
\begin{align}\label{eq:decomposition_L}
\bL^{(n)}=\;& \left\{\bM^{(n)} + \bN^{(n)}\right\} + \left\{\bA^{(n)} + \frac{1}{2} \QV{\bX^{c}}^{(n)} + \bB^{(n)}+ \bV^{(n)} + \bC^{(n)} + \bJ^{(n)}\right\},
\end{align}
where $\bN^{(n)} \in \Ma_{\loc}((\R^{d})^{\otimes n})$ and $\bB^{(n)} \in \Fv((\R^{d})^{\otimes n})$ are defined by
\begin{align*}
 \bN^{(n)} &= \pi_{(n)}\int_{(0, t]} (G-\mathrm{Id})(\ad{\trkapf_{u-}})(\dd \ol\bm_u), \\
 \bB^{(n)} &= \pi_{(n)}\int_{(0, t]}(G-\mathrm{Id})(\ad{\trkapf_{u-}})(\dd \ol\ba_u)
\end{align*}
with $\ol\ba = \pi_{(0,N)}(\ba^{(1)} + \dots + \ba^{(n-1)})\in\Fv(\tf^{N})$ and $\ol\bm = \pi_{(0,N)}(\bm^{(1)} + \dotsb + \bm^{(n-1)})\in\Ma(\tf^{N})$.

From \Cref{lem:cauchy_schwartz_tensor_bracket} and the generalized H\"older inequality we have
\begin{align}\label{eq:proof_est_cont_qv_X}
    \Abs{\QV{\bX^c}^{(n)}}_{\Fv^{N/n}} \le \sum_{i = 1}^n \Abs{\CV{\bM^{(i)c}}{\bM^{(n-i)c}}}_{\Fv^{N/n}} \lesssim \sum_{i = 1}^n \Abs{\bM^{(i)}}_{\HSe^{N/i}}\Abs{\bM^{(n-i)}}_{\HSe^{N/(n-i)}}
    \lesssim \rhox^{n}.
\end{align}
It follows from \eqref{eq:second_induction_claim} and the induction basis that for all $l \in\{1,
\dotsc, n-1\}$ it holds that
\begin{align}\label{eq:proof_gqv_kap}
\Abss{\kap^{(l)}}_{\HSe^{N/l}} = \Abss{\kap^{(l)} - \kap^{(l)}_0}_{\HSe^{N/l}} \lesssim \rhox^{l},
\end{align}
and further that
\begin{align}\label{est:kappa_sup}
\kappa^{l*}_{T} \coloneq \sup_{0\le t \le T} |\kap^{(l)}_t
|, \quad \Abss{\kap^{(l)}}_{\Se^{N/l}} = \Abss{\kappa_T^{l\ast}}_{\Lcal^{N/l}} \le |\kap^{(l)}_0| + \Abss{\kap^{(l)}-\kap^{(l)}_0}_{\Se^{N/l}}
\lesssim \rhox^{l}.
\end{align}
From the definition and linearity of
$Q(\ad\bx)$ ($\bx\in\tf^N$), \Cref{lem:estimates_g_q,lem:cauchy_schwartz_tensor_bracket} we have the following estimate
\begin{align*}
\MoveEqLeft\abs{\bV^{(n)}}_{1-\var;[0,T]} \\
&\lesssim \sum_{\Vert\ell\Vert = n, |\ell| \ge 2}\sum_{m=2}^{j} \abs{\int_0^{\cdot} \left(\ad{\trkapf^{(l_3)}_{u-}}\cdots\ad{\trkapf^{(l_m)}_{u-}} \odot \ad{\trkapf^{(l_{m+1})}_{u-}}\cdots \ad{\trkapf^{(l_j)}_{u-}}\right)\left(\dd \outerbracket{\bm^{(l_1)c}}{\bm^{(l_2)c}}_u\right)}_{1-\var;[0,T]} \nonumber\\
&\lesssim \sum_{\Vert\ell\Vert = n, |\ell| \ge 2} \int_0^t \abs{\kap^{(l_{3})}_{u-}} \cdots \abs{\kap^{(l_{j})}_{u-}} \;\dd \abs{\CV{\bm^{(l_{1})c}}{\bm^{(l_{l_2})c}}_u}\\
&\lesssim \sum_{\Vert\ell\Vert = n, |\ell| \ge 2} \kappa_{T}^{l_3*} \cdots \kappa^{l_j*}_{T}\sqrt{\abs{\GQV{\bm^{(l_{1})}}_T}}\sqrt{\abs{\GQV{\bm^{(l_{2})}}_T}}.
\end{align*}
It then follows from the generalized H\"older inequality
\begin{align}\label{eq:proof_bound_qua}
\Abs{\bV^{(n)}}_{\Fv^{N/n}} &\lesssim \sum_{\Vert\ell\Vert=n, \abs{\ell} \ge 2} \Abs{\kap_{T}^{(l_3)}}_{\Se^{N/l_3}} \cdots \Abs{\kap^{(l_j)}_{T}}_{\Se^{N/l_j}}\Abs{\bm^{(l_1)}}_{\HSe^{N/l_1}}\Abs{\bm^{(l_2)}}_{\HSe^{N/l_2}}\nonumber\\
&\lesssim  \sum_{\Vert\ell\Vert = n, |\ell| \ge 2} \rhox^{l_1} \cdots \rhox^{l_j} \nonumber\\
&\lesssim \rhox^{n},
\end{align}
where the second inequality follows from the induction basis and the estimates \eqref{est:kappa_sup}
and \eqref{eq:second_induction_claim}, noting that $\Vert\ell\Vert=n$ and $\abs{l}\ge 2$ implies
that $l_1, \dotsc, l_j \le n-1$,
and the third inequality follows from \eqref{eq:rhox_multiplicativ_estimate}.
From similar arguments we see that the following two estimate also hold
\begin{align}\label{eq:proof_bound_cov}
\Abs{\bC^{(n)}}_{\Fv^{N/n}} &\lesssim\sum_{\Vert\ell\Vert=n, \abs{\ell} \ge 2} \Abs{\int_0^\cdot \left(\Id \odot \ad{\trkapf_{u-}^{(l_3)}} \cdots \ad{\trkapf_{u-}^{(l_j)}}\right)\left(\dd\outerbracket{\bM^{(l_1)c}}{\bm^{(l_2)c}}_u\right)}_{\Fv^{N/n}}
\nonumber\\
&\lesssim \sum_{\Vert\ell\Vert=n, \abs{\ell} \ge 2} \Abs{\kap_{T}^{(l_3)}}_{\Se^{N/l_3}} \cdots \Abs{\kap^{(l_j)}_{T}}_{\Se^{N/l_j}}\Abs{\bM^{(l_1)}}_{\HSe^{N/l_1}}\Abs{\bm^{(l_2)}}_{\HSe^{N/l_2}}\nonumber\\
&\lesssim \rhox^{n}
\end{align}
and
\begin{align}\label{eq:proof_bound_leb_sti}
\Abs{\bB^{(n)}}_{\Fv^{N/n}}
\lesssim \sum_{\Vert\ell\Vert=n, |\ell|\ge2} \Abs{\kap_{T}^{(l_2)}}_{\Se^{N/l_3}} \cdots \Abs{\kap^{(l_j)}_{T}}_{\Se^{N/l_j}} \Abs{\ba^{(l_{1})}}_{\Fv^{N/l_1}} \lesssim \rhox^{n}.
\end{align}
For the local martingale $\bN^{(n)}$, we use \Cref{lem:estimates_g_q,lem:cauchy_schwartz_tensor_bracket} to estimate its quadratic variation as follows
\begin{align*}
\abs{\GQV{\bN^{(n)}}_T}
&= \abs{\GQV{\sum_{\Vert\ell\Vert=n,\, |\ell| \ge 2} \int_{(0, \cdot]}\frac{1}{k!}\left(\ad{\trkapf_{u-}^{(l_2)}}\cdots\ad{\trkapf_{u-}^{(l_k)}}\right)\left(\dd \bm^{(l_1)}_u\right)}_T}
\nonumber\\
&\lesssim \sum_{\Vert\ell\Vert = 2n, |\ell| \ge 4}\abs{\int_{(0,T]}\left(\ad{\trkapf_{u-}^{(l_2)}}\cdots\ad{\trkapf_{u-}^{(l_m)}}\odot \ad{\trkapf_{u-}^{(l_{m+1})}}\cdots\ad{\trkapf_{u-}^{(l_k)}} \right)\left(\dd\outerbracket{\bm^{(l_1)}}{\bm^{(l_2)}}_u\right)} \nonumber\\
&\lesssim \sum_{\Vert\ell\Vert = 2n, |\ell| \ge 4} \kappa_{T}^{l_3*} \cdots \kappa^{l_j*}_{T}\sqrt{\abs{\GQV{\bm^{(l_{1})}}_T}}\sqrt{\abs{\GQV{\bm^{(l_{2})}}_T}}.
\end{align*}
Then it follows once again by the generalized H\"older inequality and the induction basis that
\begin{align}\label{eq:mag_term_qv}
\Abs{\bN^{(n)}}_{\HSe^{N/n}} \lesssim \rhox^{n}.
\end{align}
Finally let us treat the term $\bJ^{(n)}$.
First define
\begin{align*}
    Z^{l} \coloneq  \sup_{0 < u \le T}\left(\max\left\{ \abs{\Delta \bX^{(l)}_u},  \abs{\trkapf^{(l)}_{u-}}, \abs{\trkapf^{(l)}_{u}} \right\}\right), \quad  l=\{1, \dotsc, n-1\}.
\end{align*}
and from\eqref{est:kappa_sup} it follows that for all $l=\{1, \dotsc, n-1\}$ it holds
\begin{align}\label{eq:proof_est_Z}
    \Abss{Z^{l}}_{\Lcal^{N/l}} &\le 2\Abss{\bX^{(l)}}_{\Se^{N/l}} + \Abss{\kap^{(l)}_T}_{\Se^{N/l}}
   \le 2 \Abss{\bX^{(l)}}_{\HSe^{N/l}} + \Abss{\kap^{(l)}_T}_{\Se^{N/l}}
   \lesssim \rhox^{l, N}.
\end{align}
Then by Taylor's theorem and \Cref{lem:exp_taylor_estimate} we have
\begin{align*}
\abs{\bJ^{(n)}}_{1-\var;[0,T]} &= \sum_{0 < u \le T}\bigg\vert\pi_{(n)}\Big(\expN{\Delta \bX_u}\expN{\trkapf_{u}}\expN{-\trkapf_{u-}} - 1 - \Delta \bX_u
- G(\ad{\trkapf_{u-}})(\Delta \trkapf_u)\Big)\bigg\vert \nonumber\\
&\lesssim \sum_{\Vert\ell\Vert = n, ||\ell| \ge 2}Z^{l_3} \cdots Z^{l_j}\sum_{0 < u \le T}\left(\abs{\Delta \bX_u^{(l_1)}} + \abs{\Delta \trkapf_u^{(l_1)}}\right)\left(\abs{\Delta \bX_u^{(l_2)}} + \abs{\Delta \trkapf_u^{(l_2)}}\right)\nonumber\\
&\lesssim \sum_{\Vert\ell\Vert = n, ||\ell| \ge 2}Z^{l_3} \cdots Z^{l_j}\left(\sqrt{\abs{\GQV{\bX^{(l_1)}}_T}}+\sqrt{\abs{\GQV{\trkapf^{(l_1)}}_T}}\right)\left(\sqrt{\abs{\GQV{\bX^{(l_2)}}_T}}+\sqrt{\abs{\GQV{\trkapf^{(l_2)}}_T}}\right).
\end{align*}
Hence it follows by the generalized H\"older inequality that
\begin{align}\label{eq:proof_taylor_estimate}
\Abs{\bJ^{(n)}}_{\Fv^{N/n}} \lesssim& \sum_{\Vert\ell\Vert = n, ||\ell| \ge 2} \Abs{Z^{l_3}}_{\Lcal^{N/l_1}} \cdots \Abs{Z^{l_j}}_{\Lcal^{N/l_j}}\left(\Abs{\bX^{(l_1)}}_{\HSe^{N/l_1}} + \Abs{\kap^{(l_1)}}_{\HSe^{N/l_1}}\right)\nonumber\\
&\cdot\left(\Abs{\bX^{(l_2)}}_{\HSe^{N/l_2}} + \Abs{\kap^{(l_2)}}_{\HSe^{N/l_2}}\right) \nonumber\\
\lesssim&\; \rhox^{n}
\end{align}
where the last estimate follows from \eqref{eq:proof_est_Z}, \eqref{eq:proof_gqv_kap} and \eqref{eq:rhox_multiplicativ_estimate}.

Summarizing the estimates
\eqref{eq:proof_bound_qua}, \eqref{eq:proof_bound_cov}, \eqref{eq:proof_bound_leb_sti}, \eqref{eq:mag_term_qv} and \eqref{eq:proof_taylor_estimate} we have
\begin{equation}\label{eq:proof_L_big_sum_estimate}
  \begin{split}
    \Abs{\bL^{(n)}}_{\HSe^{N/n}}&\lesssim \begin{multlined}[t]\Abs{\bM^{(n)}}_{\HSe^{N/n}} +
      \Abs{\bN^{(n)}}_{\HSe^{N/n}} + \Abs{\bA^{(n)}}_{\Fv^{N/n}} +
    \Abs{\QV{\bX^{c}}^{(n)}}_{\Fv^{N/n}}\\
      + \Abs{\bB^{(n)}}_{\Fv^{N/n}} + \Abs{\bV^{(n)}}_{\Fv^{N/n}} + \Abs{\bC^{(n)}}_{\Fv^{N/n}} +
    \Abs{\bJ^{(n)}}_{\Fv^{N/n}}\end{multlined}\\
    &\lesssim \rhox^{n},
  \end{split}
\end{equation}
which proofs the first part of the induction claim \eqref{eq:first_induction_claim}.
Then it follows form dominated convergence theorem that projecting \eqref{eq:proof_stopped_master} to the tensor level $n$ and passing to the $k\to\infty$ limit yields
\begin{align*}
\kap_t^{(n)} = \E_t\left(\bL^{(n)}_{T,t}\right), \quad 0 \le t \le T.
\end{align*}
Since $\bM^{(n)}$ and $\bN^{(n)}$ are true martingales (for the latter this follows from \eqref{eq:mag_term_qv}), we are able to identify a decomposition $\kap^{(n)} = \kap^{(n)}_0 + \bm^{(n)} + \ba^{(n)}$ by
\begin{align*}
\ba^{(n)} &= - \left\{ \bA^{(n)} + \frac{1}{2}\QV{\bX^{c}}^{(n)} + \bB^{(n)} + \bV^{(n)} + \bC^{(n)} + \bJ^{(n)}\right\}\\
\bm^{(n)}_t &= \E\left(\ba^{(n)}_T\right) - \E_{t}\left(\ba^{(n)}_T\right), \quad 0 \le t \le T.
\end{align*}
Again from the estimates \eqref{eq:proof_bound_qua}, \eqref{eq:proof_bound_cov}, \eqref{eq:mag_term_qv} and \eqref{eq:proof_taylor_estimate} it follows that
\begin{align*}
\Abss{\ba^{(n)}}_{\Fv^{N/n}} \lesssim \rhox^{n}
\end{align*}
and in case $n \le N - 1$ it follows from the BDG-inequality and Doob's maximal inequality that
\begin{align*}
\Abs{\bm^{(n)}}_{\HSe^{N/n}} \lesssim \Abs{\bm^{(n)}}_{\Se^{N/n}} \lesssim \Abs{\bm^{(n)}_T}_{\Lcal^{N/n}} = \Abs{\E\left(\ba^{(n)}_T\right) - \ba^{(n)}_T}_{\Lcal^{N/n}} \lesssim \Abs{\ba^{(n)}}_{\Fv^{N/n}}  \lesssim \rhox^{n},
\end{align*}
which proofs the second part of the induction claim \eqref{eq:second_induction_claim}.
\end{proof}
The estimate \eqref{eq:proof_L_estimate} immediately follows from \eqref{eq:rho_inf_rho_ma} and \eqref{eq:first_induction_claim}, which finishes the proof of \Cref{claim:induction}.
\end{proof}
Note that since $\QV{\bX^{c}}$, $\bV$, $\bC$ and $\bJ$ are independent of the decomposition $\bX = \bM + \bA$ it follows from taking the infimum over all such decompositions in the inequality \eqref{eq:proof_L_big_sum_estimate} that
\begin{align}\label{eq:proof_qvX-V-c_J_estimate}
\Abs{\QV{\bX^{c}}^{(n)}}_{\Fv^{N/n}} + \Abs{\bB^{(n)}}_{\Fv^{N/n}} + \Abs{\bV^{(n)}}_{\Fv^{N/n}} + \Abs{\bC^{(n)}}_{\Fv^{N/n}} + \Abs{\bJ^{(n)}}_{\Fv^{N/n}} \lesssim \; \rho^{n}_{\bX},
\end{align}
for all $n \in \{1, \dotsc, N\}$.
The same argument applies to $\trkapf$ and the estimate \eqref{eq:proof_gqv_kap} and we obtain
\begin{align}\label{eq:proof_kap_rhox_est}
\Abs{\trkapf^{(n)}}_{\HSe^{N/n}} \lesssim \rho_\bX^{n},
\end{align}
for all $n\in\{1,\dots, N-1\}$.

Next we are going to show that $\trkapf$ satisfies the functional equation \eqref{eq:master_with_jumps_h_form}.
Recall that $\bL + \trkapf \in \Ma_\loc(\tf^{N})$.
From Lemma \ref{lem:g_inversion} we have the following equality
\begin{align*}
\int_{(0,t]}H(\ad{\trkapf_{u-}})\left(\dd(\bL_u + \trkapf_u)\right) = \trkapf_{t} - \trkapf_0  + \wt{\bL}_t
\end{align*}
for all $0\le t \le T$, where
\begin{align}\label{eq:proof_L_tilde}
\wt{\bL}_t =\;& \int_{(0,t]}H(\ad{\trkapf_{u-}})\left\{\dd\bX_u + \frac{1}{2}\dd\QV{\bX^{c}}_{u} + \dd\bV_u + \dd\bC_u + \dd\bJ_u \right\}.
\end{align}
From \Cref{lem:emerys} (Emery's inequality) and the estimates \eqref{eq:proof_qvX-V-c_J_estimate} and \eqref{eq:proof_kap_rhox_est} it follows
\begin{align*}
  \Abss{\wt\bL^{(n)}}_{\HSe^{N/n}} &\lesssim\begin{multlined}[t] \sum_{\Vert\ell\Vert = n}
    \Abs{\kap^{(l_2)}}_{\Se^{N/l_2}} \cdots
    \Abs{\kap^{(l_j)}}_{\Se^{N/l_j}}\Big\{\Abs{\bX^{(l_1)}}_{\HSe^{N/l_1}} +
    \Abs{\QV{\bX^{c}}^{(l_1)}}_{\Fv^{N/l_1}} \\
  + \Abs{\bV^{(l_1)}}_{\Fv^{N/l_1}} + \Abs{\bC^{(l_1)}}_{\Fv^{N/l_1}} +
\Abs{\bJ^{(l_1)}}_{\Fv^{N/l_1}}\Big\}\end{multlined} \\
  &\lesssim \rho_{\bX}^{n},
\end{align*}
Hence by \Cref{lem:homn_by_product} it holds
\begin{align}\label{eq:l_tilde_estimate}
\homns{\wt\bL} \lesssim \homns{\bX}.
\end{align}
Now note that we have already shown in \Cref{claim:induction} that $\trkapf = \trkapf_0 + \bm +
\ba$, where $\bm\in\Ma(\tf^{N})$ and $\ba\in\Fv(\tf^{N})$ which satisfies that
$\Abss{\ba^{(n)}}_{\Fv} <\infty$ for all $n \in \{1, \dotsc, N\}$.
Together with the above estimate it then follows that $\kap + \wt\bL$ is indeed a true martingale and therefore
\begin{align*}
\trkapf_t = \E\left( \wt\bL_{T,t} \right), \quad 0 \le t \le T,
\end{align*}
which is precisely the identity \eqref{eq:master_with_jumps_h_form}.
\end{proof}

\bibliographystyle{arxivalpha}
\bibliography{library}

\end{document}